\documentclass[11pt,a4paper]{article}
\usepackage[T1]{fontenc}
\usepackage[utf8]{inputenc}
\usepackage{lmodern}
\usepackage[margin=1in]{geometry}
\usepackage{amsmath,amssymb,amsthm,mathrsfs}
\usepackage{graphicx,booktabs}
\usepackage[authoryear,round]{natbib}
\usepackage{bibunits}
\usepackage{appendix}
\usepackage[hidelinks]{hyperref}
\makeatletter
\let\arxiv@natlinkstart\hyper@natlinkstart
\def\hyper@natlinkstart#1{\arxiv@natlinkstart{\@bibunitname.#1}}
\let\arxiv@natlinkbreak\hyper@natlinkbreak
\def\hyper@natlinkbreak#1#2{\arxiv@natlinkbreak{#1}{\@bibunitname.#2}}
\let\arxiv@natanchorstart\hyper@natanchorstart
\def\hyper@natanchorstart#1{\arxiv@natanchorstart{\@bibunitname.#1}}
\makeatother
\hypersetup{pdftitle={TWIN: Two window inspection for online change point detection},pdfauthor={Patrick Bastian and Tim Kutta}}
\allowdisplaybreaks[1]
\theoremstyle{definition}
\newtheorem{theo}{Theorem}[section]
\newtheorem{lem}[theo]{Lemma}
\newtheorem{prop}[theo]{Proposition}
\newtheorem{cor}[theo]{Corollary}

\newtheorem{rem}[theo]{Remark}
\newtheorem{ass}[theo]{Assumption}
\newcommand{\cond}{\overset{d}{\to}}
\newcommand{\R}{\mathbb{R}}

\newcommand{\E}{\mathbb{E}}
\newcommand{\N}{\mathbb{N}}
\newcommand{\PR}{\mathbb{P}}

\defaultbibliographystyle{apalike}
\defaultbibliography{tim,twin_simulation_references}
\title{TWIN: Two window inspection for online change point detection}
\author{Patrick Bastian \qquad Tim Kutta\thanks{Corresponding author: \href{mailto:tim.kutta@math.au.dk}{tim.kutta@math.au.dk}.}\\[0.5em]
\normalsize Department of Mathematics, Aarhus University}
\date{}
\begin{document}
\maketitle
\begin{abstract}
We propose a new class of sequential change point tests based on two-window inspection (TWIN). TWIN scans for changes by comparing symmetric subsamples of data and  combining them in a weighted maximum, similar to Lévy’s modulus of continuity. The TWIN detector generally admits a Gaussian approximation, which becomes an above-threshold approximation at Lévy’s critical weight. For this weight, the detector achieves rate-optimal detection delays. We present a version of TWIN for detecting mean changes and establish convergence even for weakly dependent data taking values in a separable Banach space. We also discuss the detection of general distributional changes. Controlling the Lévy modulus of continuity of partial sum processes typically requires subgaussian data. However, we show that this requirement can be circumvented by an adaptive version of TWIN that requires only finite moments of low order. The adaptation is based on a data-dependent clock, that can be motivated by the Dambis--Dubins--Schwarz theorem. Simulations and an application to monitoring COVID-19 data illustrate the performance of the proposed methods.

\end{abstract}

\begin{bibunit}

\section{Introduction}
\textbf{Sequential change point detection} We consider the problem of sequentially detecting a change in the distribution of a time series \((X_n)_{n \in \mathbb{N}}\). For our exposition, we focus on real-valued data and a change in the mean parameter, although our theory also covers mean changes in dependent Banach-valued time series and general distributional changes for real-valued data.
The procedure begins with the analyst obtaining a \textit{training sample} \(\{X_1, \dots, X_N\}\) where the mean is unknown but constant at \(\mu^{(1)} := \mathbb{E}X_i\), \(1 \leq i \leq N\). This sample serves as a benchmark for subsequent change point detection. 
After training, the \textit{monitoring period} starts, during which data arrive sequentially, expanding the sample incrementally to \(\{X_1, \dots, X_{N+1}\}\), \(\{X_1, \dots, X_{N+2}\},\) and so on, theoretically indefinitely. 
The goal is to test sequentially at some (asymptotically) controlled nominal level $\alpha$ the hypotheses
\begin{align} \label{e:H0X}
&H_0: \mu^{(1)} := \mu_1 = \dots = \mu_N = \mu_{N+1} = \mu_{N+2} = \dots,\\ 
\label{e:H1X}
&H_1: \mu^{(1)} := \mu_1 = \dots = \mu_{N+k^\star} \neq \mu_{N+k^\star+1} 
= \mu_{N+k^\star+2} = \dots =: \mu^{(2)},
\end{align}
where $\mu_i=\mathbb E X_i$ and under $H_1$ a change occurs at an unknown time \(N+k^\star\). Asymptotics are studied for \(N \to \infty\), and in our local alternatives we allow \(k^\star = k^\star_N\) and \(\mu^{(i)} = \mu^{(i)}_N\) to depend on \(N\), with 
\begin{align} \label{e:jumpsize}
\Delta_N := |\mu^{(1)} - \mu^{(2)}|
\end{align}
denoting the jump size. The first time in the monitoring period when a change point is declared  is denoted by $\hat k$, and under $H_1$ the "detection delay" signifies the difference
\[
 \tilde \tau:= \max(\hat k-k^\star,0).
\]

\textbf{Criteria for successful monitoring} 
A sequential change point test should ideally meet three criteria: 
First, it should asymptotically control the user–specified nominal level under $H_0$. 
Second, under $H_1$, it should be powerful even against local alternatives where the jump size $\Delta_N$ is small. 
Third, when larger changes occur, it should detect them with short delay. 
While the first requirement has rightly received considerable attention in the literature, comparatively less emphasis has been placed on systematically analyzing and optimizing the latter two \citep{aue:kirch:2024}.
Much of the recent progress in the training-sample framework has focused on extending existing testing procedures to new data structures or problem settings (see also Supplement~\ref{sec:app:lit} for a short overview).
For many commonly used procedures, however, detection delays still grow at a polynomial rate in $N$, and sensitivity to small changes is limited by the size of the training sample. These limitations become particularly pronounced when changes occur late in the monitoring period.
For this reason, we develop a class of new statistics that combine rate-optimal detection delays and near optimal power with an asymptotic approximation of the nominal level.

\textbf{Main contributions.} The dominant approach to testing the hypotheses pair \eqref{e:H0X}--\eqref{e:H1X} is based on cumulative sum (CUSUM) statistics, as we review in Section~\ref{sec:lit}. In this work, we introduce a new variant, the \emph{two-window inspection} (TWIN) CUSUM, which compares two aggregated samples of approximately equal size: one formed from the earliest data and one from the most recent data.
We combine comparisons across window sizes in a weighted maximum, with strong weights motivated by L\'evy's modulus of continuity. A logarithmic exponent $\beta\ge1/2$ controls the strength of the weighting, with $\beta=1/2$ corresponding to L\'evy's critical weight. Our main contributions are as follows.
\begin{enumerate}
    \item \textit{Critical weighting and Gaussian approximation.}
    For $\beta>1/2$, we derive the asymptotic null distribution of TWIN over an unbounded monitoring horizon as a functional of the Brownian motion. At the critical choice $\beta=1/2$, ordinary weak convergence can fail because of fluctuations on very short windows. We establish ``above-threshold convergence'': the Brownian approximation remains valid above a suitable threshold, permitting sequential testing at sufficiently small nominal levels.
    The analysis relies on new bounds for the global modulus of continuity of partial sum processes (Lemmas~\ref{lem:twin:modulus} and ~\ref{lem:finite:moc:empirical}). 
    \item \textit{Optimizing detection delay and power.}
    For changes of fixed size, the TWIN CUSUM achieves delays of order $\mathcal{O}_P(\log^{2\beta}(N+k^\star))$. This is a significant improvements on the polynomial delays of common training-based procedures discussed below. At the critical choice $\beta=1/2$, this guarantee becomes rate optimal.
    Similarly, near optimal power against local alternatives  is attained (except for a log-factor), which is an important improvement for late changes, compared to existing literature. 
    \item \textit{Dependent data in Banach spaces.}
    We establish our convergence results for weakly dependent observations in a separable Banach space, covering real-valued, multivariate and functional time series. Our method is fully feasible, and critical values can be approximated by combining long-run variance estimation with Monte Carlo simulations.
    \item \textit{Time-adaptive TWIN.}
    Like most multiscale methods, the standard TWIN construction assumes subgaussian model errors. However, we can show that this requirement can be relaxed to finite moments of low order through a data-adaptive time change. This new "clock" allocates more time to large fluctuations, which automatically regularizes the partial sum process on short intervals. The adaptive clock construction is motivated by the Dambis--Dubins--Schwarz rescaling of continuous martingales to a Brownian motion.
    \item \textit{Nonparametric extensions.}
    While most of this work is devoted to mean changes, the applicability of TWIN detectors goes beyond that. We apply the two-window construction to empirical distribution functions, which then yields a fully nonparametric test for changes in the marginal distribution. We obtain analogous power and delay bounds, and the limiting distribution under the null is pivotal for independent observations with a continuous cdf.
\end{enumerate}

\textbf{Notations} Results in this work are asymptotic as $N \to \infty$. For two sequences $(a_N)_N$ and $(b_N)_N$ of positive numbers we write $a_N \ll b_N$ if $\lim_{N \to \infty} a_N/b_N =0$. We write $a_N \lesssim b_N$ if $\limsup_{N} a_N/b_N$ is finite and $a_N \asymp b_N$ if both   $a_N \lesssim b_N$ and  $b_N \lesssim a_N$ hold.

\textbf{Structure of this paper} In  the remainder of this section, we provide a brief overview of some related literature. In Section \ref{sec:2} and \ref{ac:sec} we present our new methodology and analyze its theoretical properties. Simulations are given in Section \ref{sec:finite}. Proofs, a data example and further details are given in the Supplement.

\subsection{Related literature} \label{sec:lit}

Sequential change point detection for the mean parameter of a time series is a classical statistical problem, often traced back to the work of \cite{page:1954}, who developed CUSUM statistics for "continuous inspection schemes". Classical formulations focused on a long average run length in the absence of changes and short expected detection delays after a change (see, e.g. \cite{lorden:1971,pollak:1985}). The formulation of \eqref{e:H0X}-\eqref{e:H1X}, using a sequential test description, paired with a training period, was introduced in the seminal work of \cite{chu:stinchcombe:white:1996}. This approach is flexible because it does not require the specification of in-control parameters and methods developed specifically for this setup will be the focus of this paper. Since, even with this restriction, the related literature is vast, the following overview is restricted to some close points of comparison. We provide a few more references on real-time change point detection in Section \ref{sec:app:lit} of the Supplement. 
For a more comprehensive, recent literature review, we refer the reader to \cite{aue:kirch:2024}. 

\textbf{Change point statistics} Sequential tests for the hypotheses pair \eqref{e:H0X}-\eqref{e:H1X} are most typically based on some weighted CUSUM or MOSUM (moving sum) statistic. Following \cite{aue:kirch:2024}, such statistics can be decomposed into a sum-part, denoted here by $\hat \gamma$ and a weight function, denoted by $w$. We briefly list some common choices. Let $S_j:= \sum_{i=1}^j X_i$, $j \geq 1$ and $b \in (0,1)$, then, we define 
\begin{align*}
    (1) \quad \textnormal{CUSUM} & \quad \hat \gamma^{C}(k):=\Big|\frac{k}{N}S_N-(S_{N+k}-S_{N})\Big|\\
   (2) \quad \textnormal{Page CUSUM} & \quad \hat \gamma^{PC}(\ell,k):=\Big|\frac{k-\ell}{N}S_N-(S_{N+k}-S_{N+\ell})\Big|\\
   (3) \quad  \textnormal{Full CUSUM} & \quad \hat \gamma^{FC}(\ell,k):=\Big|\frac{k-\ell}{N+\ell}S_{N+\ell}-(S_{N+k}-S_{N+\ell})\Big|\\
   (4) \quad  \textnormal{mMOSUM} & \quad \hat \gamma^{M}(k):=\Big|\frac{k-\lfloor kb \rfloor}{N}S_N-(S_{N+k}-S_{N+\lfloor kb \rfloor})\Big|.
\end{align*}
The above statistics can be calculated at time $k$ in the monitoring period, where the data $X_1,...,X_N,...,X_{N+k}$ are available. The statistic $\hat \gamma^{C}$ was introduced by \cite{horvath:huskova:kokoszka:steinebach:2004} and is geared towards the detection of early changes. The Page CUSUM and later the Full CUSUM were introduced by \cite{fremdt:2015} and \cite{dette:goesmann:2018} respectively, motivated by likelihood ratio considerations.  
The mMOSUM was introduced by \cite{chen:tian:2010} and a detailed analysis is given in \cite{kirch:weber:2018}. mMOSUM stands for modified MOSUM and for an analysis of the more classical MOSUM statistic that we do not discuss here further, we refer to \cite{aue:horvath:kuhn:steinebach:2012}. \\
To combine values of $\hat \gamma$ for different values $k=1,2,...$ during monitoring, one multiplies it with a weight $w$ and then maximizes over $k$ (and $\ell$). The choice of the weight function is essential, because it guarantees convergence of the test statistic under the null hypothesis, and critically determines the quality of change point detection under the alternative.
A popular choice in the literature, introduced by \cite{horvath:huskova:kokoszka:steinebach:2004}, is
\begin{align}\label{e:weight1}
 w^{(1), \eta}(k) & := N^{-1/2}
\Big(\frac{N+k}{N}\Big)^{-1}\Big(\frac{N+k}{k}\Big)^\eta, \quad 0 \leq \eta < 1/2. 
\end{align}
Choices of $\eta$ close to $1/2$ enhance fast detection of early changes, where  $k^\star \ll N$. $w^{(1), \eta}$ has been used in combination with all four above statistics.
A recent version that leads to shorter detection delays, even when $k^\star \asymp N$, has been introduced by \cite{kutta:doernemann:2025} for the Page CUSUM method and it is clear that it can also be combined with the Page CUSUM. It is the class of polynomial Hölder weights, defined as
\begin{align}\label{e:weight2}
 w^{(2), \eta}(\ell, k)  :=N^{1/2}(N+k)^{\eta-1}(k-\ell)^{-\eta} \log^{-1}(C_0+(N+k)/N), \quad 0 \leq \eta < 1/2,
\end{align}
where as before, $\eta$ closer to $1/2$ leads to shorter detection delays. $C_0>1$ is a constant that impacts finite sample performance and is discussed further in our experiments. \\
We provide a brief performance summary for the above statistics. Equipped with the weight function $w^{(1), \eta}$ statistics (1)-(4) deliver consistent tests, whenever $\Delta \gg N^{-1/2}$ (e.g., Theorem 2 in \cite{aue:kirch:2024}). (2)-(3) may be equipped with $w^{(2), \eta}$ and the consistency condition is the same if $k^\star \lesssim N$ and changes by logarithmic factors if $k^\star \gg N$ (Theorem 3.4 in \cite{kutta:doernemann:2025}). If the weight function $w^{(1), \eta}$ is used and $N \lesssim k^\star$, one can show that detection delays are of order $k^\star/(\Delta N^{1/2})$. If $k^\star$ is very small ($N^\zeta$ for some small enough $\zeta=\zeta(\eta)<1$) short delays of order $\bigl(N^{1/2-\eta}/\Delta_N\bigr)^{1/(1-\eta)}$ can be attained (Theorem 1.1, \cite{aue:horvath:2004}). If (2)-(3) with $w^{(2), \eta}$ are used, delays are of the shorter order $\bigl(N^{1/2-\eta}/\Delta_N\bigr)^{1/(1-\eta)}$ as long as $k^\star \lesssim N$ (Lemma 3.5, \cite{kutta:doernemann:2025}). Afterwards, neglecting logarithmic factors, delays are of order $(k^\star/N)\bigl(N^{1/2-\eta}/\Delta_N\bigr)^{1/(1-\eta)}$. In all  above combinations, delays are of polynomial order, later delays even scale linearly in $k^\star$ and 
consistency of the tests requires alternatives of size $\gg N^{-1/2}$.

\textbf{Multiscale methods} Strong weighting of short windows is also central to retrospective multiscale methods, beginning with \cite{duembgen:spokoiny:2001}. Related approaches to change point inference include SMUCE \citep{frick:munk:sieling:2014}, Narrowest Significance Pursuit \citep{fryzlewicz:2024} and the optimal detection and localization results of \cite{verzelen:fromont:lerasle:reynaud-bouret:2023}. Classical multiscale procedures based on partial sums commonly rely on concentration inequalities under subgaussian or subexponential tail assumptions.

\textbf{Theoretical benchmark} The work of \cite{yu:padilla:wang:rinaldo:2023} provides an important theoretical benchmark for detection power and delay. Their procedure repeatedly applies a retrospective CUSUM to independent subgaussian observations, with thresholds that control the false-alarm probability in finite samples. It is formulated without a prescribed training period. Expressed in the notation of our framework, its detection delays for fixed jump sizes are logarithmic in $N+k^\star$ and remain short even for late changes. Moreover, consistency requires only $\Delta \gg (N+k^\star)^{-1/2}$, up to logarithmic factors. These rates are shown to be optimal in their setting.\\
TWIN attains the corresponding optimal detection and delay rates in the training-sample framework and provides a feasible, asymptotically sharp calibration of the nominal level. The finite-sample guarantees of \cite{yu:padilla:wang:rinaldo:2023} use a known upper bound on the subgaussian noise scale. Our numerical experiments show that the empirical-Orlicz implementation of their threshold can be very conservative and illustrate the benefits of a sharper calibration for finite-sample efficiency. The authors also discuss simulation-based calibration; implementations used in our numerical comparison are described in Appendix~\ref{sec:yu-implementation}.\\
Our theory further covers many scenarios that are outside the scope of \cite{yu:padilla:wang:rinaldo:2023}. These include mean changes in weakly dependent time series taking values in a separable Banach space, as well as general changes in the marginal distribution of real-valued data. Perhaps most interestingly, for mean changes we can even drop the assumption of subgaussianity. For this purpose, we introduce a time-adaptive version of TWIN that requires only finite moments of low order. Its data-dependent clock allows strong multiscale weighting under much weaker tail assumptions.

\section{Statistical Methodology} \label{sec:2}

This section presents TWIN and its main theoretical properties. Section~\ref{sec:21} introduces the construction, detection properties and critical weighting for iid subgaussian noise. Section~\ref{sec:23} extends the theory to dependent Banach-valued data; Section~\ref{ac:sec} introduces the adaptive version under low polynomial moment conditions.

\subsection{The TWIN CUSUM for independent data}\label{sec:21}

Recall the change point model and hypotheses
\eqref{e:H0X}--\eqref{e:H1X}. To scan for mean changes, we define
the following CUSUM statistic, where $S_j=\sum_{i=1}^j X_i$:
\[
    \textnormal{TWIN CUSUM} \qquad
    \hat \gamma^{TC}(\ell,k)
    :=\Big|\min(1,\ell/N)S_{\max(\ell,N)}
    -(S_{N+k}-S_{N+k-\ell})\Big|.
\]
TWIN stands for two-window inspection and refers to the almost
symmetric structure of this contrast. For $\ell\geq N$, it compares
sums of equal length, containing the first and last $\ell$
observations. We combine the
contrasts across window lengths using the weights
\[
 w^{TC,\beta}(\ell,k)
 :=\bigg(\sqrt{\ell}\Big[
 \log^\beta(C_0+N/\ell)
 +\log^\beta(C_0+(N+k)/N)\Big]\bigg)^{-1},
 \quad \beta\geq 1/2,\,\, C_0>1.
\]
We examine the interplay between contrasts and weights in Appendix~\ref{sec:common-weight}, where we compare other existing methods using this weight, showing why the TWIN CUSUM is particularly suitable for it. Under the null, a contrast at scale $\ell$ has random fluctuations
of order $\sqrt{\ell}$, explaining the first factor in the weight.
The logarithmic terms account for the many scans over window
lengths $\ell$ and monitoring times $k$. The parameters $\beta$
and $C_0$ are discussed further in Remark~\ref{rem:main}.
The benefit under the alternative is already visible. If the
window $\ell=k-k^\star$ is admissible, the second sum in $ \hat \gamma^{TC}(\ell,k)$ contains only post-change data, while the first sum only contains pre-change data. Ignoring random fluctuations, we obtain
\[
 w^{TC,\beta}(k-k^\star,k)\,
 \hat \gamma^{TC}(k-k^\star,k)
 \approx
 \frac{|\mu^{(1)}-\mu^{(2)}|\sqrt{k-k^\star}}
 {\textnormal{log-terms}}.
\]
For a fixed jump size, this suggests that detection requires only
a (poly)logarithmic delay. The full TWIN detector maximizes over
all admissible windows:
\begin{align}\label{e:TC}
 \widehat{\Gamma}^{TC}(k)
 :=\max_{\substack{1\leq\ell\leq (N+k)/2\\ \ell\leq k}}
 \big[w^{TC,\beta}(\ell,k)\hat \gamma^{TC}(\ell,k)\big].
\end{align}
To understand its fluctuations over the entire monitoring period,
we turn to the partial sum process and its modulus of continuity.

\textbf{Moduli of continuity}
Throughout this subsection, the errors
$\varepsilon_i:=X_i-\mathbb{E}X_i$ are real-valued, iid and
subgaussian:
\begin{align}\label{e:SG}
 \|\varepsilon_1\|_{\Psi_2}\leq C_\varepsilon,
 \qquad
 \|X\|_{\Psi_2}
 :=\inf\{c>0:\mathbb{E}\exp(X^2/c^2)\leq 2\}.
\end{align}
Their partial sum process is
\begin{align}\label{eq:def:PN}
 P_N(t)
 :=N^{-1/2}\bigg(
 \sum_{i=1}^{\lfloor Nt\rfloor}\varepsilon_i
 +(Nt-\lfloor Nt\rfloor)\varepsilon_{\lfloor Nt\rfloor+1}
 \bigg),\qquad t\geq 0.
\end{align}
For Brownian motion on a bounded time interval, L\'evy's modulus
of continuity identifies $\sqrt{s\log(1/s)}$ as the scale of the
maximum fluctuations over short windows of length $s$. This
motivates the critical exponent $\beta=1/2$ in our TWIN weight. To analyze
TWIN over the full monitoring horizon, we define the
following weighted modulus for a continuous function
$f:[0,\infty)\to\mathbb{R}$:
\[
 \omega^\beta(f,\mathcal{S})
 :=\sup_{(s,t)\in\mathcal{S}}
 \frac{\big|\min(1,s)f(\max(1,s))-(f(t)-f(t-s))\big|}
 {\sqrt{s}\big\{\log^\beta(C_0+1/s)
 +\log^\beta(C_0+t)\big\}}.
\]
The full scan set is
\[
 \mathcal{S}^{full}
 :=\{(s,t):s>0,\ t-s\geq 1,\ 2s\leq t\}.
\]
Under the null, evaluating this functional at $P_N$ and restricting
the scan to grid points $(s,t)=(\ell/N,1+k/N)$ gives exactly
$\sup_{k\geq 1}\widehat{\Gamma}^{TC}(k)$. The entire set $\mathcal{S}^{full}$ is used in the limit. The next lemma controls the functionals for partial sums and Brownian motion.

\begin{lem}\label{lem:twin:modulus}
 Let $B$ be standard Brownian motion and suppose that the iid
 errors satisfy \eqref{e:SG}. For every $\beta\geq 1/2$, there
 exists a constant $C_P>0$, depending only on $\beta$ and $C_0$,
 such that
 \[
 \big\|\omega^\beta(B,\mathcal{S}^{full})\big\|_{\Psi_2}
 \leq C_P,
 \qquad
 \sup_{N\geq 1}
 \big\|\omega^\beta(P_N,\mathcal{S}^{full})\big\|_{\Psi_2}
 \leq C_P C_\varepsilon.
 \]
\end{lem}
The lemma's result for the Brownian motion is  a direct consequence of Lévy's modulus of continuity theorem over the entire positive line, which was established by \cite{chevallier2023}. We extend it in the Supplement to a large class of processes with subgaussian increments, including partial sums of subgaussian data. 

Writing $\sigma^2=\mathbb{E}\varepsilon_1^2$, Donsker's principle gives $P_N\overset{d}{\to} \sigma B$ on compact time intervals.
This suggests the Brownian functional
$\omega^\beta(\sigma B,\mathcal{S}^{full})$ as a limiting distribution under the null hypothesis.
The difficulty is to extend the Brownian approximation of $P_N$ to arbitrarily
short windows and arbitrarily late monitoring times.

\textbf{Heuristic for above-threshold convergence}
The following discussion explains our main mathematical idea (without providing
an exact roadmap for the proof). Suppose continuous functions
$f_N$ converge to $f$ uniformly on compact intervals. We are interested in convergence of their moduli of continuity. Restricting
the scan to
\[
 \mathcal{S}^{T}
 :=\mathcal{S}^{full}\cap\{(s,t):s\geq 1/T,\ t\leq T\},
 \qquad T>2,
\]
 gives for any (fixed) $T$
\[
 \omega^\beta(f_N,\mathcal{S}^{T})
 \longrightarrow \omega^\beta(f,\mathcal{S}^{T}).
\]
As $T$ increases, the truncated limit  $\omega^\beta(f,\mathcal{S}^{T})$ approaches
$\omega^\beta(f,\mathcal{S}^{full})$. The only hurdle to
convergence of the full functionals is therefore an additional
contribution from scales or times escaping these truncated sets.
In particular, convergence follows if this contribution is
asymptotically no larger than the limiting maximum:
\begin{align} \label{e:heur}
\lim_{T\to\infty}\limsup_{N\to\infty}
 \omega^\beta(f_N,\mathcal{S}^{full}\setminus\mathcal{S}^{T})
 \leq \omega^\beta(f,\mathcal{S}^{full}).
\end{align}

For TWIN, $P_N$ and $\sigma B$ take the roles of $f_N$ and $f$ (where we may ignore the negligible error when going from discrete grids to continuous scan sets).
 When
$\beta>1/2$, the critical bound in Lemma~\ref{lem:twin:modulus} implies the required control. Indeed, it implies that the left-hand side of \eqref{e:heur} vanishes in probability
as $T\to\infty$, uniformly in $N$. Then, weak convergence to the Brownian functional
 goes through immediately.

At $\beta=1/2$, this is not true, and the scan of $P_N$ over $\mathcal{S}^{full}\setminus\mathcal{S}^{T}$ does not vanish. What does, however, hold is that with probability going to $1$ we have for some sufficiently large $q>0$
\[
\lim_{T\to\infty}\limsup_{N\to\infty}
 \omega^{1/2}(P_N,\mathcal{S}^{full}\setminus\mathcal{S}^{T}) \le q.
\]
Conditioning on the event $\{\omega^{1/2}(\sigma B,\mathcal{S}^{full})>q\}$ then makes the convergence argument work again (now weakly on the conditional probability space). Unconditionally this means that the  quantiles of $\omega^\beta(P_N,\mathcal{S}^{full}) $ larger than $q$ converge to those of $\omega^\beta(\sigma B,\mathcal{S}^{full})$. This phenomenon was first documented in retrospective settings by \cite{koehne:mies:2025} and labeled \textit{above-threshold convergence}.

\textbf{Properties of TWIN.}
We now provide the main weak convergence theorem under the null hypothesis of no change.
\begin{theo}\label{thm:main:1}
    Suppose that the null hypothesis $H_0$ in \eqref{e:H0X} holds, and that model errors $(\varepsilon_i)_i$ are iid and subgaussian in the sense of \eqref{e:SG}. Moreover, denote by $\{B(x):x\ge 0\}$ a centered Brownian motion with variance parameter $\sigma^2:=\mathbb{E}\varepsilon_1^2$.  Then, it follows for any $\beta>1/2$ that
    \begin{align} \label{e:conv:1}
       \sup_{k \ge 1}\widehat{\Gamma}^{TC}(k) \overset{d}{\to} \omega^\beta(B, \mathcal{S}^{full}).
    \end{align}
    For $\beta=1/2$, the weak convergence in \eqref{e:conv:1} needs to be replaced by above-threshold convergence.
\end{theo}

Based on Theorem \ref{thm:main:1}, we now propose the following detection scheme: If $q_{1-\alpha}$ is the $(1-\alpha)$-quantile of the limiting distribution $\omega^\beta(B, \mathcal{S}^{full})$, a change point is declared at time $ k$ if 
\begin{align} \label{e:decision}
\widehat{\Gamma}^{TC}(k)>q_{1-\alpha}.
\end{align}
For $\beta=1/2$, this calibration applies when $q_{1-\alpha}$ is above the convergence threshold. The corresponding time of first detection is denoted by $\hat k$. The limiting distribution $\omega^\beta(B, \mathcal{S}^{full})$ can be readily simulated, after estimating the noise variance $\sigma^2$.

\begin{cor}\label{cor:feasible-scalar}
Suppose the assumptions of Theorem~\ref{thm:main:1} hold, with $\beta>1/2$ and $\sigma>0$. Let $\widehat\sigma_N>0$ be based on $X_1,\ldots,X_N$ and satisfy $\widehat\sigma_N\overset P\to\sigma$. Then, under $H_0$,
\[
\sup_{k\geq1}\frac{\widehat\Gamma^{TC}(k)}{\widehat\sigma_N}\overset d\to \omega^\beta(B_1, \mathcal{S}^{full}),
\]
where $B_1$ is the standard Brownian motion. The result remains true for $\beta=1/2$ in a beyond threshold sense.
\end{cor}

The corollary follows directly from Slutsky's theorem, and a simple candidate estimator $\widehat\sigma_N$ would be the empirical standard deviation from the training sample.  
Next, we  study the power and delay of the corresponding procedure. 
\begin{cor}
\label{cor:power:delay}
    Suppose that the model errors $(\varepsilon_i)_i$ are iid and subgaussian in the sense of \eqref{e:SG}. Fix $\beta \ge 1/2$ and suppose that the alternative \eqref{e:H1X} holds, with a change occurring at time $N+k_N^\star$, and change magnitude  of size $\Delta_N$ (see \eqref{e:jumpsize}). Then, if 
    \begin{align} \label{e:conscond}
    \Delta_N \gg \frac{\log^\beta(C_0+N/k_N^\star)+\log^\beta(C_0+k_N^\star/N)}{\sqrt{N+k^\star_N}} 
    \end{align}
    it follows that
    $\mathbb{P}(\hat k<\infty)\to 1.$
    If condition \eqref{e:conscond} holds the delay time satisfies
    \begin{align} \label{e:det:rate}
    \max(0, \hat k-k_N^\star)= O_P \bigg( 1+\frac{\big\{\log^\beta(N)+\log^\beta(C_0+k^\star_N/N)\big\}^2}{\Delta_N^2}\bigg).
        \end{align}
\end{cor}

The power and delay conclusions remain true for the feasible rejection rule in Corollary~\ref{cor:feasible-scalar}, because the training sample is unchanged under the alternative and $\widehat\sigma_N\overset P\to\sigma$.

To understand the implications of the above corollary, it is helpful to consider some exemplary cases. First, consider the case where the change is very small, i.e.  $\Delta_N\downarrow0$.
If $k_N^\star\asymp N$, a change is detected consistently when $\Delta_N\gg N^{-1/2}$, the established (and rate optimal) condition in the related literature. If $k_N^\star\gg N$, the change can, neglecting logarithmic terms, be even smaller, of size $\Delta_N\gg(N+k_N^\star)^{-1/2}$. Thus, a large pre-change sample can make subtle late changes detectable.
On the other hand, if changes are larger, short delays become important. For example, if $\Delta_N=\Delta$ is non-decaying, our approach produces delays of order $\big\{\log^\beta(N)+\log^\beta(C_0+k_N^\star/N)\big\}^2$.
For the common training-based procedures discussed in Section~\ref{sec:lit}, delays typically grow at a polynomial rate in $N$ and can even increase linearly with $k_N^\star$. TWIN replaces this deterioration by a logarithmic dependence on the change time, yielding short delays even for very late changes.

\begin{rem}\label{rem:main} $ $
    \begin{itemize}
        \item[1)] \textit{On the proof technique} There are broadly two standard strategies to validate sequential monitoring procedures: The first is to use strong approximations to replace the entire partial sum process of model errors by a Brownian motion. Such results are typically available from KMT-type approximations (see \cite{kmt:1976}). A different strategy is to carve up the detector into arguments $k \le TN$ for some finite $T$ and $k>NT$. To handle values of $k \le TN$ a weak invariance principle is sufficient and for $k>TN$ a H\'ajek--R\'enyi inequality is used to show that the noise contribution is (almost) negligible (see, \cite{aue:kirch:2024}). We fuse both approaches in a new way. A KMT-type approximation is used to approximate $\hat \gamma^{TC}(\ell,k)$ by a Gaussian for values of $k$ that are not too large and values of $\ell$ that are not too small. The remaining scales are challenging, but can be uniformly controlled, using our modulus of continuity result from Lemma \ref{lem:twin:modulus}. This can be seen as an expansion of the H\'ajek--R\'enyi strategy on a variety of scales, where Gaussian approximations do not hold. 
   
        \item[2)] \textit{Parameter choices} The TWIN weight function involves two parameters, $\beta$ and $C_0>1$. These parameters are similar to those in the weight functions $w^{(i),\eta}$ in Section~\ref{sec:lit}. In our simulations, we generically use the weight $\beta=1/2$, which is associated with the rate optimal detection. For the constant $C_0$ we choose $C_0=20$; alternative choices only have a modest impact on practical performance, with smaller choices generally leading to slightly larger delay in finite samples.

\end{itemize}
\end{rem}

\subsection{Extensions for dependent data in Banach spaces} \label{sec:23}

So far, we have described the idea of the TWIN detector for iid, subgaussian, real-valued data. In the following, we develop several substantial extensions, with the main body of this paper dedicated to changes in the mean parameter. Extensions to general distributional changes are discussed in  Supplement \ref{sec:22}. For the rest of this section, we justify TWIN simultaneously for data that may be serially dependent  and may live on a general, separable Banach space. Extending TWIN beyond subgaussian data is a separate challenging step, and described in the next section.

\textbf{Data in a Banach space} In the following, denote by $(\mathbb B,|\cdot|_{\mathbb B})$ the separable Banach space, where our data are observed. Important instances include (vectors of) scalar observations, data on Hilbert spaces and functional data. Covering Hilbert-valued time series is of interest, because often data can be mapped into a Hilbert space through a kernel embedding; see \cite{Zhang:Zhu:Shao:2026} for a recent discussion. Functional data arise in many applications, such as finance, and for an introduction we refer to the monograph of \cite{HKbook}. Below we will discuss the concrete case of a functional time series on the Banach space of continuous functions $\mathcal{C}([0,1])$.

As before we decompose each observation $X_i$ into
$X_i=\mu_i+\varepsilon_i$, where $\mu_i \in \mathbb B$ is a deterministic mean parameter and
$\varepsilon_i$ a centered observational noise on $ \mathbb B$. Expectations (and centeredness) are understood
in the Bochner sense. To describe Gaussian limits, we invoke the Brownian motion on a Banach space, which is discussed in the classical work of \cite{kuelbs:1973}. 

\textbf{Extension of TWIN theory}
The theory of this section is based on two high-level assumptions, that can be validated in different settings. Notice that they are obviously true in the case of iid real-valued data, treated in the previous section.

\begin{ass}\label{ass:banach:dependence}
The sequence $(\varepsilon_i)_{i\in\mathbb Z}$ is strictly stationary
in $\mathbb B$, with $\mathbb E\varepsilon_0=0$, and it satisfies:

\noindent\textit{(i) Subgaussian partial sums.}
There exists $C<\infty$ such that
\begin{equation}\label{e:banach:increments}
 \left\|
  \left|\sum_{i=a+1}^{a+m}\varepsilon_i\right|_{\mathbb B}
 \right\|_{\Psi_2}
 \le C\sqrt m,\qquad a\in\mathbb Z,\quad m\ge1.
\end{equation}

\noindent\textit{(ii) Gaussian approximation.}
There exist $\gamma\in(0,1/2)$ and a covariance operator
$\mathbf C_{LR}$ such that, for every $n$, one can couple
$\varepsilon_1,\ldots,\varepsilon_n$ with a centered
$\mathbb B$-valued Brownian motion $W_n$ satisfying
$W_n(1)\sim\mathcal N(0,\mathbf C_{LR})$ and
\begin{equation}\label{e:banach:approx}
 \max_{1\le j\le n}
 \left|\sum_{i=1}^j\varepsilon_i-W_n(j)\right|_{\mathbb B}
 =O_P(n^\gamma).
\end{equation}
The coupling may depend on $n$.
\end{ass}

Assuming stationary model errors is standard in monitoring. The Assumptions  \eqref{e:banach:increments}-\eqref{e:banach:approx} are then the two main ingredients for our proof of convergence for the TWIN detector.
Condition \eqref{e:banach:approx} provides the main Gaussian approximation; the approximation is valid for
times up to $N^{1+\zeta}$ and for windows containing at least
$N^{1-\zeta}$ observations, where $\zeta>0$ is sufficiently small.
Intuitively, sums on these scales contain enough observations for
the Gaussian approximation to become accurate.
Condition \eqref{e:banach:increments} uniformly controls the remaining
scales: very short scales, where sums contain too few observations
for a CLT to hold, and very late monitoring times, that are not reached by our Gaussian approximation.

For independent subgaussian observations 
the isoperimetric inequalities of \cite{talagrand:1989} can be used
to directly establish \eqref{e:banach:increments} provided that \eqref{e:banach:approx} holds. Polynomial-rate strong
approximations as in \eqref{e:banach:approx} are available under
suitable finite-dimensional approximation and moment conditions;
see \cite[Theorem~6.1]{dudley:philipp:1983}.
For physically dependent real-valued data,
\cite[Theorem~2.8]{koehne:mies:2025} provides the subgaussian
control, while \cite[Theorem~3(ii)]{wu:2005} gives a strong
approximation under suitable decay of the physical dependence
coefficients. Below, we verify both conditions for the more difficult case of dependent
functional time series, which requires a novel strong Gaussian approximation.

We can now generalize Theorem \ref{thm:main:1}; therefore let the B-TWIN detector be defined as the TWIN detector before, but with the absolute values replaced by the Banach-norm, i.e.
\begin{align*}
 \hat\gamma^{BTC}(\ell,k)
 &:=
 \Big|\min(1,\ell/N)S_{\max(\ell,N)}
 -(S_{N+k}-S_{N+k-\ell})\Big|_{\mathbb B},\\
 \widehat\Gamma^{BTC}(k)
 &:=
 \max_{\substack{1\le\ell\le k\\\ell\le(N+k)/2}}
 w^{TC,\beta}(\ell,k)\hat\gamma^{BTC}(\ell,k),
\end{align*}

\begin{theo}
\label{thm:banach:dependent}
Suppose Assumption~\ref{ass:banach:dependence} holds.
Let $\beta > 1/2$, $C_0>1$, and let $B$ be a centered
$\mathbb B$-valued Brownian motion with
$B(1)\sim\mathcal N(0,\mathbf C_{LR})$.
Under the null hypothesis of a constant mean parameter, the weak convergence 
\begin{equation*}
 \sup_{k \ge 1}\widehat{\Gamma}^{BTC}(k) \overset{d}{\to} \omega^\beta(B, \mathcal{S}^{full}),
\end{equation*}
holds. For $\beta=1/2$, weak convergence needs to be replaced by above-threshold convergence.
For any $\beta\ge 1/2$, the  limiting random variable is almost surely finite.\\
Now, suppose instead that a mean change occurs after observation
$N+k_N^\star$ as in \eqref{e:H1X}, with
$\Delta_N:=|\mu^{(2)}-\mu^{(1)}|_{\mathbb B}$. If
\begin{equation}\label{eq:banach:consistency}
 \frac{\Delta_N\sqrt{N+k_N^\star}}
 {\log^\beta\!\big(C_0+(N+k_N^\star)/N\big)}
 \longrightarrow\infty,
\end{equation}
then $\sup_{k\ge1}\widehat\Gamma^{BTC}(k)
\overset P\longrightarrow\infty$, and the detection delay for
any fixed positive critical value satisfies \eqref{e:det:rate}.
\end{theo}

\textbf{Illustration for dependent functional time series.} A relevant specialization of detecting a change in a Banach-valued time series comes from the study of functional data. Functional data can be modeled on different underlying spaces, with different topologies, and here we consider the space of continuous functions $\mathcal C([0,1])$ equipped with the standard supremum norm $\|\cdot \|_\infty$. Change point detection in this space was studied for retrospective settings by \cite{dette:kokot:aue:2020} and in the sequential setting by \cite{kutta:kokoszka:2025} (with the traditional single-window inspection from \cite{horvath:huskova:kokoszka:steinebach:2004} and the associated delay and power guarantees).  
We now verify Assumption~\ref{ass:banach:dependence} for such time series, making the more efficient TWIN detector suitable for this setting. The associated Gaussian coupling that we prove is of independent interest.

\begin{ass}\label{ass:fda:dependence}
The centered errors take values in $\mathcal C([0,1])$ and admit
the representation $\varepsilon_i(u)=G(u,\mathcal F_i)$,
where $\mathcal F_i=(\ldots,\eta_{i-1},\eta_i)$, for some iid innovations
$(\eta_i)_i$. Let $\varepsilon_h'$ be obtained by replacing
$\eta_0$ with an independent copy.
For some $M<\infty$, $\rho>0$ and $\theta\in(0,1]$,
all $h\ge0$ and $u,v\in[0,1]$,
\begin{gather*}
 \big\|\|\varepsilon_0\|_\infty\big\|_{\Psi_2}\le M,\qquad
 \sup_{u\in[0,1]}
 \|\varepsilon_h(u)-\varepsilon_h'(u)\|_{\Psi_2}
 \le M(h+1)^{-2-\rho},\\
 \big\|
  [\varepsilon_h(u)-\varepsilon_h(v)]
  -[\varepsilon_h'(u)-\varepsilon_h'(v)]
 \big\|_{\Psi_2}
 \le M|u-v|^\theta(h+1)^{-2-\rho}.
\end{gather*}
\end{ass}
These conditions combine strict stationarity, subgaussian tails and weak dependence.  As an example, consider the stationary functional autoregression
$\varepsilon_i(u)=\phi\varepsilon_{i-1}(u)+\eta_i(u)$,
$|\phi|<1$, with centered iid continuous innovations satisfying
\[
 \big\|\|\eta_0\|_\infty\big\|_{\Psi_2}<\infty,\qquad
 \|\eta_0(u)-\eta_0(v)\|_{\Psi_2}\lesssim|u-v|^\theta.
\]
Here $\varepsilon_h-\varepsilon_h'=\phi^h(\eta_0-\eta_0')$,
so the dependence bounds hold with geometric decay.

\begin{prop}\label{prop:fda:verification}
Suppose Assumption~\ref{ass:fda:dependence}. Then, 
Assumption~\ref{ass:banach:dependence} holds for
$\mathbb B=\mathcal C([0,1])$.
The long-run covariance operator $\mathbf C_{LR}$ in the Gaussian approximation is characterized by the kernel
\[
 c_{LR}(u,v)=\sum_{h\in\mathbb Z}
 \operatorname{Cov}(\varepsilon_0(u),\varepsilon_h(v)).
\]
\end{prop}

The proof is given in Supplement~\ref{sec:critical-c01}. Feasible calibration by a Gaussian bootstrap based on an estimated long-run covariance operator is justified in Supplement~\ref{sec:banach-bootstrap}.

\section{The adaptive TWIN detector}
\label{ac:sec}

In this section, we introduce an adaptive extension of TWIN that allows
for heavier-tailed errors. We develop the theory for dependent,
real-valued data with a finite $2+\delta$ moment for some $\delta>0$. Moving beyond
subgaussianity is non-trivial, because it plays a major role in the analysis of the original
TWIN detector. More precisely, subgaussianity ensures that the largest absolute error among $n$
observations is of order $\mathcal{O}_P((\log n)^{1/2})$, which explains
the critical threshold $\beta=1/2$ in our weights. With heavier tails (e.g. Student $t$ errors),
isolated outliers overwhelm the statistic, making it diverge under $H_0$ for any $\beta \geq 1/2$. We circumvent this problem by rescaling time in the partial sum
process.

\textbf{Motivation.}
Here, the Dambis--Dubins--Schwarz theorem offers a useful analogy.
Let $M$ be a continuous local martingale starting at zero, with quadratic
variation $A(t):=\langle M\rangle_t$ tending to infinity almost surely.
Then the theorem states that
$B(u):=M(A^{-1}(u))$
is a standard Brownian motion, where $A^{-1}$ denotes the generalized
inverse of $A$. Thus, a clock change can turn a challenging process with heavy tails and
rough sample paths into a Brownian motion, which, despite its reputation
for roughness, is comparatively smooth for a continuous local martingale.  The identity $M(t)=B(A(t))$ explains how this works:
an interval $[s,t]$ occupies $A(t)-A(s)$ units of time on the new clock,
rather than $t-s$ in calendar time. Intervals with large variation are
stretched out, spreading sharp movements over longer periods and making
the process as smooth as Brownian motion.

\textbf{Adaptive partial sum processes.}
A discrete version of this idea underlies the adaptive, self-normalized
partial sum processes studied by \cite{Rackauskas:Suquet:2004}.
In calendar time, the partial sum process progresses by $1/N$ for each new observation (the scaling in Donsker's theorem). On the adaptive clock, time
progresses according to accumulated squared errors: a large error receives
a longer interval. Geometrically, a large step in the $y$-direction is accompanied by a large step in the $x$-direction, softening the sharp movements caused by outliers. Originally,
\cite{Rackauskas:Suquet:2004} established this approach for independent
observations in an offline setting and with a known mean, making the scope of direct applications somewhat narrow.
Our main mathematical contribution here is to extend this approach to
dependent data with an unknown mean, using empirical centering, in an online setting. Moreover, we treat the critical case $\beta=1/2$, which is not covered, even in the iid case, by the arguments in \cite{Rackauskas:Suquet:2004}.

In this section, let $k$ denote the total number of observations so far,
including the training sample of size $N$.
We first write $S_j=\sum_{i=1}^jX_i$ and
$R_j=\sum_{i=1}^j\varepsilon_i$, with $S_0=R_0=0$.
For $C_0>1$ we define the continuous version of our weights for $\beta=1/2$ 
\begin{equation}\label{ac:rho}
 \rho_{1/2}(s,t)=\sqrt{s}\{\sqrt{\log(C_0+1/s)}+
 \sqrt{\log(C_0+t)}\},\qquad s>0,\ t\ge1.
\end{equation}
We now give the definition of the adaptive partial sum process. As building blocks for the adaptive clock we define
\begin{align}\label{ac:clock}
 &A_{k,j}=\sum_{i=1}^je_{k,i}^2,\qquad A_k=A_{k,k},\qquad A_{k,0}=0,\\
  \textnormal{where} \qquad  &  e_{k,i}=X_i-\bar X_k,
 \qquad \bar\varepsilon_k=R_k/k. \nonumber
\end{align}
For notational simplicity we assume that the finite-dimensional distributions of the observations are absolutely continuous. Then $e_{k,i}\ne0$ almost surely for every
$k\ge2$ and $1\le i\le k$, so \eqref{ac:clock} is strictly
increasing. Let $M_{N,k}$ be the
polygonal process obtained by linear interpolation between
\begin{equation}\label{ac:polygon}
 \left(\frac{A_{k,j}}{A_k},
       \frac{R_j}{\mathcal V_{N,k}}\right),
 \qquad j=0,\ldots,k.
\end{equation}
Here, $\mathcal V_{N,k}$ is a measurable function of
$X_1,\ldots,X_k$, chosen so that $\mathcal V_{N,k}^2/k$ estimates the long-run variance. More precisely, we assume
\begin{equation}\label{ac:global-V}
 \sup_{k\ge N}\left|\frac{\mathcal V_{N,k}^2}{k}-\sigma_{LR}^2\right|
 \overset P\longrightarrow0.
\end{equation}
We present two explicit constructions, based on the empirical variance for iid data and a non-overlapping batch-means estimator for dependent time series (see e.g. \cite{flegal:jones:2010}; consistency is shown in Proposition~\ref{ac:dependent-clock}).
\begin{enumerate}
    \item For iid errors, the long-run variance is equal to the error variance $\sigma_{LR}^2=\sigma_0^2:=\E\varepsilon_0^2$, and taking
    $\mathcal V_{N,k}^2=A_k$
    means that $\mathcal V_{N,k}^2/k$ is the empirical variance of all data observed until time $k$. Thus \eqref{ac:global-V} is obvious.
    \item For dependent errors, let $b=b_N$ be an integer with $b_N\to\infty$ and
    $b_N/N\to0$. Set
    \[
     B_{N,k,r}=\{(r-1)b+1,\ldots,\min(rb,k)\},\qquad
     \ell_{N,k,r}=|B_{N,k,r}|,\qquad 1\le r\le\lceil k/b\rceil,
    \]
    and define
    \begin{equation}\label{ac:block-clock}
     D_{N,k,r}=\sum_{i\in B_{N,k,r}}e_{k,i},\qquad
     \mathcal V_{N,k}^2=\sum_rD_{N,k,r}^2.
    \end{equation}
\end{enumerate}
We call $\big(A_{k,j}/A_k\big)_{j=0,\ldots,k}$ the adaptive clock for our partial sum process. On macroscopic scales, it behaves like calendar time: for every fixed $T>1$,
\begin{equation}\label{ac:compact-clock}
 \sup_{N\le k\le NT}\max_{0\le j\le k}
 \left|\frac{A_{k,j}}{A_k}-\frac jk\right|
 \overset P\longrightarrow0.
\end{equation}
Because of this macroscopic equivalence, we recover the Brownian motion in our limits, just as Donsker's theorem suggests. On short time intervals, however, the adaptive clock assigns more time to intervals with larger errors. This allows us to control small increments of $M_{N,k}$ under the strong weighting used by TWIN.

\textbf{Adaptive TWIN.}
The construction is now simple in principle: we use the TWIN contrast and replace the window length in its weight by the adaptive length. We call the resulting detector $\widehat\Gamma^{ACTC}(k)$ and the test statistic $T_N^{AC}$. Readers who are not interested in the details of the construction can proceed to the next subsection. For $1\le\ell\le\min(k/2,k-N)$, we put
$m=\max(N,\ell)$ and $\alpha_{N,\ell}=\min(1,\ell/N)$, and define
\begin{align}
 h^{(1)}_{N,k,\ell}&=\frac{A_{k,m}}{A_k},&
 h^{(2)}_{N,k,\ell}&=\frac{A_k-A_{k,k-\ell}}{A_k},\nonumber\\
 s_{N,k,\ell}&=(k/N)\max\{\alpha_{N,\ell}h^{(1)}_{N,k,\ell},
                         h^{(2)}_{N,k,\ell}\}.
 \label{ac:effective-window}
\end{align}
Here, $h_{N,k,\ell}^{(1)}$ and $h_{N,k,\ell}^{(2)}$ are the adaptive durations of the initial and final TWIN windows. In calendar time, both entries in the maximum equal $\ell/k$, since $\alpha_{N,\ell}m=\ell$. Their adaptive counterparts can differ; taking the maximum allows the weight to accommodate fluctuations in either window. By \eqref{ac:compact-clock}, $s_{N,k,\ell}=\ell/N+o_P(1)$ on fixed monitoring horizons, so the clock change vanishes on macroscopic scales. We define
\begin{equation}\label{ac:detector}
 \widehat\Gamma^{ACTC}(k)=
 \max_{\substack{1\le\ell\le k/2\\\ell\le k-N}}
 \frac{\sqrt{k/N}\,
 |\alpha_{N,\ell}S_m-(S_k-S_{k-\ell})|}
 {\mathcal V_{N,k}\rho_{1/2}(s_{N,k,\ell},k/N)},
 \qquad T_N^{AC}=\sup_{k>N}\widehat\Gamma^{ACTC}(k).
\end{equation}
If $s_{N,k,\ell}=0$, its numerator is zero and its
ratio is defined to be zero. Set $\widehat\Gamma^{ACTC}(k)=0$ when
$A_k\mathcal V_{N,k}=0$, an asymptotically negligible event.

\subsection{Assumptions and convergence beyond a threshold}
We now state the mathematical assumptions. Besides $(2+\delta)$ moments and mild decay of dependence coefficients, we need two high-level concentration assumptions that can be verified explicitly for certain time series models.

\begin{ass}\label{ac:ass-physical}
The errors have the stationary causal representation
$\varepsilon_i=g(\ldots,\nu_{i-1},\nu_i)$ for iid innovations $(\nu_i)$ and a measurable function $g$. They satisfy
$\E\varepsilon_0=0$, and $\|\varepsilon_0\|_q<\infty$ for $q=2+\delta>2$.
If $\varepsilon_j^{\{0\}}$ replaces $\nu_0$ by an independent copy, then
\begin{equation}\label{ac:physical}
 \delta_q(j):=\|\varepsilon_j-\varepsilon_j^{\{0\}}\|_q
 \le C(1+j)^{-\nu},\qquad j\ge0,\quad\nu>2.
\end{equation}
The long-run variance
$\sigma_{LR}^2=\sum_{j\in\mathbb Z}\operatorname{cov}(\varepsilon_0,
\varepsilon_j)$ is finite under the above physical dependence; we assume that it is also strictly positive.
\end{ass}

For a nonempty finite integer interval $I$, set
$R_I=\sum_{i\in I}\varepsilon_i$ and $Q_I=\sum_{i\in I}\varepsilon_i^2$.
\begin{ass}\label{ac:ass-local}
There are positive constants $C_S,c_S,C_E,c_E,v$ such that, uniformly in
$I$ and $x\ge0$,
\begin{align}
 \PR(|R_I|>x\sqrt{Q_I})&\le C_Se^{-c_Sx^2},\label{ac:SN}\\
 \PR(Q_I<v|I|)&\le C_Ee^{-c_E|I|}.\label{ac:LE}
\end{align}
\end{ass}

These concentration assumptions are closely related to classical self-normalization results; see \cite{Gine:Gotze:Mason}. They require separate verification for the error distribution and dependence model; Section~\ref{sec:example} verifies them for iid data and autoregressive processes. We obtain the same Brownian limit for the adaptive TWIN statistic.
\begin{theo}\label{ac:main}
Suppose Assumptions~\ref{ac:ass-physical} and \ref{ac:ass-local} hold.
Let the clock be defined by \eqref{ac:clock}, suppose that
$\mathcal V_{N,k}$ satisfies \eqref{ac:global-V} and let $B_0$ be a standard Brownian motion. 
Then, under $H_0$, $T_N^{AC}$ converges in an above-threshold sense to the limit
\begin{equation}\label{ac:main-limit}
L_{1/2}(1):=
 \sup_{t>1}\sup_{\substack{0<s\le t/2\\t-s\ge1}}
 \frac{|\min(1,s)B_0(\max(1,s))-[B_0(t)-B_0(t-s)]|}
 {\rho_{1/2}(s,t)}.
\end{equation}
\end{theo}
$ $

\section{Finite Sample Performance} \label{sec:finite}

We evaluate the performance of our new methods using simulations and a data example, the latter can be found in Appendix \ref{sec:application}. We compare them with a range of existing sequential change point tests, as well as with empirical and size-adjusted versions of the method of \cite{yu:padilla:wang:rinaldo:2023}. 

\subsection{Simulation study}\label{sec:simulation}

In our simulations, data are generated according to the following scheme:
\begin{align}
    X_i = \varepsilon_i + \Delta 1\{i \geq N + k^\star\}, \quad i = 1, \ldots, N + T.
\end{align}
Here, $\Delta$ denotes the size of the change, $N$ the length of the training period, and $T$ the length of the monitoring period. In our theory, $T=\infty$, but for simulations we have to choose a large but finite value. Throughout, we set $T=20N$. The location of the change within the monitoring period is denoted by $k^\star$.
We consider independent errors and AR(1) errors with lag-one correlation $1/3$. More precisely, the centered noise sequence $(\varepsilon_i)_i$ is generated as
\[
    \varepsilon_i=\phi\varepsilon_{i-1}+\sqrt{1-\phi^2}\,Z_i,
    \qquad \phi\in\{0,1/3\},
\]
where the innovations $(Z_i)_i$ are iid and follow one of the following distributions:
\begin{itemize}
    \item[(1)] Standard normal distribution.
    \item[(2)] Uniform distribution on $[-\sqrt{3},\sqrt{3}]$.
    \item[(3)] Student $t_3$ distribution scaled by $1/\sqrt{3}$.
\end{itemize}
All three distributions have mean zero and variance one. The scaling in the AR(1) model preserves the marginal variance. For dependent errors, the first $500$ observations are discarded. The normal and uniform settings satisfy the subgaussian requirement, whereas the $t_3$ setting allows us to study performance for heavier tails.

We compare our new methods to a range of benchmarks. Our new methods are denoted by TC (TWIN CUSUM for mean changes, labeled TWIN in the figures) and ACTC (its adaptive, self-normalized version). The remaining methods are the approaches discussed in Section~\ref{sec:lit}: C (CUSUM), PC (Page CUSUM), FC (Full CUSUM), and MM (modified MOSUM), each equipped with the weight function $w^{(1),\eta}$ from \eqref{e:weight1}. The WC (weighted CUSUM) method combines the Page CUSUM aggregator with the weight function \eqref{e:weight2}. We also include feasible versions of \cite{yu:padilla:wang:rinaldo:2023}: First,  one can apply their method using  an estimate of the data's Orlicz norm from the training data. Below this is referred to as  Yu-Orlicz. Second, one can use the training data for a bootstrap to get critical values of \cite{yu:padilla:wang:rinaldo:2023}. For iid data one can use an empirical bootstrap  (Yu-ECDF) and for dependent data we suggest circular block resampling with block lengths $5$ and $8$ (Yu-CBB).  Details on the implementation are given in Appendix~\ref{sec:yu-implementation}.

The weight functions in our above CUSUM methods require parameter choices. We fix $\eta=0.4$ in $w^{(1),\eta}$ and $w^{(2),\eta}$, and $\beta=1/2$ in $w^{TC,\beta}$. Moreover, in $w^{(2),\eta}$ and $w^{TC,\beta}$ we set $C_0=20$. For MM we let $b=0.4$ as in \cite{kirch:weber:2018}.
For independent data, TC, C, PC, FC, WC, and MM use the empirical variance of all observations available at the current monitoring time, with divisor $N+k$. Under dependence, we use a prewhitened long-run variance estimator following \cite{twin-sim:andrews:monahan:1992}, with prewhitening order one, OLS fitting, the Bartlett kernel, and the AR(1) plug-in bandwidth of \cite{twin-sim:andrews:1991}. Autoregressive coefficient estimates are restricted to $[-0.99,0.99]$. For TC, the estimate is based on the observations preceding the candidate window; for the remaining methods, it uses all currently available observations. ACTC uses the independent normalization for iid data and the dependent batch normalization with batch lengths $5$ and $8$ for AR(1) data, including the final incomplete batch.

The nominal level is always set to $\alpha=0.05$. Brownian critical values are approximated separately for each $N$ using $1000$ draws, with $1/N$-width grids in monitoring time and window length. All calibration horizons are $20N$ observations after training. Unless stated otherwise, rejection rates in the tables and figures are based on $1000$ simulation runs per setting.

\textbf{Level approximation:} We first investigate how well the methods approximate the nominal level for different lengths $N\in\{50,100,200\}$ of the training period. The results are displayed in Table \ref{tab:null}. For independent normal and uniform data, TC approximates the nominal level well, while ACTC is somewhat conservative. The remaining variance-normalized methods generally give reasonable level approximations, although some settings show moderate deviations from the nominal level.

\begin{table}[htbp]
\centering
\footnotesize
\setlength{\tabcolsep}{3pt}
\renewcommand{\arraystretch}{0.95}

\textit{(a) Independent errors}\par\smallskip
\begin{tabular}{@{}lrrrrrrrrr@{}}
\toprule
 & \multicolumn{3}{c}{$N=50$}
 & \multicolumn{3}{c}{$N=100$}
 & \multicolumn{3}{c}{$N=200$} \\
\cmidrule(lr){2-4}\cmidrule(lr){5-7}\cmidrule(lr){8-10}
Method & Normal & Uniform & $t_3$ & Normal & Uniform & $t_3$ & Normal & Uniform & $t_3$ \\
\midrule
ACTC & 2.4 & 4.6 & 1.5 & 3.1 & 3.7 & 1.8 & 3.9 & 3.4 & 2.7 \\
TWIN & 4.5 & 6.0 & 91.6 & 3.9 & 4.7 & 98.7 & 5.1 & 4.5 & 99.9 \\
C & 3.4 & 4.4 & 5.6 & 4.8 & 7.1 & 6.5 & 4.6 & 5.4 & 5.3 \\
PC & 3.2 & 4.4 & 5.3 & 5.0 & 7.8 & 6.9 & 4.2 & 5.1 & 6.0 \\
FC & 3.7 & 4.5 & 5.6 & 4.2 & 5.7 & 5.9 & 5.8 & 6.2 & 6.1 \\
WC & 3.2 & 4.8 & 23.6 & 3.5 & 5.7 & 41.9 & 4.7 & 5.8 & 55.9 \\
MM & 2.9 & 3.2 & 7.1 & 3.3 & 4.4 & 8.6 & 6.7 & 5.8 & 9.6 \\
Yu-Orlicz & 0.0 & 0.0 & 0.0 & 0.0 & 0.0 & 0.0 & 0.0 & 0.0 & 0.0 \\
Yu-ECDF & 15.9 & 9.8 & 84.5 & 12.0 & 7.1 & 86.1 & 9.7 & 5.9 & 89.1 \\
\bottomrule
\end{tabular}

\vspace{0.7em}
\textit{(b) AR(1) errors}\par\smallskip
\begin{tabular}{@{}lrrrrrrrrr@{}}
\toprule
 & \multicolumn{3}{c}{$N=50$}
 & \multicolumn{3}{c}{$N=100$}
 & \multicolumn{3}{c}{$N=200$} \\
\cmidrule(lr){2-4}\cmidrule(lr){5-7}\cmidrule(lr){8-10}
Method & Normal & Uniform & $t_3$ & Normal & Uniform & $t_3$ & Normal & Uniform & $t_3$ \\
\midrule
ACTC, $b=5$ & 4.4 & 4.6 & 2.2 & 6.6 & 6.7 & 3.3 & 5.6 & 7.5 & 3.7 \\
ACTC, $b=8$ & 3.2 & 2.5 & 1.0 & 3.3 & 3.4 & 1.9 & 3.9 & 5.0 & 2.6 \\
TWIN & 9.0 & 7.8 & 67.1 & 5.6 & 5.4 & 81.7 & 4.4 & 4.9 & 91.5 \\
C & 3.3 & 2.7 & 2.6 & 4.6 & 3.3 & 4.9 & 4.2 & 3.7 & 3.6 \\
PC & 2.8 & 2.3 & 2.2 & 4.6 & 3.9 & 5.1 & 4.2 & 3.9 & 3.8 \\
FC & 3.7 & 2.5 & 2.2 & 4.0 & 3.2 & 4.5 & 4.5 & 4.8 & 4.2 \\
WC & 2.3 & 2.1 & 5.6 & 2.7 & 3.0 & 11.8 & 3.6 & 3.7 & 20.9 \\
MM & 1.4 & 1.3 & 1.8 & 2.0 & 2.0 & 3.7 & 3.3 & 3.5 & 6.8 \\
Yu-Orlicz & 0.0 & 0.0 & 0.0 & 0.0 & 0.0 & 0.0 & 0.0 & 0.0 & 0.0 \\
Yu-CBB, $b=5$ & 48.4 & 42.9 & 85.2 & 38.1 & 31.0 & 87.2 & 31.6 & 27.0 & 88.5 \\
Yu-CBB, $b=8$ & 53.5 & 49.2 & 85.6 & 41.5 & 34.8 & 88.7 & 30.9 & 25.3 & 89.9 \\
\bottomrule
\end{tabular}

\caption{\footnotesize Empirical null rejection rates (\%) for
(a) independent errors and (b) AR(1) errors with lag-one correlation $1/3$.
All innovation distributions have variance one; $t_3$ is scaled by
$1/\sqrt{3}$. Here, $b$ denotes the batch or bootstrap block length.
The monitoring horizon is $T=20N$, the nominal level is $5\%$,
and each setting uses $1000$ simulation runs.}
\label{tab:null}
\end{table}

Under dependence and with normal or uniform innovations, TC is oversized for $N=50$, but its level approximation improves quickly for the larger training samples. For ACTC, batch length $8$ gives more conservative results than batch length $5$. C, PC, and FC are generally close to or below the nominal level, while WC and MM tend to be undersized. For heavy-tailed data, TC is severely oversized (as predicted by our theory), whereas ACTC, which is robust w.r.t. heavy tails, holds the nominal level. WC also displays increasing overrejection as $N$ grows. 

Yu-Orlicz is highly conservative, with no null rejections in any of the settings considered. The bootstrap versions show the opposite behavior. Yu-ECDF is oversized even for independent normal data, and the circular block bootstrap does not provide an accurate level approximation under dependence.

\textbf{Power and detection times:} We next study the performance of the different procedures in the presence of a change ($H_1$). Here, we set $N=100$, $T=2000$, and consider change point locations $k^\star\in\{4N,10N,16N\}$. The size of the change is fixed at $\Delta=0.5$. For brevity, we present results only for normally distributed errors and innovations, considering both independence and AR(1) dependence.

The results are displayed in Figure~\ref{fig:power:05} as cumulative rejection probabilities. Denoting the first alarm in simulation run $r$ by $\widehat k_r$, the curves show
\[
    \widehat p(k)=\frac{1}{1000}\sum_{r=1}^{1000}1\{\widehat k_r\leq k\},
    \qquad k=0,\ldots,T.
\]
The horizontal axis counts observations from the beginning of monitoring, and vertical lines indicate the change point. Thus, the curves display both the rejection probability and how quickly it is attained. Alarms before the change are included, so the pre-change part of each curve also shows the accumulation of false alarms.

For independent data, TC and ACTC show a rapid increase in rejection probability after the change and attain rejection rates close to one for all three change locations. The difference from C, PC, FC, WC, and MM becomes particularly pronounced for later changes. For $k^\star=16N$, TC and ACTC reject in $100\%$ and $99.9\%$ of runs, respectively, while the rejection rates of these five competitors are at most $8\%$.

\begin{figure}[h]
    \centering
    \includegraphics[width=\textwidth]{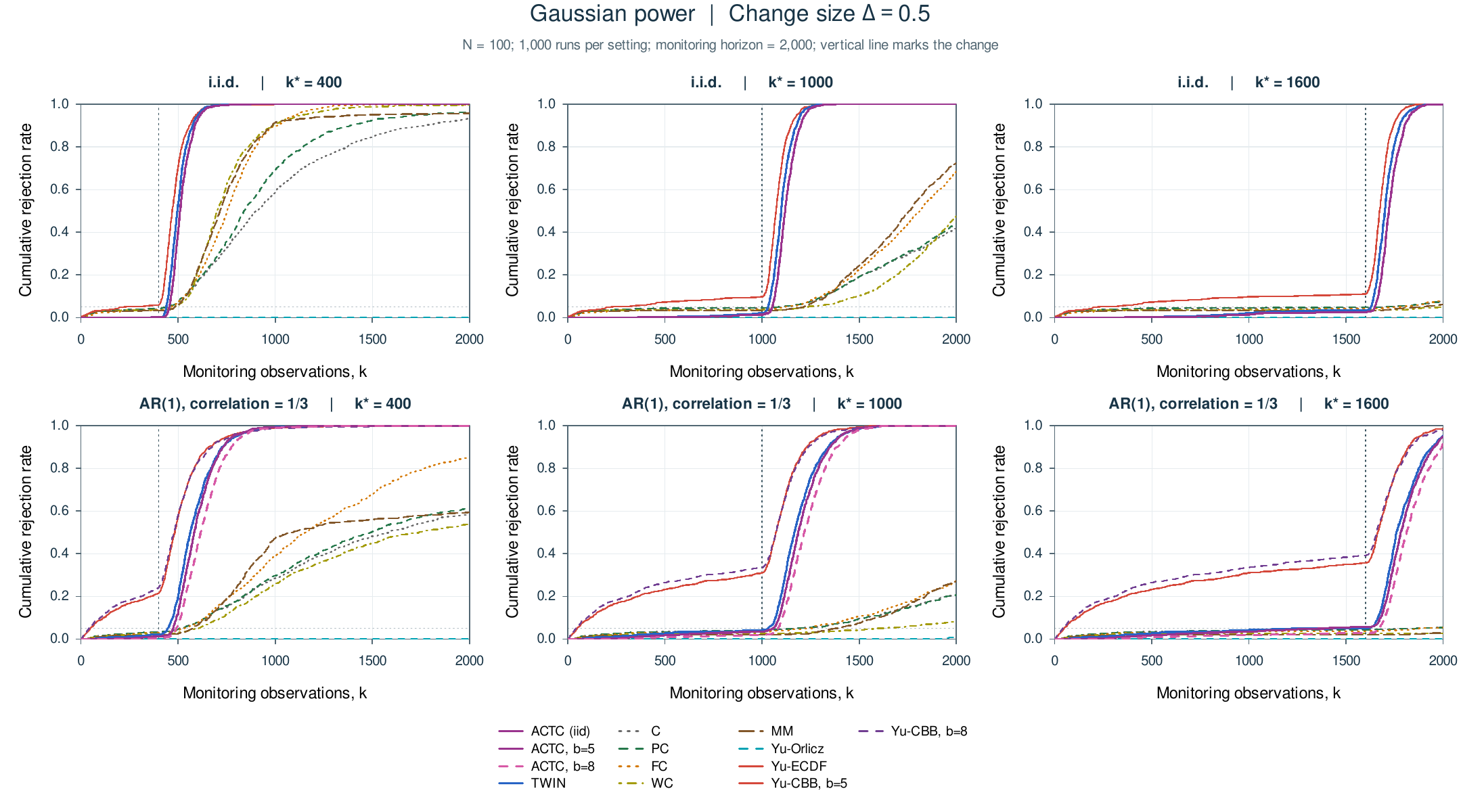}
    \caption{Cumulative rejection probabilities for normally distributed errors and innovations. The training period has length $N=100$, the monitoring period has length $T=2000$, and the change magnitude is $\Delta=0.5$. Columns correspond to $k^\star=4N$, $10N$, and $16N$; rows to independent and AR(1) errors with lag-one correlation $1/3$. Vertical lines indicate the change point.}
    \label{fig:power:05}
\end{figure}

A similar picture emerges under dependence, although the increase in rejection probability is slower. For $k^\star=16N$, TC attains a rejection rate of $95.4\%$, while ACTC attains $94.9\%$ and $91.4\%$ for batch lengths $5$ and $8$. The rejection rates of C, PC, FC, WC, and MM are at most $5.6\%$ in this setting. 

The bootstrap versions of Yu's method also attain high cumulative rejection rates and their curves rise earlier than those of TC and ACTC, but they suffer from anti-conservativeness, particularly in the dependent case. Yu-Orlicz, on the other hand, is practically powerless against the changes considered here.

\textbf{Size-adjusted comparison:} Finally, we compare TC and Yu's method after size adjustment for independent standard normal errors, with $N=100$, $T=2000$, $k^\star\in\{4N,10N\}$, and $\Delta\in\{0.25,0.5\}$. TC retains the normalization above; for Yu, we use $D_N(k)$ from Appendix~\ref{sec:yu-implementation}. For each method, the critical value is the empirical $0.95$ quantile of its maximum statistic over the monitoring horizon, based on $50{,}000$ Gaussian null trajectories. Table~\ref{tab:size-adjusted-gaussian} reports results from $10{,}000$ runs per setting. For $\Delta=0.25$, TC attains higher rejection probabilities and shorter median delays at both change locations. For $\Delta=0.5$, both methods reject in all runs, while Yu's method detects the change somewhat earlier.

\begin{table}[htbp]
    \centering
    \small
    \caption{Size-adjusted comparison for independent Gaussian errors. Rejection probabilities include alarms before the change. Median delays $\widehat k-k^\star$ are computed only over runs with $k^\star\leq\widehat k\leq T$, excluding pre-change alarms and non-detections.}
    \label{tab:size-adjusted-gaussian}
    \begin{tabular}{ccrrrr}
        \toprule
        &&\multicolumn{2}{c}{Rejection probability (\%)}&\multicolumn{2}{c}{Median delay}\\
        \cmidrule(lr){3-4}\cmidrule(lr){5-6}
        $\Delta$&$k^\star$&TC&Yu&TC&Yu\\
        \midrule
        $0.25$&$4N$&96.05&72.43&314.5&386.5\\
        $0.25$&$10N$&98.50&90.27&332&367\\
        $0.50$&$4N$&100.00&100.00&96&78\\
        $0.50$&$10N$&100.00&100.00&98&81\\
        \bottomrule
    \end{tabular}
\end{table}

Our simulation study thus shows that TC and ACTC combine high rejection probabilities with a rapid response to changes, especially when the change occurs late in the monitoring period. ACTC additionally provides substantially more stable level behavior for heavy-tailed data. These findings are consistent with our theoretical results on detection power and delay.

\vspace{-0.3cm}
\section*{Acknowledgments} 
This work was partially
funded by AUFF grants 47331, 47222 and  by the
Deutsche Forschungsgemeinschaft (DFG);  TRR 391 Spatio-temporal Statistics for the Transition of Energy and Transport (520388526). We are also grateful for the use of a computer cluster funded by the DFG, under Germany's Excellence Strategy - EXC 2092 CASA - 390781972. We thank Professor Jaime Briseño for sharing the COVID data set with us.

\section*{AI use declaration}
We used LLMs in the course of this project, mainly in its revision. Use has included code-writing, code-debugging, text editing, spell-checking, interpretation of finite sample results, proof-reading, and proof editing. We did not use LLMs to write the original proofs. We did not use LLMs for methodological development.  We reviewed the AI-assisted material and take responsibility for the final manuscript and code.
\vspace{-0.4cm}
{\footnotesize \putbib}
\end{bibunit}

\newpage

\begin{bibunit}
\begin{appendices}
    \begin{center}
  {\Large \textbf{Supplement to ``TWIN: Two window inspection for online change point detection''}} \\[1em]
  \textbf{Patrick Bastian}$^{1}$, 
  \textbf{Tim Kutta}$^{1}$, 
  
  $^{1}$\textit{Department of Mathematics, Aarhus University, Denmark} \\[1em]
  \texttt{patrick.bastian@math.au.dk}, 
  \texttt{tim.kutta@math.au.dk}
\end{center}

\section{Proofs}
\label{sec:proofs:standard}
This section treats $\beta>1/2$. The critical case is proved for Banach-valued data in Supplement~\ref{sec:critical-beta}, which also covers the scalar case.

 Our methodology relies on several new and powerful probabilistic results pertaining to partial sum processes on an unbounded domain. The most notable of these is the derivation of bounds for the global modulus of continuity. Our results apply in both independent and dependent settings, relying only on the assumption of uniformly subgaussian increments of the process, as shown in Lemma \ref{lem:finite:moc:proofs}.
Using classical isoperimetric inequalities and a generalized Garsia-Rodemich-Rumsey Lemma, we also provide extensions to sequential empirical processes. These new results are significant in the context of open-end multiscale monitoring, as they allow the decomposition of the test statistic into two parts. The first part consists of moderate scales over a finite (but growing) time horizon, where the test statistic essentially behaves like the maximum of a compactly supported Gaussian process. This part can be analyzed through non-trivial applications of existing strong approximations.
The second part includes all scales where Gaussian approximations fail -i.e., extremely short and extremely long scales. While these scales cannot be handled with previous Gaussian approximations, they can be uniformly controlled using our new global modulus of continuity. A key challenge is the (almost) maximal weighting of these scales, which is required to optimize detection delays and statistical power.

\textbf{Notational conventions} In Section~\ref{sec:2}, we have followed standard conventions in sequential change point testing, in that we have defined detectors such as $\widehat \Gamma^{TC}(k)$ using the data $X_1,...,X_{N+k}$ (the $k$th time point in the monitoring period). While this notation is generally reader-friendly it is fairly burdensome in our proofs and therefore, we use a different, more mathematically fruitful notation here. In the proofs, as in Section~\ref{ac:sec}, $k$ refers to the current time, where data $X_1,...,X_{k}$ are available and the potential change point is located at the observation $X_{k^\star}$, i.e. $k^\star>N$.  In particular, we write for $k>N$
\begin{align*}
    \widehat \Gamma^{TC}(k)=\max_{\substack{1 \leq \ell \leq k/2 \\ k-\ell \geq N}}w^{TC,\beta}(\ell,k)\hat \gamma^{TC}(\ell,k)
\end{align*}
where
\begin{align*}
    \hat \gamma^{TC}(\ell,k)=\Big|\min(1, \ell/N) \cdot S_{\max(\ell,N)}-(S_k-S_{k-\ell})\Big|\\
     w^{TC, \beta}(\ell, k)  :=\frac{\ell^{-1/2}}{\log^\beta(C_0+N/\ell)+\log^\beta(C_0+k/N)}.
\end{align*}
The advantage of this approach is that it allows natural rescalings between discrete and continous time (where $t=k/N$ and $s=\ell/N$) making expressions more concise. In null proofs for mean changes we set the common mean to zero, since it cancels in the contrasts.

\subsection{Proofs for TC}
We recall the definition of the partial sum process $P_N$ in \eqref{eq:def:PN} and further denote, for $\beta>1/2$, the additive normalization
\[
\rho_\beta(x,t)= x^{1/2}\big\{\log^\beta(C_0+1/x)+\log^\beta(C_0+t)\big\}.
\]

\textbf{Modulus of continuity results}

The following result was formulated by \cite{chevallier2023} for the Brownian motion (which satisfies \eqref{eq:subg:condition}), and is generalized here to processes with subgaussian increments.

\begin{lem}
\label{lem:finite:moc:proofs}
    Let $(G(s))_{s \in [0,\infty)}$ be a continuous stochastic process satisfying
    \begin{align}
    \label{eq:subg:condition}
        \sup_{s,t\geq 0}\frac{\|G(s)-G(t)\|_{\Psi_2}}{\sqrt{|t-s|}}\leq K~.
    \end{align}
    Further let $\delta>0$ be arbitrary and define the random variable
    \begin{equation*}
    M_{G} := \sup_{\substack{0< s<t<\infty, \\ h=|t-s|}} \frac{|G(t)-G(s)|}{\sqrt{h \left(1+ \log(\frac{t}{h}) + \delta |\log (t)| \right)}}.
    \end{equation*}
    Then, there exists $\lambda_0$ such that
    \[
    \mathbb E\exp(\lambda M_G^2)<\infty
    \quad\text{for }0<\lambda\le\lambda_0,
    \]
    and in particular they are bounded by some constant depending only on $K$. We also note that \eqref{eq:subg:condition} is always satisfied for partial sum processes of random variables that are independent and uniformly subgaussian (see Proposition 2.7.1 in \cite{vershynin:2018}).
\end{lem}

\begin{proof}
    We essentially follow the proof of \cite{chevallier2023} for the case of the Brownian motion, clarifying some very minor issues along the way. We restrict ourselves to $0<\delta<1$ which also implies the result for all $\delta\geq 1$ by monotonicity. \\

    We will apply Lemma 1 from \cite{chevallier2023} with the choices
    \begin{align}
    \nonumber
        \Psi(x)&:=e^{x^2/2}-1,\\
        \mu(x)&:=K\sqrt{cx},\label{eq:def:mu}
    \end{align}
    where $c>1$ is some constant. The inverse of $\Psi$ is given by
    \[
        \Psi^{-1}(y)=\sqrt{2\log(y+1)},
    \]
    and the Lebesgue-Stieltjes measure $d\mu$ associated to $\mu$ satisfies
    \[
        d\mu(x)=\frac{K\sqrt{c}}{2\sqrt{x}}~.
    \]
    For a meaningful application of Lemma 1 from \cite{chevallier2023} we need to find a bound for 
    \[
        \int_0^t\int_0^t \Psi\left(\frac{|G(x)-G(y)}{\mu(|x-y|)}\right)dxdy~.
    \]
    To that end define the random variable
    \begin{align}
        \xi_T:=f_\delta(T)\int_0^T\int_0^T\Psi\left(\frac{|G(x)-G(y)}{\mu(|x-y|)}\right)dxdy,
    \end{align}
    where 
    \begin{align*}
        f_\delta(T):=\begin{cases}
            (1/T+1)^{2(1-\delta)} &\quad \text{if }T<1\\
            1 &\quad \text{if } T=1\\
            (T-1)^{-2(1+\delta)} &\quad \text{if } T>1
        \end{cases}
    \end{align*}
    diverges for $T \downarrow 0$ and vanishes for $T\uparrow \infty$. By some elementary calculations the following claims can be verified:
    \begin{itemize}
        \item[(1)] If $t \leq 1$ we have $\frac{1}{\lfloor 1/t \rfloor+1}<t\leq \frac{1}{\lfloor 1/t \rfloor}$ and $f_\delta(\frac{1}{\lfloor 1/t  \rfloor})\geq t^{-2(1-\delta)}$.
        \item[(2)] If $t\geq 1$ we have $\lceil t-1 \rceil<t\leq \lceil t \rceil $ and $f_\delta(\lceil t \rceil)\geq t^{-2(1+\delta)}$.
    \end{itemize}
   Next, we define the variable
    \[
        \xi=\sup\{\xi_T,\xi_{1/T}| T\in \N\}.
    \]
    Again by construction we obtain, for $t\geq 1$, that
    \[
        \int_0^t\int_0^t \Psi\left(\frac{|G(x)-G(y)}{\mu(|x-y|)}\right)dxdy\leq (f_\delta(\lceil t \rceil))^{-1}\xi_t\leq t^{2(1+\delta)}\xi
    \]
    and, for $t\leq 1$, that
    \[
        \int_0^t\int_0^t \Psi\left(\frac{|G(x)-G(y)}{\mu(|x-y|)}\right)dxdy\leq (f_\delta(1/\lfloor 1/t\rfloor^{-1}))^{-1}\xi_{1/\lfloor 1/t\rfloor}\leq t^{2(1-\delta)}\xi~.
    \]
    We then may apply Lemma 1 from \cite{chevallier2023} to obtain
    \begin{align*}
    |G(s)-G(t)|\leq 
        \begin{cases}
            8\int_0^{|s-t|}\Psi^{-1}(\frac{4t^{2(1+\delta)}\xi}{u^2})d\mu(u) \quad \text{if } t\geq 1\\
             8\int_0^{|s-t|}\Psi^{-1}(\frac{4t^{2(1-\delta)}\xi}{u^2})d\mu(u) \quad \text{if } t<1~.
        \end{cases}
    \end{align*}
    From here on, one may follow the arguments in the proof of \cite{chevallier2023} to arrive at
    \[
        M_G^2\leq K^2128c(\log(4\xi+1)+32))
    \]
    which yields the desired result by Lemma \ref{lem:finite:moment:xi}. To be precise we have
    \[
        \E[\exp(\lambda M^2)]\leq \exp(K^24096c\lambda)\E[(4\xi+1)^{128K^2c\lambda}],
    \]
    so that choosing a $\lambda$ satisfying
    \[
        1<128K^2c\lambda<c
    \]
    finishes the proof.

\end{proof}

\begin{lem}\label{lem:finite:moment:xi}
Recall $c$ as the constant in \eqref{eq:def:mu} and suppose that $c\delta>1$. Then, for any $p\in (1,c)$ we have $\mathbb{E}[\xi^{p}]< +\infty$. 
\end{lem}
\begin{proof}
    Let $q \in (1,c)$ and $p \in (1,q)$. For $T\geq 1$ we may apply Jensen's inequality to obtain
    \begin{align}
      \mathbb{E}\left[\xi_{T}^{q}\right]
     =& 
   f^{q}_{\delta}(T) \, \mathbb{E}\left[\left(\int_{0}^{T} \int_{0}^{T} \exp\left(\frac{|G(t)-G(s)|^{2}}{2cK^2|t-s|}\right) dsdt\right)^{q}\right]
    \\
     \leq&  f^{q}_{\delta}(T)\,  T^{2(q-1)} \mathbb{E}\left[\int_{0}^{T} \int_{0}^{T} \exp\left(\frac{|G(t)-G(s)|^{2}}{2cK^2|t-s|}\right)^{q} dsdt\right]
    \\
      = &\frac{T^{2q}}{(T-1)^{(2+2\delta)q}} \, \int_{0}^{T}  T^{-2}\int_{0}^{T} \mathbb{E}\left[\exp\left(\frac{q}{2c}\left(\frac{|G(t)-G(s)|}{K\sqrt{|t-s|}}\right)^{2}\right)\right] dsdt. \label{e:int}  
    \end{align}

By assumption we have 
\[
    \|G(s)-G(t)\|_{\Psi_2}\leq K\sqrt{|t-s|}
\]
as well as $q<c<2c$, so that by the definition of $\|\cdot\|_{\Psi_2}$ we obtain
\[
    \mathbb{E}\left[\exp\left(\frac{q}{2c}\left(\frac{|G(t)-G(s)|}{K\sqrt{|t-s|}}\right)^{2}\right)\right]\leq 2.
\]
Consequently, the double integral in \eqref{e:int} is bounded by 2, and we obtain from the chain of inequalities leading up to it that
\[
    \E[\xi^q_T]\leq \frac{2T^{2q}}{(T-1)^{(2+2\delta)q}}.
\]
Similar arguments for the case $T<1$ establish that 
\[
    \E[\xi^q_{1/T}]\leq \frac{2(T+1)^{(2-2\delta)q}}{T^{2q}}~.
\]
Defining
\[
    g(T):=\frac{2T^{2q}}{(T-1)^{(2+2\delta)q}}+\frac{2(T+1)^{(2-2\delta)q}}{T^{2q}}
\]
we observe that 
\begin{align}
\label{e:g:bound}
    g(T)\lesssim T^{-2\delta q}
\end{align}
and then apply Markov's inequality to obtain
\begin{align*}
    \PR(\max(\xi^p_T,\xi^p_{1/T})>n)\lesssim g(T)n^{-q/p}.
\end{align*}
The union bound in combination with summability of $g(T)$ over $T$ (choose $q$ close to $c$ and observe that $c\delta>1$ by assumption) then yields
\[
    \PR(\xi^p>n)\lesssim \sum_{T=1}^\infty g(T)n^{-q/p}\lesssim n^{-q/p},
\]
which in turn implies
\[
    \E[\xi^p]\leq \sum_{n=0}^\infty \PR(\xi^p>n)\lesssim \sum_{n=0}^\infty n^{-q/p}<\infty
\]
as desired.
\end{proof}

\textbf{Derivation of distributional limit}

The next lemma provides a technical tool to show that in our monitoring, "late times" (large $t$) do not contribute to the limiting distribution under the null hypothesis.

\begin{lem}
\label{lem:cutoff} 
    Let $B$ be a centered Brownian motion and $P_N$ be the partial sum process of model errors. Suppose that $P_N$ satisfies 
    \[
        \sup_{N \geq 1}\sup_{s,t \geq 0}\frac{\|P_N(s)-P_N(t)\|_{\Psi_2}}{\sqrt{|t-s|}}\leq K~.
    \] Then it holds, for any $\eta>0$, that
    \begin{align*}
        \sup_{N^{\eta} \leq t}\sup_{\substack{0 \leq s \leq t/2\\ t-s \geq 1 }}\frac{|\min(1,s)B(\max(1,s))-(B(t)-B(t-s))|}{\rho_\beta(s,t)}=o_{\PR}(1),            
    \end{align*}
    as well as
    \begin{align*}
       \sup_{N^\eta \leq t}\sup_{\substack{0 \leq s \leq t/2\\ t-s \geq 1 }}\frac{|\min(1,s)P_N(\max(1,s))-(P_N(t)-P_N(t-s))|}{\rho_\beta(s,t)}=        o_{\PR}(1).
    \end{align*}    
\end{lem}
\begin{proof}
    We use $M_{P_N}$ from Lemma \ref{lem:finite:moc:proofs} with $\delta=1$.
    For $t\geq1$ and $0<s\leq t/2$, we have
    \[
        1+\log(t/s)+\log t
        =1+2\log t-\log s
        \lesssim \log(C_0+1/s)+\log(C_0+t),
    \]
    since $\log(C_0+1/s),\log(C_0+t)\geq\log C_0>0$,
    $\log t\leq\log(C_0+t)$, and
    $-\log s\leq\log(C_0+1/s)$. Therefore
    \begin{align}
    \label{e:increment:bound}
        \frac{|P_N(t)-P_N(t-s)|}{\rho_\beta(s,t)}
        &\leq M_{P_N}
        \frac{\sqrt{1+\log(t/s)+\log t}}
        {\log^\beta(C_0+1/s)+\log^\beta(C_0+t)}
        \nonumber\\
        &\lesssim M_{P_N}
        \bigl\{\log(C_0+1/s)+\log(C_0+t)\bigr\}^{1/2-\beta},
    \end{align}
    where we used
    \[
        (x+y)^\beta
        \leq 2^{\max(\beta-1,0)}(x^\beta+y^\beta),
        \qquad x,y>0.
    \]
    For the initial-window term, continuity at zero gives
    $|P_N(u)|\leq M_{P_N}\sqrt{u(1+|\log u|)}$. Consequently, for $s\leq1$,
    \[
        \frac{|sP_N(1)|}{\rho_\beta(s,t)}
        \leq M_{P_N}
        \frac{\sqrt{s}}
        {\log^\beta(C_0+1/s)+\log^\beta(C_0+t)}
        \lesssim
        M_{P_N}
        \bigl\{\log(C_0+1/s)+\log(C_0+t)\bigr\}^{1/2-\beta}.
    \]
    For $1<s\leq t/2$, we instead obtain
    \[
        \frac{|P_N(s)|}{\rho_\beta(s,t)}
        \leq M_{P_N}
        \frac{\sqrt{1+\log s}}
        {\log^\beta(C_0+1/s)+\log^\beta(C_0+t)}
        \lesssim
        M_{P_N}
        \bigl\{\log(C_0+1/s)+\log(C_0+t)\bigr\}^{1/2-\beta},
    \]
    because $\log s\leq\log t\leq\log(C_0+t)$. Combining these bounds yields
    \begin{align*}
        &\sup_{t\geq N^\eta}\sup_{\substack{0<s\leq t/2\\t-s\geq1}}
        \frac{|\min(1,s)P_N(\max(1,s))-(P_N(t)-P_N(t-s))|}{\rho_\beta(s,t)}
        \\
        &\quad\lesssim M_{P_N}\log^{1/2-\beta}(C_0+N^\eta)=o_{\PR}(1),
    \end{align*}
    since $M_{P_N}=O_{\PR}(1)$ and $\beta>1/2$. The same argument works for Brownian motion.
    Notice also that restricting to $s\leq\nu$ instead gives the bound
    $M_{P_N}\log^{1/2-\beta}(C_0+1/\nu)$, which vanishes as $\nu\downarrow0$.
\end{proof}

In our test statistic we start with a time series of observations made at discrete times $k=1,2,...$ . Later in the proofs we need to transition to a rescaled continuous time to describe the limiting distributions. The next lemma is a transition tool from discrete to continuous time.

\begin{lem}
\label{lem:cont:discrete}
Suppose that $P_N$ satisfies 
    \[
        \sup_{N \geq 1}\sup_{s,t \geq 0}\frac{\|P_N(s)-P_N(t)\|_{\Psi_2}}{\sqrt{|t-s|}}\leq K~.
    \]     
    Then we have, for any $\eta>0$, that
    \begin{align*}
           &\sup_{1 \leq t \leq N^{\eta}}\sup_{ \substack{ 0<s\leq t/2 \\ t-s\geq 1}}\frac{\Big|\min(1,s)P_N(\max(1,s))-(P_N(t)-P_N(t-s))\Big|}{\rho_\beta(s,t)}\\       
        =&\sup_{N \leq k \leq N^{1+\eta}}\sup_{ \substack{1 \leq \ell \leq k/2 \\ k-\ell \geq N}} \frac{\big|\min(1,\ell/N)S_{\max(N,\ell)}-(S_k-S_{k-\ell}) \big|}{\sqrt{N}\rho_\beta(\ell/N,k/N)}+o_{\PR}(1).
    \end{align*}   
 Moreover, with $B$ a centered Brownian motion and $ \tilde B$ its rescaled version  $ \tilde B(x) = N^{-1/2}B(xN)$ we have
    \begin{align*}           
        &\sup_{1 \leq t \leq N^{\eta}}\sup_{ \substack{ 0<s\leq t/2 \\ t-s\geq 1}}\frac{\Big|\min(1,s)\tilde B(\max(1,s))-(\tilde B(t)-\tilde B(t-s))\Big|}{\rho_\beta(s,t)}+o_{\PR}(1)   \\
        =&\sup_{1 \leq k \leq N^{1+\eta}}\sup_{ \substack{1 \leq \ell \leq k/2 \\ k-\ell \geq N}} \frac{\big|\min(1,\ell/N)B(\max(N,\ell))-(B(k)-B(k-\ell)) \big|}{\sqrt{N}\rho_\beta(\ell/N,k/N)}+o_{\PR}(1)~.
    \end{align*}    
\end{lem}
\begin{proof}
    We only show the statement for the partial sum process. The proof for Brownian motion follows by the same arguments. Write
    \[
        Q_N(s,t)=\min(1,s)P_N(\max(1,s))-(P_N(t)-P_N(t-s)).
    \]
    We compare admissible pairs $(s,t),(s',t')$ with $|s-s'|,|t-t'|\leq2/N$ and $t,t'\leq N^\eta$, treating small and large windows separately.

    \textbf{Case 1: $\min(s,s')\leq\nu$}\\
    For $N\nu\to\infty$, both window sizes are at most $2\nu$. The bounds in the proof of Lemma \ref{lem:cutoff} give
    \[
        \sup_{\substack{t\geq1,\ 0<s\leq2\nu\\s\leq t/2,\ t-s\geq1}}
        \frac{|Q_N(s,t)|}{\rho_\beta(s,t)}
        \lesssim M_{P_N}\log^{1/2-\beta}(C_0+1/(2\nu))=o_{\PR}(1),
    \]
    whenever $\nu\downarrow0$.\\

    \textbf{Case 2: $s,s'>\nu$}\\
    Here we use the decomposition
    \begin{align*}
        &\left|\frac{|Q_N(s,t)|}{\rho_\beta(s,t)}-
        \frac{|Q_N(s',t')|}{\rho_\beta(s',t')}\right|\\
        &\quad\leq\frac{|Q_N(s,t)-Q_N(s',t')|}{\rho_\beta(s,t)}
        +\frac{|Q_N(s',t')|}{\rho_\beta(s',t')}
        \left|\frac{\rho_\beta(s',t')}{\rho_\beta(s,t)}-1\right|
        =:J_1+J_2.
    \end{align*}

    For $J_1$, let $T_N=N^\eta+2$. Lemma \ref{lem:finite:moc:proofs}, with $\delta=1$, gives
    \[
    \begin{aligned}
        \omega_N&:=\sup_{\substack{0\leq u<v\leq T_N\\v-u\leq4/N}}|P_N(v)-P_N(u)|\\
        &\leq M_{P_N}\sup_{0<h\leq4/N}
        \sqrt{h\{1+2\log T_N+\log(1/h)\}}\\
        &\lesssim M_{P_N}\sqrt{N^{-1}\log N}.
    \end{aligned}
    \]
    Here $\log(v/h)+|\log v|\leq2\log T_N+\log(1/h)$. Put $c(s)=\min(1,s)$ and $v(s)=\max(1,s)$. Since these maps are Lipschitz and $c(s)\neq c(s')$ implies $v(s')\leq1+2/N$,
    \[
    \begin{aligned}
        |Q_N(s,t)-Q_N(s',t')|
        &\leq |P_N(v(s))-P_N(v(s'))|\\
        &\quad+|c(s)-c(s')|\,|P_N(v(s'))|+2\omega_N\\
        &\leq3\omega_N+\frac{2}{N}(|P_N(1)|+\omega_N),\\
        J_1&\lesssim\nu^{-1/2}\!\left\{3\omega_N+
        \frac{2}{N}(|P_N(1)|+\omega_N)\right\}\\
        &=O_{\PR}\!\left(\sqrt{\frac{\log N}{N\nu}}\right).
    \end{aligned}
    \]

    For $J_2$, since
    $\log(C_0+1/s),\log(C_0+t)\geq\log C_0$, we obtain
    \[
    \begin{aligned}
        |\partial_s\log\rho_\beta(s,t)|
        &=
        \left|
        \frac{1}{2s}
        -
        \frac{\beta\log^{\beta-1}(C_0+1/s)}
        {s(C_0s+1)\{\log^\beta(C_0+1/s)+\log^\beta(C_0+t)\}}
        \right|\\
        &\leq\frac{1}{2s}+\frac{\beta}{s\log C_0}
        \lesssim s^{-1},\\
        |\partial_t\log\rho_\beta(s,t)|
        &=
        \frac{\beta\log^{\beta-1}(C_0+t)}
        {(C_0+t)\{\log^\beta(C_0+1/s)+\log^\beta(C_0+t)\}}\\
        &\leq\frac{\beta}{(C_0+t)\log C_0}
        \lesssim1.
    \end{aligned}
    \]
    Put $s_\theta=s+\theta(s'-s)$ and
    $t_\theta=t+\theta(t'-t)$, $0\leq\theta\leq1$. Since
    $s_\theta>\nu$, integrating the preceding derivative bounds gives
    \[
    \begin{aligned}
    \left|\log\frac{\rho_\beta(s',t')}{\rho_\beta(s,t)}\right|
    &=\left|\int_0^1\left\{
    (s'-s)\partial_s\log\rho_\beta(s_\theta,t_\theta)
    +(t'-t)\partial_t\log\rho_\beta(s_\theta,t_\theta)
    \right\}d\theta\right|\\
    &\lesssim \frac{|s-s'|}{\nu}+|t-t'|
    \lesssim (N\nu)^{-1}.
    \end{aligned}
    \]
    Using $|e^x-1|\leq e^{|x|}|x|$, we therefore obtain
    \[
        \left|\frac{\rho_\beta(s',t')}{\rho_\beta(s,t)}-1\right|
        =O((N\nu)^{-1}).
    \]
    Moreover, \eqref{e:increment:bound} and the initial-window bounds
    in the proof of Lemma \ref{lem:cutoff} give
    \[
        \frac{|Q_N(s',t')|}{\rho_\beta(s',t')}
        =O_{\PR}(1)
    \]
    uniformly. Hence
    \[
        J_2=O_{\PR}((N\nu)^{-1}).
    \]
    Choosing $\nu=1/\log N$ makes both terms vanish.

    Finally, take $k=\lfloor Nt\rfloor$ and
    $\ell=\min(\lfloor Ns\rfloor,\lfloor k/2\rfloor,k-N)$.
    Outside the small-window region this is an admissible grid pair, with $|t-k/N|\leq1/N$ and $|s-\ell/N|\leq2/N$. The preceding bounds and $P_N(j/N)=N^{-1/2}S_j$ give the desired result.
\end{proof}

\begin{theo}
\label{theo:limit}
    Assume that $(\varepsilon_i)_{i \in \N}$ are iid and satisfy $\|\varepsilon_i\|_{\Psi_2}\leq K$.
    When $H_0$ holds, let $B$ be a Brownian motion with variance $\sigma^2$. Then
    \begin{align}        
        &\sup_{k\geq N}\sup_{ \substack{1 \leq \ell \leq k/2 \\ k-\ell \geq N}}\frac{\Big|\min(1,\frac{\ell}{N})S_{\max(N,\ell)}-(S_k-S_{k-\ell})\Big|}{\sqrt{N}\rho_\beta(\ell/N,k/N)}\\        
        \overset{d}{\to}&\sup_{t \geq 1}\sup_{ \substack{ 0<s\leq t/2 \\ t-s\geq 1}}\frac{\Big|\min(1,s)B(\max(1,s))-(B(t)- B(t-s))\Big|}{\rho_\beta(s,t)}  .  
    \end{align}
\end{theo}
\begin{proof}
    By \cite{kmt:1976} we can find iid normal random variables $(Z_i)_i$ with variance $\sigma^2$ such that a.s.
    \[
        \Big|\sum_{i=1}^j\varepsilon_i-\sum_{i=1}^jZ_i\Big|\lesssim \log(j).
    \]
    By Lemma \ref{lem:construct:brownian} we may therefore even  find a Brownian motion $W$ with variance $\sigma^2$ such that
     \[
        \Big|\sum_{i=1}^j\varepsilon_i-W(j)\Big|\lesssim \log(j).
    \]
    holds almost surely.\\
    \textbf{Case $\ell \leq N$:}    
    We begin by decomposing
    \begin{align*}
        &\Bigg|\frac{|\frac{\ell}{N}S_N-(S_k-S_{k-\ell})|}{\sqrt{N}\rho_\beta(\ell/N,k/N)}-\frac{|\frac{\ell}{N}W(N)-(W(k)-W(k-\ell))|}{\sqrt{N}\rho_\beta(\ell/N,k/N)}\Bigg|\\
        \leq &\frac{|\frac{\ell}{N}(S_N-W(N))|}{\sqrt{N}\rho_\beta(\ell/N,k/N)}+\frac{|(S_k-S_{k-\ell})-(W(k)-W(k-\ell))|}{\sqrt{N}\rho_\beta(\ell/N,k/N)}\\
        :=& R_1(k,\ell)+R_2(k,\ell).
    \end{align*}
    First we consider $R_1$ and note that
    \begin{align*}
        \sup_{k \geq N}\sup_{1 \leq \ell \leq N}R_1(k,\ell)\lesssim  \sup_{k \geq N}\sup_{1 \leq \ell \leq N}\frac{\sqrt{\frac{\ell}{N}}\log(N)}{\sqrt{N}\big\{\log^\beta(C_0+N/\ell)+\log^\beta(C_0+k/N)\big\}}=o(1).
    \end{align*}

    In the following we let $\eta>0$ be an arbitrary but fixed constant. For $R_2$ we estimate
    \begin{align*}
       & \sup_{N \leq k \leq N^{1+\eta}}\sup_{1 \leq \ell \leq N}R_2(k,\ell)\\
        \lesssim  &\sup_{N \leq k \leq N^{1+\eta}}\sup_{1 \leq \ell \leq N^{1/2}}R_{2}(k,\ell)+ \sup_{N \leq k \leq N^{1+\eta}}\sup_{N^{1/2} \leq \ell \leq N}\frac{\log(k)}{N^{1/4}\log^\beta(C_0+k/N)}\\
        =&\sup_{N \leq k \leq N^{1+\eta}}\sup_{1 \leq \ell \leq N^{1/2}}R_{2}(k,\ell)+o(1),
    \end{align*}
    and then use Theorem 1 from \cite{chevallier2023} and Lemma \ref{lem:finite:moc:proofs} to obtain
    \begin{align*}
        &\sup_{N \leq k \leq N^{1+\eta}}\sup_{1 \leq  \ell \leq N^{1/2}}R_2(k,\ell)\\
        \lesssim  & (M_{P_N}+M_W)\sup_{N \leq k \leq N^{1+\eta}}\sup_{1 \leq \ell \leq N^{1/2}}\frac{\sqrt{\ell/N(1+\log(k/\ell)+\log(C_0+k/N)}}{\rho_\beta(\ell/N,k/N)}\\
        \lesssim &(M_{P_N}+M_W)\frac{\sqrt{\log(N)}}{\log^\beta(C_0+N^{1/2})}=o_{\PR}(1)~.
    \end{align*}
    Combining the previous inequalities we hence have
    \[
        \sup_{N \leq k \leq N^{1+\eta}}\sup_{1 \leq \ell \leq N}R_1(k,\ell)+R_2(k,\ell)=o_{\PR}(1)~,
    \]
    so that, using Lemma \ref{lem:cont:discrete} to switch between the discretized and continuous versions, we obtain
    \begin{align*}
           &\sup_{1 \leq t \leq N^{\eta}}\sup_{\substack{0<s\leq 1 \\ t-s \geq 1}}\frac{\Big|sP_N(1)-(P_N(t)-P_N(t-s))\Big|}{\rho_\beta(s,t)}\\       
        =&\sup_{1 \leq t \leq N^{\eta}}\sup_{\substack{0<s\leq 1 \\ t-s \geq 1}}\frac{\Big|s\tilde W(1)-(\tilde W(t)-\tilde W(t-s))\Big|}{\rho_\beta(s,t)}+o_{\PR}(1),   
    \end{align*}     
    where $\tilde W(t)=N^{-1/2}W(Nt)$.  \\
    \textbf{Case $\ell >N$:}
    Proceeding as in the previous case we need to bound
    \begin{align*}
        R_3(k,\ell)&=\frac{|S_\ell-W(\ell)|}{\sqrt{N}\rho_\beta(\ell/N,k/N)},\\
        R_4(k,\ell)&=\frac{|(S_k-S_{k-\ell})-(W(k)-W(k-\ell))|}{\sqrt{N}\rho_\beta(\ell/N,k/N)}
    \end{align*}
    uniformly over $N \leq k \leq N^{1+\eta}, N<\ell\leq k/2$. These quantities vanish by upper bounding the numerator by $\log(k)\simeq \log(N)$, yielding
    \begin{align*}
         &\sup_{1 \leq t \leq N^{\eta}}\sup_{\substack{1<s\leq t/2\\t-s\geq 1}}\frac{\Big|P_N(s)-(P_N(t)-P_N(t-s))\Big|}{\rho_\beta(s,t)}\\       
        =&\sup_{1 \leq t \leq N^{\eta}}\sup_{\substack{1<s\leq t/2\\t-s\geq 1}}\frac{\Big|\tilde W(s)-(\tilde W(t)-\tilde W(t-s))\Big|}{\rho_\beta(s,t)}+o(1).
    \end{align*}
    
    \textbf{Combining:} The two cases now jointly yield 
    \begin{align}
    \label{e:brownian:approx}
        &\sup_{1 \leq t \leq N^{\eta}}\sup_{ \substack{ 0<s\leq t/2 \\ t-s\geq 1 }}\frac{\Big|\min(1,s)P_N(\max(1,s))-(P_N(t)-P_N(t-s))\Big|}{\rho_\beta(s,t)}\\       
        =&\sup_{1 \leq t \leq N^{\eta}}\sup_{ \substack{ 0<s\leq t/2 \\ t-s\geq 1 }}\frac{\Big|\min(1,s)\tilde W(\max(1,s))-(\tilde W(t)-\tilde W(t-s))\Big|}{\rho_\beta(s,t)}+o(1).\\
    \end{align}    
    We hence obtain
    \begin{align*}
        &\sup_{k\geq N}\sup_{ \substack{1 \leq \ell \leq k/2 \\ k-\ell \geq N}}\frac{\Big|\min(1,\frac{\ell}{N})S_{\max(N,\ell)}-(S_k-S_{k-\ell})\Big|}{\sqrt{N}\rho_\beta(\ell/N,k/N)}\\
       =&\sup_{\substack{t \geq 1 \\ tN \in \N}}\sup_{\substack{0<s \leq t/2 \\ sN \in \mathbb{N} \\ t-s\geq 1}}\frac{\Big|\min(1,s)P_N(\max(1,s))-(P_N(t)-P_N(t-s))\Big|}{\rho_\beta(s,t)}+o(1)\\
        \overset{\ref{lem:cutoff}}{=}&\sup_{\substack{1 \leq t \leq N^{\eta}\\ tN \in \N}}\sup_{\substack{0<s\leq t/2\\ sN\in \N \\ t-s\geq 1}}\frac{\Big|\min(1,s)P_N(\max(1,s))-(P_N(t)-P_N(t-s))\Big|}{\rho_\beta(s,t)}+o_{\PR}(1)\\     
        \overset{\ref{lem:cont:discrete}}{=}&\sup_{1 \leq t \leq N^{\eta}}\sup_{ \substack{ 0<s\leq t/2 \\ t-s\geq 1}}\frac{\Big|\min(1,s)P_N(\max(1,s))-(P_N(t)-P_N(t-s))\Big|}{\rho_\beta(s,t)}+o_{\PR}(1)\\ 
        \overset{\ref{e:brownian:approx}}{=}&\sup_{1 \leq t \leq N^{\eta}}\sup_{ \substack{ 0<s\leq t/2 \\ t-s\geq 1}}\frac{\Big|\min(1,s)\tilde W(\max(1,s))-(\tilde W(t)-\tilde W(t-s))\Big|}{\rho_\beta(s,t)}+o_{\PR}(1).      
    \end{align*}
    Noting that $\tilde W\overset{d}{=}B$ for some Brownian motion $B$ with variance $\sigma^2$ we have that
    \begin{align*}
     &\sup_{1 \leq t \leq N^{\eta}}\sup_{ \substack{ 0<s\leq t/2 \\ t-s\geq 1}}\frac{\Big|\min(1,s)\tilde W(\max(1,s))-(\tilde W(t)-\tilde W(t-s))\Big|}{\rho_\beta(s,t)} \\
    \overset{d}{=}&\sup_{1 \leq t \leq N^{\eta}}\sup_{ \substack{ 0<s\leq t/2 \\ t-s\geq 1}}\frac{\Big|\min(1,s)B(\max(1,s))-(B(t)-B(t-s))\Big|}{\rho_\beta(s,t)}  \\
    \overset{\ref{lem:cutoff}}{=}&\sup_{t \geq 1}\sup_{ \substack{ 0<s\leq t/2 \\ t-s\geq 1}}\frac{\Big|\min(1,s)B(\max(1,s))-(B(t)-B(t-s))\Big|}{\rho_\beta(s,t)}+o_{\PR}(1)
    \end{align*}
    which yields the desired conclusion.

\end{proof}

\begin{lem}
    Suppose that the observations $(X_i)_{i \in \N}$ satisfy
    \begin{align*}
        X_i=\begin{cases}
            \varepsilon_i+\mu^{(1)} \quad 1 \leq i \leq k^\star\\
            \varepsilon_i+\mu^{(2)}\quad i>k^\star 
        \end{cases}
    \end{align*}
    and define $\Delta=|\mu^{(1)}-\mu^{(2)}|$. Then, for $1\leq\ell\leq k^\star$ and $k=k^\star+\ell$, if $\ell$ satisfies 
    \[
        \sqrt{\ell}\gg\big\{\log^\beta(C_0+N/\ell)+\log^\beta(C_0+k/N)\big\}\Delta^{-1}
    \]
    we have
    \[
        \PR(\widehat \Gamma^{TC}(k^\star+\ell)>q_{1-\alpha})=1+o(1).
    \]
\end{lem}
\begin{proof} 
    At $k=k^\star+\ell$, the first window is entirely before the change and the last window entirely after it. We obtain (using the uniform modulus bounds for the noise)
    \begin{align*}
         &\widehat \Gamma^{TC}(k^\star+\ell)\\
         \geq& \frac{\sqrt{\ell}\Delta}{\log^\beta(C_0+N/\ell)+\log^\beta(C_0+k/N)}\\
         &   -\sup_{1 \leq t}\sup_{\substack{0 \leq s \leq t/2\\ t-s \geq 1 }}\frac{|\min(1,s)P_N(\max(1,s))-(P_N(t)-P_N(t-s))|}{\rho_\beta(s,t)}\\
        =&\frac{\sqrt{\ell}\Delta}{\log^\beta(C_0+N/\ell)+\log^\beta(C_0+k/N)}+O_{\PR}(1).
    \end{align*}    
    Consequently, 
    \[
    \widehat \Gamma^{TC}(k^\star+\ell)> q_{1-\alpha}
    \]
    happens with probability tending to $1$, whenever
    \[
        \sqrt{\ell}\gg\big\{\log^\beta(C_0+N/\ell)+\log^\beta(C_0+k/N)\big\}\Delta^{-1},
    \]
    as desired. To obtain the delay bounds in Corollaries \ref{cor:power:delay} and \ref{cor:power:delay:np}, put $m=N+k_N^\star$ in the main-text notation and
    \[
        H_N=\log^\beta N+\log^\beta(C_0+k_N^\star/N),\qquad
        r_N=1+H_N^2/\Delta_N^2.
    \]
    For fixed $C>0$, choose $\ell=\min(m,\lceil Cr_N\rceil)$ at absolute time $m+\ell$. Since $1\leq\ell\leq m$, the penalty is bounded by
    \[
        \log^\beta(C_0+N/\ell)+\log^\beta(C_0+(m+\ell)/N)
        \leq\log^\beta(C_0+N)+\log^\beta(C_0+2m/N)\lesssim H_N.
    \]
    If $\ell<m$, the normalized signal is therefore at least of order
    $\sqrt{Cr_N}\Delta_N/H_N\geq\sqrt C$. If $\ell=m$, it is at least of order
    \[
        \frac{\sqrt m\Delta_N}{1+\log^\beta(C_0+m/N)}\to\infty
    \]
    by the consistency condition. Since the normalized noise is uniformly $O_{\PR}(1)$, choosing $C$ sufficiently large gives the $O_P(r_N)$ bound. The same argument with $\Delta_N^F$ and the empirical modulus bounds gives the nonparametric result. The term $1$ accounts for the integer-valued delay.
\end{proof}

\subsection{Proofs for TC under dependence}
\label{sec:proofs:dependence}
We prove the scalar case of Theorem~\ref{thm:banach:dependent} for $\beta>1/2$ under Assumption~\ref{ass:banach:dependence}, with $\mathbb B=\mathbb R$. Condition~\eqref{e:banach:increments}, together with linear interpolation, gives the increment bound needed for Lemma~\ref{lem:finite:moc:proofs}. Condition~\eqref{e:banach:approx} supplies the Gaussian approximation used below.
\begin{theo}
\label{theo:limit:dep}
    Let $\beta>1/2$ and suppose Assumption~\ref{ass:banach:dependence} holds with $\mathbb B=\mathbb R$. Let $B$ 
    be a Brownian motion with variance $\sigma_{LR}^2=\mathbf C_{LR}$. When $H_0$ holds it follows that
    \begin{align}        
        &\sup_{k\geq N}\sup_{ \substack{1 \leq \ell \leq k/2 \\ k-\ell \geq N}}\frac{\Big|\min(1,\frac{\ell}{N})S_{\max(N,\ell)}-(S_k-S_{k-\ell})\Big|}{\sqrt{N}\rho_\beta(\ell/N,k/N)}\\        
        \overset{d}{\to}&\sup_{t \geq 1}\sup_{ \substack{ 0<s\leq t/2 \\ t-s\geq 1}}\frac{\Big|\min(1,s)B(\max(1,s))-(B(t)- B(t-s))\Big|}{\rho_\beta(s,t)} .   
    \end{align}
\end{theo}
\begin{proof}
    Choose $r\in(2\gamma,1)$ and $\eta>0$ such that $\gamma(1+\eta)<r/2$. Put $n_N=\lceil N^{1+\eta}\rceil$ and write $W=W_{n_N}$ for the coupling in \eqref{e:banach:approx}. Then
    \[
        \max_{1\le j\le n_N}\Big|\sum_{i=1}^j\varepsilon_i-W(j)\Big|=O_P\big(N^{\gamma(1+\eta)}\big).
    \]
    \textbf{Case $\ell \leq N$:}    
    We begin by decomposing
    \begin{align*}
        &\Bigg|\frac{|\frac{\ell}{N}S_N-(S_k-S_{k-\ell})|}{\sqrt{N}\rho_\beta(\ell/N,k/N)}-\frac{|\frac{\ell}{N}W(N)-(W(k)-W(k-\ell))|}{\sqrt{N}\rho_\beta(\ell/N,k/N)}\Bigg|\\
        \leq &\frac{|\frac{\ell}{N}(S_N-W(N))|}{\sqrt{N}\rho_\beta(\ell/N,k/N)}+\frac{|(S_k-S_{k-\ell})-(W(k)-W(k-\ell))|}{\sqrt{N}\rho_\beta(\ell/N,k/N)}\\
        := &R_1(k,\ell)+R_2(k,\ell).
    \end{align*}
    $R_1$ vanishes by the same arguments as in the independent case.

    For $R_2$ we estimate
    \begin{align*}
       & \sup_{N \leq k \leq N^{1+\eta}}\sup_{1 \leq \ell \leq N}R_2(k,\ell)\\
        \lesssim  &\sup_{N \leq k \leq N^{1+\eta}}\sup_{1 \leq \ell \leq N^r}R_{2}(k,\ell)+ O_P\big(N^{\gamma(1+\eta)-r/2}\big)\\
        =&\sup_{N \leq k \leq N^{1+\eta}}\sup_{1 \leq \ell \leq N^r}R_{2}(k,\ell)+o_{\PR}(1)
    \end{align*}
    and then use Theorem 1 from \cite{chevallier2023} and Lemma \ref{lem:finite:moc:proofs} to obtain
    \begin{align*}
        &\sup_{N \leq k \leq N^{1+\eta}}\sup_{1 \leq \ell \leq N^r}R_2(k,\ell)\\
        \lesssim & (M_{P_N}+M_W)\sup_{N \leq k \leq N^{1+\eta}}\sup_{1 \leq \ell \leq N^r}\frac{\sqrt{\ell/N(1+\log(k/\ell)+\log(C_0+k/N)}}{\rho_\beta(\ell/N,k/N)}\\
        \lesssim &(M_{P_N}+M_W)\frac{\sqrt{\log(N)}}{\log^\beta(C_0+N^{1-r})}=o_{\PR}(1)~.
    \end{align*}
    Combining the previous inequalities we hence have
    \[
        \sup_{N \leq k \leq N^{1+\eta}}\sup_{1 \leq \ell \leq N}R_1(k,\ell)+R_2(k,\ell)=o_{\PR}(1),
    \]
    so that, using Lemma \ref{lem:cont:discrete}  to switch between the discretized and continuous versions,
    \begin{align*}
           &\sup_{1 \leq t \leq N^{\eta}}\sup_{\substack{0<s\leq 1\\ t-s\geq 1}}\frac{\Big|sP_N(1)-(P_N(t)-P_N(t-s))\Big|}{\rho_\beta(s,t)}\\       
        =&\sup_{1 \leq t \leq N^{\eta}}\sup_{\substack{0<s\leq 1\\ t-s\geq 1}}\frac{\Big|s\tilde W(1)-(\tilde W(t)-\tilde W(t-s))\Big|}{\rho_\beta(s,t)}+o_{\PR}(1),  
    \end{align*}     
    where $\tilde W(t)=N^{-1/2}W(Nt)$.  \\
    \textbf{Case: $\ell >N$:}
    Proceeding as in the previous case we need to bound
    \begin{align*}
        R_3(k,\ell)&=\frac{|S_\ell-W(\ell)|}{\sqrt{N}\rho_\beta(\ell/N,k/N)},\\
        R_4(k,\ell)&=\frac{|(S_k-S_{k-\ell})-(W(k)-W(k-\ell))|}{\sqrt{N}\rho_\beta(\ell/N,k/N)}
    \end{align*}
    uniformly over $N \leq k \leq N^{1+\eta}, N<\ell\leq k/2$. These quantities vanish by upper bounding the numerator by $O_P(N^{\gamma(1+\eta)})=o_P(N^{r/2})$,  yielding
    \begin{align*}
         &\sup_{1 \leq t \leq N^{\eta}}\sup_{\substack{1<s\leq t/2\\ t-s\geq 1}}\frac{\Big|P_N(s)-(P_N(t)-P_N(t-s))\Big|}{\rho_\beta(s,t)}\\       
        =&\sup_{1 \leq t \leq N^{\eta}}\sup_{\substack{1<s\leq t/2\\ t-s\geq 1}}\frac{\Big|\tilde W(s)-(\tilde W(t)-\tilde W(t-s))\Big|}{\rho_\beta(s,t)}+o_{\PR}(1).   
    \end{align*}
   From here we may proceed verbatim as in the independent case.

\end{proof}

\subsection{Proofs for BTC}

For $\beta>1/2$, the proof of Theorem~\ref{thm:banach:dependent} follows the preceding scalar argument with absolute values replaced by the Banach norm. The modulus bounds extend to this metric by the same Garsia--Rodemich--Rumsey argument. The required increment bounds are recorded below.

\begin{lem}
    Suppose Assumption~\ref{ass:banach:dependence} holds. Then, for some constant $C>0$ and all $0\le s<t$, it holds that
    \[
        \| |P_N(t)-P_N(s)|_{\mathbb B}\|_{\Psi_2}\leq C\sqrt{t-s},
    \]
    as well as 
    \[
        \| |B(t)-B(s)|_{\mathbb B}\|_{\Psi_2}\leq C\sqrt{t-s}.
    \]
\end{lem}
\begin{proof}
The first bound follows from \eqref{e:banach:increments} by linear interpolation. The second follows from Brownian scaling and Fernique's theorem.
\end{proof}
The power and delay conclusions follow from the proof of Corollary~\ref{cor:power:delay}, using the reverse triangle inequality in $\mathbb B$ and the same uniform bound on the noise.

\subsection{Additional Lemmas}

\begin{lem}
\label{lem:construct:brownian}
    Let $(\varepsilon_i)_{i \in N}$ be mean zero normally distributed random elements on $C[K]:=\{ f:K\to \R\ | \text{ f continuous}\}$ for some compact $K \subset [0,1]$ such that 
    \begin{align}
      \label{ass:mom}     
          \E[(\varepsilon_i(x)-\varepsilon_i(y))^8]\lesssim |x-y|^4~.
    \end{align}  
    Then there exists a $C[K]$ valued Brownian motion (see \cite{Kuelbs:1975} for a definition) $(B(x))_{x \geq 0}$ with $\E[(B(1))(t)(B(1))(s)]=\E[\epsilon_1(t)\epsilon_1(s)]$ such that for all $n \in \N$
    \[
    B(n)=\sum_{i=1}^n\varepsilon_i~.
    \]    
\end{lem}
\begin{proof}
    We will iteratively construct the stochastic process on the dyadic grid of $[0,\infty)$. By a routine application of the Kolmogorov-Chentsov theorem we may then extend the process to $[0,\infty)$.\\
    
    Given the random variables $(\varepsilon_i)_{i \in \N}$ we proceed as follows. For $t \in \N$ we let $B(t)=\sum_{i=1}^t\varepsilon_i$.  In a first step we define
    \[
        B(t+1/2)=\frac{B(t)+B(t+1)}{2}+Y_{0,t}
    \]
    where the $Y_{0,t}$ have the same distribution as $1/2 \varepsilon_1$ but are independent of one another and of $(\varepsilon_i)_{i \in \N}$. In the next step we then define
    \[
        B(t+1/4)=\frac{B(t)+B(t+1/2)}{2}+Y_{1,t}, \quad t \in N_1:=(\N+1/2)
    \]
    where $Y_{1,t}$ have the same distribution as $1/\sqrt{8} \varepsilon_1$ and are jointly independent of all random variables we used so far in this construction. Continuing like this we have defined $B(t)$ for all dyadics. By the moment assumption \eqref{ass:mom} and Theorem 2.2.4 from \cite{vaart:wellner:1996} we obtain that
    \[
        \E[\|B(t)(\cdot)-B(s)(\cdot)\|_\infty^4]\lesssim (t-s)^2,
    \]
    which allows us to apply a version of the Kolmogorov-Chentsov theorem (e.g. from \cite{kraetschmer:urusov:2023}) to obtain (locally uniform) continuity on the dyadic grid on $[0,\infty)$. We may then extend the process (pathwise) to $[0,\infty)$ due to the locally uniform continuity. It is a routine calculation to verify that the resulting Gaussian process has the correct covariance structure.
\end{proof}

\section{The critical choice \texorpdfstring{$\beta=1/2$}{beta=1/2} and a logarithmic lower bound}
\label{sec:critical-beta}

This section proves the critical case $\beta=1/2$ of Theorems~\ref{thm:main:1} and \ref{thm:banach:dependent}. Its delay bounds match the order of the lower bound established at the end of the section. We establish this for dependent Banach valued data, which of course includes the finite dimensional setting by a trivial embedding argument. There is one complication though: at the boundary
$\beta=1/2$ the weight agrees, up to universal constants, with the L\'evy modulus-of-continuity normalization. The shortest windows then remain of constant order instead of vanishing. Thus ordinary weak convergence need not hold, while the upper tail is still governed by the Brownian approximation. This is the setting in which convergence beyond a threshold is natural (see \cite{koehne:mies:2025}). We say that $Z_N$ converges in distribution beyond a threshold $\tau$ to $Z$ if
\[
 \mathbb P(Z_N\ge x)\longrightarrow\mathbb P(Z\ge x),\qquad x>\tau,
\]
for every continuity point $x$ of the distribution function of $Z$, and write
$Z_N\overset d\to_{\tau}Z$.

Throughout this section, $\widehat\Gamma^{BTC}$ denotes the Banach-valued TWIN
statistic of Section~\ref{sec:23}, with $\beta=1/2$. We use the notational
convention introduced at the beginning of the Supplement: $k>N$ denotes the
absolute current time, so that the observations $X_1,\ldots,X_k$ are available.
Thus, if the scanned window has length $\ell$, we write
$t=k/N$ and $s=\ell/N$. In particular, every supremum over $k$ below is a
supremum over absolute times $k>N$. We otherwise use the notation of the main
text. In particular, $(\mathbb B,|\cdot|_{\mathbb B})$ is a separable Banach
space and $P_N$ is the rescaled partial-sum process. Put
\begin{equation}
\label{eq:critical-rho}
 \rho_{1/2}(s,t)
 :=\sqrt{s}\{\sqrt{\log(C_0+1/s)}+\sqrt{\log(C_0+t)}\}.
\end{equation}
We will frequently use that 
\begin{equation}
\label{eq:critical-weight-comparison}
 \sqrt{s\{\log(C_0+1/s)+\log(C_0+t)\}}
 \le \rho_{1/2}(s,t)
 \le \sqrt{2s\{\log(C_0+1/s)+\log(C_0+t)\}}
\end{equation}
without further comment.

The section is organized as follows: In the first subsection we establish convergence beyond a threshold for $\widehat \Gamma^{BTC}$ for physically dependent Banach valued time series. This is done under two high level conditions which we verify for physically dependent $C([0,1])$ valued time series in Section \ref{sec:critical-c01}. In the second subsection we provide some auxiliary bounds that are required for the proof. They provide a refined version of the global modulus of continuity result used for the $\beta>1/2$ case. In the last subsection we provide a lower bound  on detection times that is valid for all detectors and that matches the TWIN delay bounds for $\beta=1/2$.

\subsection{Physically dependent Banach-valued observations}
\label{sec:critical-beta-banach}

Suppose Assumption~\ref{ass:banach:dependence} holds. By \eqref{e:banach:increments} and linear interpolation,
\begin{equation}
\label{eq:banach-subg-increments-critical}
 \big\|\,|P_N(t)-P_N(s)|_{\mathbb B}\,\big\|_{\Psi_2}
 \le C\sqrt{|t-s|},\qquad s,t\ge0,
\end{equation}
uniformly in $N$.

Condition~\eqref{e:banach:approx} supplies, for each $n$, a $\mathbb B$-valued Brownian motion $W_n$ and some $\gamma<1/2$ such that
\begin{equation}
\label{eq:critical-strong-approx-banach}
 \max_{1\le j\le n}
 \left|\sum_{i=1}^j\varepsilon_i-W_n(j)\right|_{\mathbb B}
 =O_{\mathbb P}(n^\gamma),\qquad n\to\infty.
\end{equation}
The covariance operator of $W_n(1)$ is $\mathbf C_{LR}$. Whenever the covariance series converges absolutely, it is characterized, for $f,g\in\mathbb B^*$, by
\begin{equation}
\label{eq:critical-long-run-cov-banach}
 \operatorname{Cov}\{f(W_n(1)),g(W_n(1))\}
 =\sum_{h\in\mathbb Z}
 \operatorname{Cov}\{f(\varepsilon_0),g(\varepsilon_h)\}.
\end{equation}
Section~\ref{sec:critical-c01} below gives primitive sufficient conditions for
both \eqref{eq:banach-subg-increments-critical} and
\eqref{eq:critical-strong-approx-banach} when
$\mathbb B=\mathcal C([0,1])$. These results then of course immediately transfer for other Banach spaces used in practice such as $L^2([0,1])$.

\begin{theo}\label{thm:critical-banach}
Suppose the Banach-valued errors satisfy the preceding conditions. Under
$H_0$ there is a finite $\tau$ such that
\begin{align}
\label{eq:critical-banach-limit}
 \sup_{k>N}\widehat\Gamma^{BTC}(k)
 \overset d\to_{\tau}
 &\sup_{t>1}\sup_{\substack{0<s\le t/2\\t-s\ge1}}
 \frac{\big|\min(1,s)B(\max(1,s))-[B(t)-B(t-s)]\big|_{\mathbb B}}
 {\rho_{1/2}(s,t)}\\
 &=:L_{1/2}(\mathbf C_{LR}).\nonumber
\end{align}
Here $B$ is a $\mathbb B$-valued Brownian motion with covariance operator
$\mathbf C_{LR}$.
\end{theo}

\begin{proof}
The proof follows the same decomposition as in the case $\beta>1/2$,
with the essential difference that the short scales are bounded rather than
$o_{\mathbb P}(1)$. We give all steps explicitly.

Write
\begin{align*}
 G_N(t,s)&:=\big|\min(1,s)P_N(\max(1,s))
 -[P_N(t)-P_N(t-s)]\big|_{\mathbb B},\\
 D_N(t,s)&:=\frac{G_N(t,s)}{\rho_{1/2}(s,t)},
\end{align*}
where $t=k/N$ and $s=\ell/N$ on the discrete grid. Choose $0<\delta<1$ and then
$\zeta>0$ so small that
\begin{equation}
\label{eq:critical-choice-delta}
 \gamma(1+\zeta)<\frac{1-\delta}{2}.
\end{equation}
We split the scan into
\[
 \mathcal S_N=\{t\le N^\zeta:0<s<N^{-\delta}\},\qquad
 \mathcal M_N=\{t\le N^\zeta:s\ge N^{-\delta}\},\qquad
 \mathcal L_N=\{t>N^\zeta\},
\]
always imposing the constraints $0<s\le t/2$ and $t-s\ge1$.

\emph{Moderate scales.}
Set $n_N=\lceil N^{1+\zeta}\rceil$, use the coupling $W=W_{n_N}$, and define the rescaled Brownian motion
\[
 B_N(u):=N^{-1/2}W(Nu),\qquad u\ge0.
\]
By Brownian scaling, $B_N$ has the same distribution as a $\mathbb B$-valued
Brownian motion $B$ with covariance operator $\mathbf C_{LR}$. We write $D_B$
for the analog to $D_N$ obtained from this generic Brownian motion. On the grid
$u=j/N$, \eqref{eq:critical-strong-approx-banach} gives
\begin{align}
\label{eq:critical-SA-rescaled}
 &\max_{0\le j\le n_N}
 \left|P_N(j/N)-B_N(j/N)\right|_{\mathbb B}\\
 &\qquad=N^{-1/2}
 \max_{0\le j\le n_N}
 \left|\sum_{i=1}^j\varepsilon_i-W(j)\right|_{\mathbb B}
 =O_{\mathbb P}\big(N^{\gamma(1+\zeta)-1/2}\big).\nonumber
\end{align}
Define
\begin{align*}
 G_{B_N}(t,s)&:=\big|\min(1,s)B_N(\max(1,s))
 -[B_N(t)-B_N(t-s)]\big|_{\mathbb B},\\
 D_{B_N}(t,s)&:=\frac{G_{B_N}(t,s)}{\rho_{1/2}(s,t)}.
\end{align*}
At every grid pair $(t,s)$, the reverse triangle inequality gives
\begin{align*}
 |G_N(t,s)-G_{B_N}(t,s)|
 &\le\big|\min(1,s)\{P_N(\max(1,s))-B_N(\max(1,s))\}\big|_{\mathbb B}\\
 &\quad+|P_N(t)-B_N(t)|_{\mathbb B}
 +|P_N(t-s)-B_N(t-s)|_{\mathbb B}\\
 &\le3\max_{0\le j\le n_N}
 |P_N(j/N)-B_N(j/N)|_{\mathbb B}.
\end{align*}
On $\mathcal M_N$,
\[
 \rho_{1/2}(s,t)\ge c\sqrt s\ge cN^{-\delta/2},
\]
and hence
\begin{align}
\label{eq:critical-moderate-approx}
 &\sup_{\substack{(t,s)\in\mathcal M_N\\Nt,Ns\in\mathbb N}}
 |D_N(t,s)-D_{B_N}(t,s)|\\
 &\qquad=O_{\mathbb P}\left(
 N^{\gamma(1+\zeta)-1/2+\delta/2}\right)
 =o_{\mathbb P}(1)\nonumber
\end{align}
by \eqref{eq:critical-choice-delta}. This includes $s>1$: in that case the
first TWIN term is $P_N(s)$ rather than $sP_N(1)$, and $s\le t\le N^\zeta$.

We next pass from the discrete Brownian scan to the continuous one, and record
the increment calculation explicitly. For a feasible $(t,s)\in\mathcal M_N$ set
\[
 t_N:=\frac{\lceil Nt\rceil}{N},\qquad
 s_N:=\frac{\lfloor Ns\rfloor}{N}.
\]
For all large $N$, $s_N>0$ because $s\ge N^{-\delta}$, and
\[
 0\le t_N-t\le N^{-1},\qquad
 0\le s-s_N\le N^{-1}.
\]
Moreover $s_N\le s\le t/2\le t_N/2$ and
$t_N-s_N\ge t-s\ge1$, so $(t_N,s_N)$ is again feasible. Put
\[
 A_B(s):=\min(1,s)B_N(\max(1,s)),
\]
and, for $T>0$ and $h>0$,
\[
 \omega_{B_N}(h;T):=
 \sup_{\substack{0\le u,v\le T\\|u-v|\le h}}
 |B_N(u)-B_N(v)|_{\mathbb B}.
\]
If $s,s_N\le1$, then
\[
 |A_B(s)-A_B(s_N)|_{\mathbb B}
 =|s-s_N|\,|B_N(1)|_{\mathbb B}
 \le N^{-1}|B_N(1)|_{\mathbb B}.
\]
If $s,s_N\ge1$, then
\[
 |A_B(s)-A_B(s_N)|_{\mathbb B}
 \le\omega_{B_N}(N^{-1};N^\zeta+1).
\]
Finally, if $s_N\le1<s$, then $s-s_N\le N^{-1}$ and
\begin{align*}
 |A_B(s)-A_B(s_N)|_{\mathbb B}
 &\le |B_N(s)-B_N(1)|_{\mathbb B}
 +(1-s_N)|B_N(1)|_{\mathbb B}\\
 &\le\omega_{B_N}(N^{-1};N^\zeta+1)
 +N^{-1}|B_N(1)|_{\mathbb B}.
\end{align*}
Thus in all cases
\begin{equation}
\label{eq:critical-grid-first-term}
 |A_B(s)-A_B(s_N)|_{\mathbb B}
 \le\omega_{B_N}(N^{-1};N^\zeta+1)
 +N^{-1}|B_N(1)|_{\mathbb B}.
\end{equation}
For the recent-window term,
\begin{align}
\label{eq:critical-grid-recent-term}
 &\left|[B_N(t)-B_N(t-s)]-
 [B_N(t_N)-B_N(t_N-s_N)]\right|_{\mathbb B}\\
 &\quad\le |B_N(t)-B_N(t_N)|_{\mathbb B}
 +|B_N(t-s)-B_N(t_N-s_N)|_{\mathbb B}\nonumber\\
 &\quad\le2\omega_{B_N}(2N^{-1};N^\zeta+1),\nonumber
\end{align}
because $|(t_N-s_N)-(t-s)|\le2N^{-1}$. Combining
\eqref{eq:critical-grid-first-term} and
\eqref{eq:critical-grid-recent-term},
\begin{equation}
\label{eq:critical-grid-numerator}
 |G_{B_N}(t,s)-G_{B_N}(t_N,s_N)|
 \le3\omega_{B_N}(2N^{-1};N^\zeta+1)
 +N^{-1}|B_N(1)|_{\mathbb B}.
\end{equation}
Choose any $\delta_1\in(0,1-\delta)$. Since
$2N^{-1}\le N^{-1+\delta_1}$ for all large $N$, Corollary~\ref{cor:critical-small-scales}, applied to $B_N$ (which has the
same law as $B$), gives with probability $1-o(1)$
\begin{align*}
 \omega_{B_N}(2N^{-1};N^\zeta+1)
 &\le C\sup_{0<h\le2/N}\sqrt{h\log(C_0+h^{-1})}\\
 &\le C N^{-1/2}\sqrt{\log N},
\end{align*}
where the last inequality follows from monotonicity of
$h\mapsto h\log(C_0+h^{-1})$ for all sufficiently small $h$. Hence
\[
 \omega_{B_N}(2N^{-1};N^\zeta+1)
 =O_{\mathbb P}(N^{-1/2}\sqrt{\log N}).
\]
Since $\rho_{1/2}(s,t)\ge cN^{-\delta/2}$ on $\mathcal M_N$,
\begin{align}
\label{eq:critical-grid-numerator-normalized}
 &\sup_{(t,s)\in\mathcal M_N}
 \frac{|G_{B_N}(t,s)-G_{B_N}(t_N,s_N)|}
 {\rho_{1/2}(s,t)}\\
 &\qquad=O_{\mathbb P}\left(
 N^{-(1-\delta)/2}\sqrt{\log N}+N^{-1+\delta/2}\right)
 =o_{\mathbb P}(1).\nonumber
\end{align}

It remains to control the denominator. Writing
$a(s)=\sqrt{\log(C_0+1/s)}$ and
$b(t)=\sqrt{\log(C_0+t)}$, we have
$\rho_{1/2}(s,t)=\sqrt s\{a(s)+b(t)\}$. Direct differentiation gives
\begin{align*}
 \left|\partial_t\log\rho_{1/2}(s,t)\right|
 &=\frac{1}{2(C_0+t)b(t)\{a(s)+b(t)\}}
 \le C,\\
 \left|\partial_s\log\rho_{1/2}(s,t)\right|
 &\le\frac1{2s}+
 \frac{1}{2s(C_0s+1)a(s)\{a(s)+b(t)\}}
 \le\frac Cs.
\end{align*}
Because $s\ge N^{-\delta}$ and $|s-s_N|\le N^{-1}$, we have
$s_N\ge s/2$ for all large $N$. The mean-value theorem therefore yields,
uniformly on $\mathcal M_N$,
\begin{align}
\label{eq:critical-grid-weight}
 \left|\log\frac{\rho_{1/2}(s_N,t_N)}{\rho_{1/2}(s,t)}\right|
 &\le C\left(\frac{|s-s_N|}{s}+|t-t_N|\right)\\
 &\le C\{N^{\delta-1}+N^{-1}\}=o(1).\nonumber
\end{align}
Hence
\[
 \sup_{(t,s)\in\mathcal M_N}
 \left|\frac{\rho_{1/2}(s_N,t_N)}{\rho_{1/2}(s,t)}-1\right|=o(1).
\]
Lemma~\ref{lem:critical-horizon}, together with
\eqref{eq:critical-weight-comparison}, implies
$\sup_{t>1,s}|D_{B_N}(t,s)|=O_{\mathbb P}(1)$. Combining this fact with
\eqref{eq:critical-grid-numerator-normalized} and
\eqref{eq:critical-grid-weight} gives
\begin{equation}
\label{eq:critical-grid-to-continuous}
 \sup_{(t,s)\in\mathcal M_N}
 |D_{B_N}(t,s)-D_{B_N}(t_N,s_N)|=o_{\mathbb P}(1).
\end{equation}
Thus the discrete and continuous Brownian suprema on $\mathcal M_N$ differ by
$o_{\mathbb P}(1)$.

\emph{Short scales.}
On $\mathcal S_N$ necessarily $s<1$, so the first term equals $sP_N(1)$. Since
$|P_N(1)|_{\mathbb B}=O_{\mathbb P}(1)$,
\begin{align*}
 \sup_{(t,s)\in\mathcal S_N}
 \frac{s|P_N(1)|_{\mathbb B}}{\rho_{1/2}(s,t)}
 &\le |P_N(1)|_{\mathbb B}
 \sup_{0<s<N^{-\delta}}\frac{\sqrt s}{\sqrt{\log(C_0+1/s)}}\\
 &=o_{\mathbb P}(1).
\end{align*}
For the recent-window term, \eqref{eq:banach-subg-increments-critical} and
Corollary~\ref{cor:critical-small-scales} imply that there is a deterministic
$C_S<\infty$ such that
\begin{equation}
\label{eq:critical-short-threshold}
 \mathbb P\left(\sup_{(t,s)\in\mathcal S_N}D_N(t,s)>C_S\right)=o(1).
\end{equation}
The same statement holds for $D_{B_N}$, since Brownian increments satisfy the same
subgaussian condition. This is the key difference to the case $\beta>1/2$: there is no remaining factor
$\log^{\beta-1/2}(C_0+1/s)$ which would send the short-scale contribution to zero. 

\emph{Long horizon.}
By the triangle inequality, the numerator of $D_N(t,s)$ is bounded by the sum
of at most two increments of $P_N$: the recent increment of length $s$ and,
depending on whether $s\le1$ or $s>1$, either $s|P_N(1)|_{\mathbb B}$ or
$|P_N(s)-P_N(0)|_{\mathbb B}$. Lemma~\ref{lem:critical-horizon}, together with
\eqref{eq:critical-weight-comparison}, therefore yields a deterministic
$C_L<\infty$ such that
\begin{equation}
\label{eq:critical-late-threshold}
 \mathbb P\left(\sup_{(t,s)\in\mathcal L_N}D_N(t,s)>C_L\right)=o(1).
\end{equation}
For the term $s|P_N(1)|$ when $s\le1$, use additionally
$s/\rho_{1/2}(s,t)\lesssim\sqrt{s}/\sqrt{\log(C_0+t)}$, which is even smaller
as $t\to\infty$. Again the same bound holds for $D_{B_N}$.

Let $\tau>\max(C_S,C_L)$, enlarged if necessary to absorb the fixed constants in
\eqref{eq:critical-weight-comparison}. By
\eqref{eq:critical-short-threshold}--\eqref{eq:critical-late-threshold}, an
exceedance of any fixed $x>\tau$ is, with probability $1-o(1)$, generated on
$\mathcal M_N$. On this set the coupling error is $o_{\mathbb P}(1)$ by
\eqref{eq:critical-moderate-approx}. Hence, for every continuity point $x>\tau$,
\begin{align*}
 \mathbb P\left(\sup_{k>N}\widehat\Gamma^{BTC}(k)>x\right)
 &=\mathbb P\left(\sup_{(t,s)\in\mathcal M_N}D_{B_N}(t,s)>x\right)+o(1)\\
 &=\mathbb P\left(\sup_{(t,s)\in\mathcal M_N}D_B(t,s)>x\right)+o(1)\\
 &\longrightarrow
 \mathbb P\left(\sup_{t>1}\sup_{\substack{0<s\le t/2\\t-s\ge1}}
 D_B(t,s)>x\right).
\end{align*}
Here the second equality uses $B_N\overset d=B$. The last passage follows because
the short-scale and long-horizon Brownian contributions are below
$\tau<x$ with probability tending to one, while the moderate sets exhaust the
remaining parameter region. This is precisely \eqref{eq:critical-banach-limit}.
\end{proof}

Let $q^{\mathbf C_{LR}}_{1-\alpha,1/2}$ be a $(1-\alpha)$ quantile of
$L_{1/2}(\mathbf C_{LR})$. If
\begin{equation}
\label{eq:critical-quantile-threshold}
 q^{\mathbf C_{LR}}_{1-\alpha,1/2}>\tau,
\end{equation}
then
\[
 \mathbb P\!\left(\sup_{k>N}\widehat\Gamma^{BTC}(k)>
 q^{\mathbf C_{LR}}_{1-\alpha,1/2}\right)
 \longrightarrow
 \mathbb P\!\left(L_{1/2}(\mathbf C_{LR})>
 q^{\mathbf C_{LR}}_{1-\alpha,1/2}\right)=\alpha.
\]

\begin{cor}\label{cor:critical-delay}
Suppose the conditions of Theorem~\ref{thm:critical-banach} hold and the
alternative \eqref{e:H1X} has jump size
$\delta_N=|\mu^{(1)}-\mu^{(2)}|_{\mathbb B}$. If
\begin{equation}
\label{eq:critical-consistency}
 \delta_N\gg
 \frac{\sqrt{\log(C_0+N/k_N^\star)}+
 \sqrt{\log(C_0+k_N^\star/N)}}{\sqrt{N+k_N^\star}},
\end{equation}
then the critical TWIN procedure is consistent for every critical value above
$\tau$. Moreover, if $\widetilde\tau$ denotes its detection delay,
\begin{equation}
\label{eq:critical-delay}
 \widetilde\tau=O_{\mathbb P}\left(
 1+\frac{\{\sqrt{\log N}+\sqrt{\log(C_0+k_N^\star/N)}\}^2}
 {\delta_N^2}\right).
\end{equation}
In particular, for fixed $C_0>1$,
\[
 \widetilde\tau=O_{\mathbb P}\left(
 1+\frac{\log(N+k_N^\star)}{\delta_N^2}\right).
\]
\end{cor}

\begin{proof}
Exactly the same as the one for Corollary~\ref{cor:power:delay}. 
\end{proof}

\subsection{Auxiliary bounds for small scales and long horizons}
\label{sec:critical-general-bounds}

\begin{lem}\label{lem:critical-horizon}
Let $W$ be a separable $(M,d)$-valued process with continuous paths such that
\[
 \|d(W(t),W(r))\|_{\Psi_2}\le K\sqrt{|t-r|},
 \qquad t,r\ge0.
\]
There
is $x_0=x_0(K)<\infty$ and, for each $a\ge1$, a finite $C(a,K)$ such that
$C(a,K)\to0$ as $a\to\infty$ and
\begin{align}
\label{eq:critical-horizon-tail}
 \mathbb P\!\left(
 \sup_{t\ge a}\sup_{0<s\le t}
 \frac{d(W(t),W(t-s))}
 {\sqrt{s\{\log(C_0+1/s)+\log(C_0+t)\}}}\ge x_0\right)
 \le C(a,K).
\end{align}
Moreover,
\begin{equation}
\label{eq:critical-horizon-op}
 \sup_{t\ge0}\sup_{0<s\le t}
 \frac{d(W(t),W(t-s))}
 {\sqrt{s\{\log(C_0+1/s)+\log(C_0+t)\}}}
 =O_{\mathbb P}(1).
\end{equation}
\end{lem}

\begin{proof}
We use the same Garsia--Rodemich--Rumsey construction as in the proof of
Lemma~\ref{lem:finite:moc:proofs}, but keep the random Orlicz integral instead
of passing directly to the global random modulus $M_G$. This gives us finer local control compared to finding a global bound.

Using the same notation as in the proof of Lemma~\ref{lem:finite:moc:empirical:proofs} fix
\[
    \eta=\frac12,\qquad c=3,
\]
so in particular $c\eta>1$. As in the proof of
Lemma~\ref{lem:finite:moc:proofs}, put
\[
    \Psi(x)=e^{x^2/2}-1,
    \qquad
    \mu(x)=K\sqrt{cx},
\]
where $\mu$ is the function defined in \eqref{eq:def:mu}. Thus
\[
    \Psi^{-1}(y)=\sqrt{2\log(1+y)},
    \qquad
    d\mu(u)=\frac{K\sqrt c}{2\sqrt u}\,du.
\]
For $T>0$, define
\[
 \xi_T
 :=
 f_\eta(T)
 \int_0^T\int_0^T
 \Psi\!\left(
 \frac{d(W(x),W(y))}{\mu(|x-y|)}
 \right)\,dx\,dy,
\]
where
\[
 f_\eta(T)
 :=
 \begin{cases}
 (1/T+1)^{2(1-\eta)},&0<T<1,\\
 1,&T=1,\\
 (T-1)^{-2(1+\eta)},&T>1,
 \end{cases}
\]
and set
\[
 \xi:=\sup\{\xi_m,\xi_{1/m}:m\in\mathbb N\}.
\]
The proof of Lemma~\ref{lem:finite:moment:xi} applies verbatim to the present
metric-valued process, so that for some $p>1$,
\begin{equation}
\label{eq:critical-xi-moment}
    \mathbb E[\xi^p]<\infty.
\end{equation}

We next record explicitly the two bounds on the Orlicz integral which enter
the GRR inequality. Write
\[
 I(t)
 :=
 \int_0^t\int_0^t
 \Psi\!\left(
 \frac{d(W(x),W(y))}{\mu(|x-y|)}
 \right)\,dx\,dy.
\]
If $t\ge1$ and $m=\lceil t\rceil$, then $m-1<t\le m$ and hence
\begin{align}
 I(t)
 &\le I(m)
 =f_\eta(m)^{-1}\xi_m
 =(m-1)^{2(1+\eta)}\xi_m
 \le t^{2(1+\eta)}\xi .
 \label{eq:critical-chev-large-t}
\end{align}
If $0<t<1$, let $m=\lfloor1/t\rfloor$. Then
\[
 \frac1{m+1}<t\le\frac1m,
\]
and therefore
\begin{align}
 I(t)
 &\le I(1/m)
 =f_\eta(1/m)^{-1}\xi_{1/m}
 =(m+1)^{-2(1-\eta)}\xi_{1/m}
 \le t^{2(1-\eta)}\xi .
 \label{eq:critical-chev-small-t}
\end{align}

Let now $0\le r<t$ and write $h=t-r$. Applying the metric-space GRR
inequality exactly as in the proof of Lemma~\ref{lem:finite:moc:proofs},
together with \eqref{eq:critical-chev-large-t} and
\eqref{eq:critical-chev-small-t}, gives
\begin{align}
 d(W(t),W(r))
 \le
 \begin{cases}
 8\displaystyle\int_0^h
 \Psi^{-1}\!\left(
 \dfrac{4t^{2(1+\eta)}\xi}{u^2}
 \right)d\mu(u),
 &t\ge1,\\[1.2em]
 8\displaystyle\int_0^h
 \Psi^{-1}\!\left(
 \dfrac{4t^{2(1-\eta)}\xi}{u^2}
 \right)d\mu(u),
 &0<t<1.
 \end{cases}
 \label{eq:critical-chev-grr}
\end{align}

We bound the integrals in \eqref{eq:critical-chev-grr}. For any $A\ge0$ and
$0<u\le h$,
\[
 1+\frac{A}{u^2}
 \le
 \frac{h^2}{u^2}
 \left(1+\frac{A}{h^2}\right),
\]
and consequently
\[
 \sqrt{\log\left(1+\frac{A}{u^2}\right)}
 \le
 \sqrt{\log\left(1+\frac{A}{h^2}\right)}
 +\sqrt{2\log(h/u)}.
\]
Since
\[
 \int_0^h u^{-1/2}\,du=2\sqrt h
\]
and, after the substitution $u=hv$,
\[
 \int_0^h
 u^{-1/2}\sqrt{2\log(h/u)}\,du
 =
 \sqrt h\int_0^1
 v^{-1/2}\sqrt{2\log(1/v)}\,dv
 =
 2\sqrt{\pi h},
\]
we obtain from \eqref{eq:critical-chev-grr}
\begin{align}
 d(W(t),W(r))
 &\le
 C_K\sqrt h
 \left[
 \sqrt\pi+
 \sqrt{
 \log\left(
 1+\frac{4t^{2(1+\eta)}\xi}{h^2}
 \right)}
 \right],
 &&t\ge1,
 \label{eq:critical-chev-bound-large}\\
 d(W(t),W(r))
 &\le
 C_K\sqrt h
 \left[
 \sqrt\pi+
 \sqrt{
 \log\left(
 1+\frac{4t^{2(1-\eta)}\xi}{h^2}
 \right)}
 \right],
 &&0<t<1,
 \label{eq:critical-chev-bound-small}
\end{align}
where
\[
 C_K:=8K\sqrt{2c}=8K\sqrt6.
\]

Put
\[
 L(h,t)
 :=
 \log(C_0+h^{-1})+\log(C_0+t).
\]
For $t\ge1$,
\begin{align*}
 \log\left(
 1+\frac{4t^{2(1+\eta)}\xi}{h^2}
 \right)
 &\le
 \log(1+4\xi)
 +2(1+\eta)\log(1+t)
 +2\log(1+h^{-1})\\
 &\le
 \log(1+4\xi)+2(1+\eta)L(h,t)
 =
 \log(1+4\xi)+3L(h,t),
\end{align*}
where we used $\eta=1/2$. Hence
\begin{equation}
\label{eq:critical-chev-ratio-large}
 \frac{d(W(t),W(t-h))}
 {\sqrt{hL(h,t)}}
 \le
 C_K\left\{
 \sqrt{\frac{\pi}{L(h,t)}}
 +
 \sqrt{3+\frac{\log(1+4\xi)}{L(h,t)}}
 \right\},
 \qquad t\ge1,\quad 0<h\le t.
\end{equation}
For $0<t<1$, since $t^{2(1-\eta)}\le1$, the same argument gives
\begin{align*}
 \log\left(
 1+\frac{4t^{2(1-\eta)}\xi}{h^2}
 \right)
 &\le
 \log(1+4\xi)+2\log(1+h^{-1})
 \le \log(1+4\xi)+2L(h,t),
\end{align*}
and therefore
\begin{equation}
\label{eq:critical-chev-ratio-small}
 \frac{d(W(t),W(t-h))}
 {\sqrt{hL(h,t)}}
 \le
 C_K\left\{
 \sqrt{\frac{\pi}{L(h,t)}}
 +
 \sqrt{2+\frac{\log(1+4\xi)}{L(h,t)}}
 \right\},
 \qquad 0<t<1,\quad 0<h\le t.
\end{equation}

We first prove \eqref{eq:critical-horizon-tail}. For $t\ge a\ge1$,
\[
 L(h,t)\ge\log(C_0+t)\ge\log(C_0+a)=:\ell_a.
\]
Consequently, \eqref{eq:critical-chev-ratio-large} yields
\begin{equation}
\label{eq:critical-chev-tail-majorant}
 \sup_{t\ge a}\sup_{0<h\le t}
 \frac{d(W(t),W(t-h))}{\sqrt{hL(h,t)}}
 \le
 C_K\left\{
 \sqrt{\frac{\pi}{\ell_a}}
 +
 \sqrt{3+\frac{\log(1+4\xi)}{\ell_a}}
 \right\}.
\end{equation}
Choose, for example,
\[
 x_0:=2C_K(\sqrt\pi+\sqrt3+1).
\]
For all sufficiently large $a$, $\ell_a\ge1$. On that range,
\eqref{eq:critical-chev-tail-majorant} shows that the event on the
left-hand side of \eqref{eq:critical-horizon-tail} is contained in
\[
 \{\log(1+4\xi)\ge c_\star\ell_a\}
\]
for the numerical constant
\[
 c_\star
 :=
 \left\{
 \frac{x_0}{C_K}-\sqrt\pi
 \right\}^2-3
 >0.
\]
By \eqref{eq:critical-xi-moment},
\begin{align*}
 \mathbb P(\log(1+4\xi)\ge c_\star\ell_a)
 &=
 \mathbb P\left(
 \xi\ge\frac{e^{c_\star\ell_a}-1}{4}
 \right)\\
 &\le
 4^p\mathbb E[\xi^p]
 \bigl(e^{c_\star\ell_a}-1\bigr)^{-p}
 \longrightarrow0.
\end{align*}
Thus \eqref{eq:critical-horizon-tail} follows, after enlarging the bound to
$1$ for the finitely many smaller values of $a$. This also gives an explicit
choice of $C(a,K)$ tending to zero.

Finally,
\[
 L(h,t)\ge2\log C_0>0
\]
for all $t>0$ and $0<h\le t$. Combining
\eqref{eq:critical-chev-ratio-large} and
\eqref{eq:critical-chev-ratio-small}, we obtain
\begin{align*}
 &\sup_{t\ge0}\sup_{0<h\le t}
 \frac{d(W(t),W(t-h))}{\sqrt{hL(h,t)}}\\
 &\qquad\le
 C_K\left\{
 \sqrt{\frac{\pi}{2\log C_0}}
 +
 \sqrt{
 3+\frac{\log(1+4\xi)}{2\log C_0}
 }
 \right\}.
\end{align*}
The right-hand side is $O_{\mathbb P}(1)$ by
\eqref{eq:critical-xi-moment}. This proves
\eqref{eq:critical-horizon-op}.
\end{proof}

\begin{lem}\label{lem:critical-metric}
Let $W_N$ be processes with continuous paths in a metric space $(M,d)$ and let
$\rho_0:[0,N^\zeta]\to[0,\infty)$ be increasing with $\rho_0(0)=0$. Suppose
that there are $C>0$ and, for every $x>C$, continuous functions $K(x)>0$ and
$\kappa(x)>1$ such that
\begin{equation}
\label{eq:critical-tail-tight-metric}
 \mathbb P\!\left(
 \frac{d(W_N(t),W_N(r))}{\rho_0(|t-r|)}>x\right)
 \le K(x)|t-r|^{\kappa(x)},\qquad t,r\in[0,N^\zeta].
\end{equation}
Assume in addition that for some $p>0$ and $\widetilde K>0$,
\begin{equation}
\label{eq:critical-rho-factor}
 \rho_0(zh)\le\widetilde K z^p\rho_0(h),\qquad z,h\in[0,1].
\end{equation}
If $0<\delta<1$ and
\begin{equation}
\label{eq:critical-chain-condition}
 \zeta+1+\delta-\kappa(x)(1-\delta)<0,
\end{equation}
then
\begin{equation}
\label{eq:critical-cont-modulus}
 \mathbb P\!\left(
 \sup_{\substack{t,r\in[0,N^\zeta]\\|t-r|\le N^{\delta-1}}}
 \frac{d(W_N(t),W_N(r))}{\rho_0(|t-r|)}>x\right)\longrightarrow0.
\end{equation}
\end{lem}

\begin{proof}
The proof is a chaining argument and essentially a slightly more quantitative version of Theorem 3.1 in \cite{koehne:mies:2025}.

Fix $x>C$ and use the same $\delta$ as in the statement. For $m\ge1$ define
\begin{align*}
 A_m:=\Bigg\{&\max_{1\le l\le2^{\lfloor m\delta\rfloor}}
 \max_{0\le j\le\lfloor N^\zeta2^m\rfloor-l}
 \frac{d(W_N((l+j)2^{-m}),W_N(j2^{-m}))}
 {\rho_0(l2^{-m})}\le x\Bigg\}.
\end{align*}
By the union bound and \eqref{eq:critical-tail-tight-metric},
\begin{align*}
 \mathbb P(A_m^c)
 &\le K(x)\lfloor N^\zeta2^m\rfloor
 \sum_{l=1}^{2^{\lfloor m\delta\rfloor}}
 (l2^{-m})^{\kappa(x)}\\
 &\le K(x)N^\zeta2^m 2^{-m\kappa(x)}
 \sum_{l=1}^{2^{\lfloor m\delta\rfloor}}l^{\kappa(x)}\\
 &\lesssim K(x)N^\zeta
 2^{m\{1+\delta-\kappa(x)(1-\delta)\}}.
\end{align*}
The exponent $1+\delta-\kappa(x)(1-\delta)$ is negative by
\eqref{eq:critical-chain-condition}. Hence the last display is summable over
$m$.

Let
\[
 T_m:=\{j2^{-m}:0\le j\le\lfloor N^\zeta2^m\rfloor\},
 \qquad
 \underline A_m:=\bigcap_{k=m}^\infty A_k.
\]
Fix $0\le r<t\le N^\zeta$ and put $h=t-r$. Choose $m$ so that
\begin{equation}
\label{eq:critical-dyadic-h}
 2^{-(m+1)(1-\delta)}<h\le2^{-m(1-\delta)}.
\end{equation}
For all sufficiently large $m$, the interval $[r,t]$ contains grid points of
$T_m$. Define the inward approximations
\[
 r_m:=2^{-m}\lceil2^m r\rceil,
 \qquad
 t_m:=2^{-m}\lfloor2^m t\rfloor.
\]
Then $r\le r_m\le t_m\le t$ and
\[
 0\le t_m-r_m\le h,
 \qquad
 2^m(t_m-r_m)\le 2^m h\le2^{m\delta}.
\]
Thus $A_m$ gives
\begin{equation}
\label{eq:critical-chain-initial}
 d(W_N(t_m),W_N(r_m))
 \le x\rho_0(t_m-r_m)
 \le x\rho_0(h).
\end{equation}
For $k\ge m$ define similarly
\[
 r_k:=2^{-k}\lceil2^k r\rceil,
 \qquad
 t_k:=2^{-k}\lfloor2^k t\rfloor.
\]
Then $r_k\downarrow r$, $t_k\uparrow t$, and
\[
 0\le r_k-r_{k+1}\le2^{-(k+1)},
 \qquad
 0\le t_{k+1}-t_k\le2^{-(k+1)}.
\]
Both endpoints in each of these two increments belong to $T_{k+1}$ and differ
by at most one grid step. Hence, on $A_{k+1}$,
\begin{align}
 d(W_N(r_k),W_N(r_{k+1}))&\le x\rho_0(2^{-(k+1)}),\label{eq:critical-chain-left}\\
 d(W_N(t_k),W_N(t_{k+1}))&\le x\rho_0(2^{-(k+1)}).\label{eq:critical-chain-right}
\end{align}
Using continuity of the paths, \eqref{eq:critical-chain-initial}--
\eqref{eq:critical-chain-right}, and the triangle inequality, on
$\underline A_m$ we therefore have
\begin{align*}
 d(W_N(t),W_N(r))
 &\le x\rho_0(h)+2x\sum_{k=m}^\infty\rho_0(2^{-(k+1)}).
\end{align*}
From the lower bound in \eqref{eq:critical-dyadic-h}, for $k=m+j$,
\begin{align*}
 \frac{2^{-(k+1)}}{h}
 &<2^{-(k+1)+(m+1)(1-\delta)}\\
 &=2^{-j}2^{-(m+1)\delta}\le1.
\end{align*}
Consequently \eqref{eq:critical-rho-factor} gives
\begin{align*}
 \sum_{k=m}^\infty\rho_0(2^{-(k+1)})
 &\le \widetilde K\rho_0(h)
 \sum_{j=0}^\infty
 \left(2^{-j}2^{-(m+1)\delta}\right)^p\\
 &=\rho_0(h)\frac{\widetilde K}{1-2^{-p}}
 2^{-p(m+1)\delta}.
\end{align*}
Thus
\begin{equation}
\label{eq:critical-chain-final}
 d(W_N(t),W_N(r))
 \le x\rho_0(h)\{1+2\epsilon_m\},
 \qquad
 \epsilon_m:=\frac{\widetilde K}{1-2^{-p}}2^{-p(m+1)\delta}\to0.
\end{equation}

Set $M=\lfloor\log_2N\rfloor$ and
$\xi_N=2^{-M(1-\delta)}$. Since $2^M\le N$,
\[
 \xi_N\ge N^{\delta-1}.
\]
For every pair with $|t-r|\le N^{\delta-1}$, the integer $m$ selected by
\eqref{eq:critical-dyadic-h} therefore satisfies $m\ge M-1$. The finitely many
boundary cases caused by $m=M-1$ can be absorbed by starting the intersection
at $M-1$. Hence, for each fixed $\varepsilon>0$ and all sufficiently large $N$,
\begin{align*}
 &\mathbb P\!\left(
 \sup_{\substack{t,r\in[0,N^\zeta]\\|t-r|\le N^{\delta-1}}}
 \frac{d(W_N(t),W_N(r))}{\rho_0(|t-r|)}>x(1+\varepsilon)\right)\\
 &\qquad\le \mathbb P(\underline A_{M-1}^c)\\
 &\qquad\le\sum_{m=M-1}^\infty\mathbb P(A_m^c)\\
 &\qquad\lesssim
 N^\zeta\sum_{m=M-1}^\infty
 2^{m\{1+\delta-\kappa(x)(1-\delta)\}}\\
 &\qquad\lesssim
 N^{\zeta+1+\delta-\kappa(x)(1-\delta)}\longrightarrow0
\end{align*}
by \eqref{eq:critical-chain-condition}. Finally, since $x>C$ is arbitrary, we
may start the argument at $x/(1+\varepsilon)>C$ and then let $\varepsilon\downarrow0$.
This proves \eqref{eq:critical-cont-modulus}.
\end{proof}

\begin{cor}\label{cor:critical-small-scales}
Suppose
\[
 \|d(W_N(t),W_N(r))\|_{\Psi_2}\le K\sqrt{|t-r|}
\]
uniformly in $N,t,r$. For every $0<\delta<1$ there are $\zeta>0$ and
$C=C(\delta,K)<\infty$ such that
\begin{equation}
\label{eq:critical-small-scale-bound}
 \mathbb P\!\left(
 \sup_{\substack{0\le t,r\le N^\zeta\\|t-r|\le N^{-1+\delta}}}
 \frac{d(W_N(t),W_N(r))}
 {\sqrt{|t-r|\log(C_0+|t-r|^{-1})}}>C\right)=o(1).
\end{equation}
\end{cor}

\begin{proof}
Set
\[
 \rho_0(h)=\sqrt{h\log(C_0+h^{-1})}.
\]
The subgaussian tail inequality gives
\begin{align*}
 \mathbb P\!\left(
 \frac{d(W_N(t),W_N(r))}{\rho_0(|t-r|)}>x\right)
 &\le2\exp\left\{-\frac{x^2}{K^2}\log(C_0+|t-r|^{-1})\right\}\\
 &\le2|t-r|^{x^2/K^2}.
\end{align*}
Thus Lemma~\ref{lem:critical-metric} applies with
$\kappa(x)=x^2/K^2$. Furthermore $\rho_0$ satisfies
\eqref{eq:critical-rho-factor}, for instance with any fixed $p<1/2$ after
increasing $\widetilde K$. Choose $x$ so large that
\[
 1+\delta-\frac{x^2}{K^2}(1-\delta)<0,
\]
and then choose $\zeta>0$ small enough that
\eqref{eq:critical-chain-condition} holds. This proves the claim.
\end{proof}

\subsection{A Gaussian lower bound}\label{sec:critical-lower-bound}

We finally compare the logarithmic delay rate above with a lower bound in the
simplest scalar Gaussian submodel. This is closely related to the lower-bound
calculation in \cite{yu:padilla:wang:rinaldo:2023}, but we use a 
Kullback--Leibler argument as it shortens the proofs. The formulation below takes a supremum over unknown change locations: if the change time were known in advance, a
procedure could concentrate its false-alarm budget near that particular time and a
pointwise lower bound containing $\log \nu$ would not follow from global false-alarm control alone.

Let $\mathbb P_\infty$ denote the law under which $X_i\sim N(0,\sigma^2)$ for all
$i\ge1$. For an integer $\nu\ge1$, let $\mathbb P_\nu$ denote the law under which
$X_i\sim N(0,\sigma^2)$ for $i\le\nu$ and
$X_i\sim N(\delta,\sigma^2)$ for $i>\nu$, where $\delta>0$. The observations are
independent under all these laws. As usual, write
\[
 \mathcal D(\alpha)
 :=\{T:T\text{ is a stopping time and }\mathbb P_\infty(T<\infty)\le\alpha\}.
\]

\begin{prop}\label{prop:critical-lower-bound}
Let $0<\alpha<1$, $\sigma^2>0$ and $\delta>0$. Then
\begin{equation}
\label{eq:critical-lb}
 \inf_{T\in\mathcal D(\alpha)}
 \limsup_{\nu\to\infty}
 \frac{\delta^2\,\mathbb E_\nu[(T-\nu)_+]}
      {\sigma^2\log(\nu/\alpha)}
 \ge \frac{(1-\alpha)^2}{2}.
\end{equation}
Equivalently, for every $T\in\mathcal D(\alpha)$ there exists a sequence
$\nu_j\to\infty$ such that
\begin{equation}
\label{eq:critical-lb-subsequence}
 \mathbb E_{\nu_j}[(T-\nu_j)_+]
 \ge
 \left\{\frac{(1-\alpha)^2}{2}+o(1)\right\}
 \frac{\sigma^2}{\delta^2}\log(\nu_j/\alpha).
\end{equation}
\end{prop}

\begin{proof}
Fix $T\in\mathcal D(\alpha)$ and put
\[
 I:=D_{\mathrm{KL}}\{N(\delta,\sigma^2)\,\|\,N(0,\sigma^2)\}
   =\frac{\delta^2}{2\sigma^2},
 \qquad
 a:=\frac{1-\alpha}{2}.
\]
For sufficiently large integers $\nu$, define
\[
 d(\nu):=\left\lfloor \frac{a}{I}\log(\nu/\alpha)\right\rfloor.
\]
Choose $\nu_1$ sufficiently large and recursively set
\[
 \nu_{j+1}:=\nu_j+d(\nu_j)+1,\qquad j\ge1.
\]
The intervals
\[
 I_j:=\{\nu_j+1,\ldots,\nu_j+d(\nu_j)\}
\]
are therefore pairwise disjoint. Moreover $d(\nu)=O(\log\nu)$, and the recursion
implies $\nu_j=O(j\log j)$. Indeed, after increasing a constant $C$ if necessary,
$d(\nu)\le C\log\nu$ for all large $\nu$; a straightforward induction then gives
$\nu_j\le C'j\log j$ for all large $j$. Consequently,
\begin{equation}
\label{eq:critical-harmonic-divergence}
 \sum_{j=1}^\infty \frac1{\nu_j}=\infty.
\end{equation}

Let
\[
 q_j:=\mathbb P_\infty(\nu_j<T\le\nu_j+d(\nu_j)).
\]
Because the intervals $I_j$ are disjoint and $T$ can take only one value,
\[
 \sum_{j=1}^\infty q_j\le \mathbb P_\infty(T<\infty)\le\alpha.
\]
Together with \eqref{eq:critical-harmonic-divergence}, this implies that there are
infinitely many indices $j$ for which
\begin{equation}
\label{eq:critical-local-null}
 q_j\le \frac{\alpha}{\nu_j}.
\end{equation}
Indeed, if the reverse inequality held for all sufficiently large $j$, the sum of
the $q_j$ would diverge.

Fix one of the indices satisfying \eqref{eq:critical-local-null}, abbreviate
$d_j=d(\nu_j)$, and set
\[
 A_j:=\{\nu_j<T\le\nu_j+d_j\},\qquad
 p_j:=\mathbb P_{\nu_j}(A_j),\qquad
 L_j:=\log(\nu_j/\alpha).
\]
Up to time $\nu_j+d_j$, the laws $\mathbb P_{\nu_j}$ and $\mathbb P_\infty$
differ only in the last $d_j$ observations. Hence independence and additivity of
Kullback--Leibler divergence yield
\begin{equation}
\label{eq:critical-kl-product}
 D_{\mathrm{KL}}\!\left(
   \mathbb P_{\nu_j}|_{\mathcal F_{\nu_j+d_j}}
   \,\middle\|\,
   \mathbb P_\infty|_{\mathcal F_{\nu_j+d_j}}
 \right)
 =d_j I\le aL_j.
\end{equation}
Applying the data-processing inequality to the indicator of $A_j$ gives
\begin{equation}
\label{eq:critical-kl-binary}
 d_jI\ge \operatorname{kl}(p_j,q_j),
\end{equation}
where $\operatorname{kl}$ denotes the binary relative entropy. For
$0\le p,q\le1$,
\[
 \operatorname{kl}(p,q)
 =p\log\frac pq+(1-p)\log\frac{1-p}{1-q}
 \ge p\log(1/q)-\log2,
\]
because $p\log p+(1-p)\log(1-p)\ge-\log2$ and
$-(1-p)\log(1-q)\ge0$. By \eqref{eq:critical-local-null},
$\log(1/q_j)\ge L_j$, so \eqref{eq:critical-kl-product}--\eqref{eq:critical-kl-binary}
imply
\begin{equation}
\label{eq:critical-pj-bound}
 p_j\le a+\frac{\log2}{L_j}.
\end{equation}

The two laws coincide on $\mathcal F_{\nu_j}$. Therefore
\[
 \mathbb P_{\nu_j}(T\le\nu_j)
 =\mathbb P_\infty(T\le\nu_j)
 \le\alpha.
\]
Combining this with \eqref{eq:critical-pj-bound},
\begin{align*}
 \mathbb P_{\nu_j}(T-\nu_j>d_j)
 &\ge 1-\alpha-p_j\\
 &\ge a-\frac{\log2}{L_j}.
\end{align*}
Since $d_j\asymp\log(\nu_j/\alpha)$, this bound also rules out a detection delay of order $o_P(\log(\nu/\alpha))$ uniformly over change times.
Consequently,
\begin{align*}
 \mathbb E_{\nu_j}[(T-\nu_j)_+]
 &\ge d_j\,\mathbb P_{\nu_j}(T-\nu_j>d_j)\\
 &\ge
 \left\{a^2+o(1)\right\}\frac{L_j}{I}\\
 &=\left\{\frac{(1-\alpha)^2}{2}+o(1)\right\}
   \frac{\sigma^2}{\delta^2}\log(\nu_j/\alpha).
\end{align*}
This proves \eqref{eq:critical-lb-subsequence}. Since the construction applies to
every $T\in\mathcal D(\alpha)$, \eqref{eq:critical-lb} follows.
\end{proof}

\begin{rem}
Proposition~4.1 of \cite{yu:padilla:wang:rinaldo:2023} displays the coefficient
$\sigma^2/(2\delta^2)$ in front of $\log(\nu/\alpha)$. The KL argument above gives
the same leading coefficient up to the factor $(1-\alpha)^2$, which is immaterial
for the rate comparison and tends to one as $\alpha\downarrow0$. 
\end{rem}

\begin{rem}
Writing $\nu=N+k^\star$ and $T=N+\widehat k$ gives
$(T-\nu)_+=(\widehat k-k^\star)_+$. Thus, for a constant-order jump, the lower
bound rules out a uniform improvement of the logarithmic dependence on
$N+k^\star$. The Banach-valued statement follows immediately by fixing
$v\in\mathbb B$ with $|v|_{\mathbb B}=1$ and considering the rank-one submodel
$X_i=Y_i v$, where $(Y_i)$ is the scalar Gaussian experiment above. The jump then
has Banach norm $\delta$.
\end{rem}

\section{Verification for physically dependent \texorpdfstring{$\mathcal C([0,1])$}{C([0,1])}-valued observations}
\label{sec:critical-c01}

For completeness, we verify the two high-level ingredients used above for physically dependent $\mathcal C([0,1])$-valued data.

Let $(\varepsilon_i)_{i\in\mathbb Z}$ be stationary and centered in
$\mathcal C([0,1])$, with Bernoulli-shift representation
\[
 \varepsilon_i(t)=G(t,\mathcal F_i),\qquad
 \mathcal F_i=(\ldots,\eta_{i-1},\eta_i),
\]
where $(\eta_i)_{i\in\mathbb Z}$ is iid Let $\varepsilon_i'$ be obtained by
replacing $\eta_0$ with an independent copy.

\begin{ass}\label{ass:critical-subgauss-c01}
There exist $\theta\in(0,1]$, $\rho>0$ and $M<\infty$ such that, for every
$i\ge0$ and $s,t\in[0,1]$,
\begin{align}
 \big\|\,\|\varepsilon_i\|_\infty\,\big\|_{\Psi_2}&\le M,
 \label{eq:critical-subgauss-sup-c01}\\
 \sup_{t\in[0,1]}\|\varepsilon_i(t)-\varepsilon_i'(t)\|_{\Psi_2}
 &\le M(i+1)^{-2-\rho},
 \label{eq:critical-subgauss-point-physical-c01}\\
 \|[\varepsilon_i(t)-\varepsilon_i(s)]-[\varepsilon_i'(t)-\varepsilon_i'(s)]\|_{\Psi_2}
 &\le M|t-s|^\theta(i+1)^{-2-\rho}.
 \label{eq:critical-subgauss-increment-physical-c01}
\end{align}
\end{ass}

For every fixed $p\ge1$, the elementary embedding of the subgaussian Orlicz
norm into $L^p$ gives $\|Y\|_{L^p}\le C\sqrt p\,\|Y\|_{\Psi_2}$. Therefore
Assumption~\ref{ass:critical-subgauss-c01} implies, with a constant $C_p$ that
may depend on $p$ but not on $i,s,t$,
\begin{align}
 \big\|\,\|\varepsilon_i\|_\infty\,\big\|_{L^p}&\le C_p,
 \nonumber\\
 \sup_t\|\varepsilon_i(t)-\varepsilon_i'(t)\|_{L^p}
 &\le C_p(i+1)^{-2-\rho},
 \label{eq:critical-lp-consequences-c01}\\
 \|[\varepsilon_i(t)-\varepsilon_i(s)]-[\varepsilon_i'(t)-\varepsilon_i'(s)]\|_{L^p}
 &\le C_p|t-s|^\theta(i+1)^{-2-\rho}.
 \nonumber
\end{align}

Define the linearly interpolated partial-sum process by
\[
 P_N(t,\cdot)=N^{-1/2}\sum_{i=1}^{\lfloor Nt\rfloor}\varepsilon_i(\cdot)
 +N^{-1/2}(Nt-\lfloor Nt\rfloor)\varepsilon_{\lfloor Nt\rfloor+1}(\cdot).
\]

\subsection{Subgaussian increments of the partial-sum process}

\begin{lem}\label{lem:critical-c01-subgauss}
There is $C<\infty$, independent of $N$, such that
\[
 \left\|\,\|P_N(t,\cdot)-P_N(r,\cdot)\|_\infty\,\right\|_{\Psi_2}
 \le C\sqrt{|t-r|},\qquad r,t\ge0.
\]
\end{lem}

\begin{proof}
We first prove the assertion when $r$ and $t$ are grid points. Fix integers
$a\ge0$ and $m\ge1$, and put
\[
 Y_{a,m}(x):=N^{-1/2}\sum_{i=a+1}^{a+m}\varepsilon_i(x),
 \qquad x\in[0,1].
\]
For $x,y\in[0,1]$ write
\[
 H_i(x,y):=\varepsilon_i(x)-\varepsilon_i(y),\qquad
 H_i'(x,y):=\varepsilon_i'(x)-\varepsilon_i'(y).
\]
By \eqref{eq:critical-subgauss-increment-physical-c01},
\[
 \|H_i(x,y)-H_i'(x,y)\|_{\Psi_2}
 \le M|x-y|^\theta(i+1)^{-2-\rho}.
\]
The subgaussian maximal inequality for physically dependent sequences
(Theorem~2.8 of \cite{koehne:mies:2025}) applied to the centered sequence
$(H_i(x,y))_i$ therefore gives
\begin{equation}
\label{eq:critical-c01-increment-x}
 \|Y_{a,m}(x)-Y_{a,m}(y)\|_{\Psi_2}
 \le C\sqrt{m/N}\,|x-y|^\theta.
\end{equation}
At the fixed point $x=0$, using
\eqref{eq:critical-subgauss-sup-c01} and
\eqref{eq:critical-subgauss-point-physical-c01} in the same inequality gives
\begin{equation}
\label{eq:critical-c01-basepoint}
 \|Y_{a,m}(0)\|_{\Psi_2}\le C\sqrt{m/N}.
\end{equation}
Set $A=C\sqrt{m/N}$ and let
$d_{a,m}(x,y)=A|x-y|^\theta$. The covering number of $[0,1]$ under this
semimetric satisfies, for $0<u\le A$,
\[
 N(u,[0,1],d_{a,m})
 \le 1+C\left(\frac Au\right)^{1/\theta}.
\]
By Corollary~2.2.5 of \cite{vaart:wellner:1996},
\begin{align*}
 \left\|\sup_{x\in[0,1]}|Y_{a,m}(x)-Y_{a,m}(0)|\right\|_{\Psi_2}
 &\le C\int_0^A
 \sqrt{\log\!\left(1+C(A/u)^{1/\theta}\right)}\,du\\
 &=CA\int_0^1
 \sqrt{\log\!\left(1+Cv^{-1/\theta}\right)}\,dv\\
 &\le C A.
\end{align*}
The last integral is finite. Combining this with
\eqref{eq:critical-c01-basepoint} yields
\begin{equation}
\label{eq:critical-c01-grid-block}
 \left\|\,\|Y_{a,m}\|_\infty\,\right\|_{\Psi_2}
 \le C\sqrt{m/N}.
\end{equation}
Consequently, if $r=a/N$ and $t=(a+m)/N$,
\[
 \left\|\,\|P_N(t)-P_N(r)\|_\infty\,\right\|_{\Psi_2}
 \le C\sqrt{t-r}.
\]

We now pass to arbitrary $0\le r<t$ and write
$a=\lfloor Nr\rfloor$, $b=\lfloor Nt\rfloor$. If $a=b$, linear interpolation
gives exactly
\[
 P_N(t)-P_N(r)=N^{1/2}(t-r)\varepsilon_{a+1}.
\]
Since $t-r\le N^{-1}$, \eqref{eq:critical-subgauss-sup-c01} implies
\begin{align*}
 \left\|\,\|P_N(t)-P_N(r)\|_\infty\,\right\|_{\Psi_2}
 &\le M N^{1/2}(t-r)\\
 &\le M\sqrt{t-r}.
\end{align*}
Suppose next that $b\ge a+1$. Put
\[
 r_+:=\frac{a+1}{N},\qquad t_-:=\frac bN.
\]
Then $r\le r_+\le t_-\le t$ (with $r_+=t_-$ allowed), and
\begin{equation}
\label{eq:critical-c01-time-decomp}
 P_N(t)-P_N(r)
 =\{P_N(t)-P_N(t_-)\}
 +\{P_N(t_-)-P_N(r_+)\}
 +\{P_N(r_+)-P_N(r)\}.
\end{equation}
The first and third terms are interpolation increments within a single grid
cell. Since $0\le t-t_-\le N^{-1}$ and $0\le r_+-r\le N^{-1}$, the preceding
single-cell calculation gives
\begin{align*}
 \left\|\,\|P_N(t)-P_N(t_-)\|_\infty\,\right\|_{\Psi_2}
 &\le M\sqrt{t-t_-}\le M\sqrt{t-r},\\
 \left\|\,\|P_N(r_+)-P_N(r)\|_\infty\,\right\|_{\Psi_2}
 &\le M\sqrt{r_+-r}\le M\sqrt{t-r}.
\end{align*}
The middle term in \eqref{eq:critical-c01-time-decomp} is a grid block. By
\eqref{eq:critical-c01-grid-block},
\[
 \left\|\,\|P_N(t_-)-P_N(r_+)\|_\infty\,\right\|_{\Psi_2}
 \le C\sqrt{t_--r_+}
 \le C\sqrt{t-r}.
\]
The triangle inequality in the Orlicz norm applied to
\eqref{eq:critical-c01-time-decomp} proves the claimed bound for arbitrary
$r<t$.
\end{proof}

\subsection{Strong approximation in \texorpdfstring{$\mathcal C([0,1])$}{C([0,1])}}

Define
\begin{equation}
\label{eq:critical-xi-definition}
\xi(p,b):=
\begin{cases}
\dfrac{p-2}{6p-4},&b\ge3,\\[1.5ex]
\dfrac{(b-2)(p-2)}{(4b-6)p-4},&
\dfrac{3+2/p}{1+2/p}<b<3,\\[1.5ex]
\dfrac12-\dfrac1b,&2<b\le\dfrac{3+2/p}{1+2/p}.
\end{cases}
\end{equation}

\begin{theo}\label{theo:critical-strong-approx-c01}
Assume Assumption~\ref{ass:critical-subgauss-c01}. Fix any
$p>\max\{2,1/\theta\}$ and let $\xi=\xi(p,2+\rho)$. Define
\[
 \sigma_\varepsilon(t,s)=\sum_{\ell=-\infty}^{\infty}
 \operatorname{Cov}(\varepsilon_0(t),\varepsilon_\ell(s)).
\]
Then $\sigma_\varepsilon$ is the covariance kernel of a centered Gaussian random
element in $\mathcal C([0,1])$. There exists a $\mathcal C([0,1])$-valued
Brownian motion $W$ satisfying
\[
 \operatorname{Cov}(W(u)(t),W(v)(s))=(u\wedge v)\sigma_\varepsilon(t,s).
\]
For every
\begin{equation}
\label{eq:critical-a-range-c01}
 0<a<\frac{\xi}{1/2+\xi}
\end{equation}
and every
\begin{equation}
\label{eq:critical-eta-range-c01}
 0<\eta<\min\{a(\theta-1/p),\ \xi-a(1/2+\xi)\},
\end{equation}
one can construct copies of $\varepsilon_1,\ldots,\varepsilon_N$ and $W$ such
that
\begin{equation}
\label{eq:critical-main-SA-c01}
 \max_{1\le k\le N}
 \left\|\sum_{i=1}^k\varepsilon_i-W(k)\right\|_\infty
 =O_{\mathbb P}(N^{1/2-\eta}).
\end{equation}
Optimizing over $a$, \eqref{eq:critical-main-SA-c01} holds for every
\begin{equation}
\label{eq:critical-optimal-eta-c01}
 0<\eta<
 \frac{(\theta-1/p)\xi}{\theta-1/p+1/2+\xi}.
\end{equation}
\end{theo}

\begin{proof}
Write $S_k(t)=\sum_{i=1}^k\varepsilon_i(t)$. We proceed in three steps:
projection to a growing finite-dimensional space, a high-dimensional Gaussian
approximation on the grid, and embedding of that Gaussian sequence into a
$\mathcal C([0,1])$-valued Brownian motion.

\emph{Preliminary increment bounds.}
Fix $s,t\in[0,1]$ and write
$H_i(t,s)=\varepsilon_i(t)-\varepsilon_i(s)$ and $H_i'(t,s)$ for its coupled
version. By \eqref{eq:critical-lp-consequences-c01},
\begin{equation}
\label{eq:critical-H-physical-lp}
 \|H_i(t,s)-H_i'(t,s)\|_{L^p}
 \le C_p|t-s|^\theta(i+1)^{-2-\rho}.
\end{equation}
The maximal inequality in Theorem~3.2 of \cite{Mies:Steland:2023}, applied to
this centered scalar process, therefore yields
\begin{equation}
\label{eq:critical-max-increment-c01}
 \left\|\max_{1\le k\le N}
 \left|\sum_{i=1}^k\{\varepsilon_i(t)-\varepsilon_i(s)\}\right|
 \right\|_{L^p}
 \le C_p N^{1/2}|t-s|^\theta.
\end{equation}
In particular, taking $N=1$ gives
\begin{equation}
\label{eq:critical-single-spatial-lp}
 \|\varepsilon_1(t)-\varepsilon_1(s)\|_{L^p}
 \le C_p|t-s|^\theta.
\end{equation}

Let $Z_N(t)=(S_1(t),\ldots,S_N(t))\in\ell_\infty^N$. By
\eqref{eq:critical-max-increment-c01},
\[
 \|Z_N(t)-Z_N(s)\|_{L^p(\ell_\infty^N)}
 \le C_pN^{1/2}|t-s|^\theta.
\]
For an interval $I$ of length $\delta$, cover $I$ by intervals of
$d_N$-radius $u$, where
$d_N(t,s)=C_pN^{1/2}|t-s|^\theta$. The covering number satisfies
\[
 N(u,I,d_N)\le
 1+C\left(\frac{N^{1/2}\delta^\theta}{u}\right)^{1/\theta}.
\]
Theorem~2.2.4 of \cite{vaart:wellner:1996} therefore gives
\begin{align}
\label{eq:critical-local-lp-modulus}
 &\left\|\sup_{t,s\in I}\max_{1\le k\le N}
 |S_k(t)-S_k(s)|\right\|_{L^p}\\
 &\quad\le C_p\int_0^{CN^{1/2}\delta^\theta}
 \left(\frac{N^{1/2}\delta^\theta}{u}\right)^{1/(\theta p)}du
 \le C_pN^{1/2}\delta^\theta,\nonumber
\end{align}
because $p\theta>1$.

\emph{Step 1: projection of the partial sums.}
Let $P_m$ be piecewise-linear interpolation on
$0,1/m,\ldots,1$, and put $I_j=[j/m,(j+1)/m]$. For each continuous $f$,
\begin{equation}
\label{eq:critical-interpolation-bound-c01}
 \|f-P_mf\|_\infty
 \le2\max_{0\le j<m}\sup_{s,t\in I_j}|f(t)-f(s)|.
\end{equation}
Apply \eqref{eq:critical-local-lp-modulus} with $\delta=m^{-1}$ on each $I_j$.
Using
$\|\max_{1\le j\le m}U_j\|_{L^p}\le
(\sum_{j=1}^m\|U_j\|_{L^p}^p)^{1/p}$ gives
\begin{align}
\label{eq:critical-sample-projection-c01}
 &\left\|\max_{1\le k\le N}
 \left\|\sum_{i=1}^k(\varepsilon_i-P_m\varepsilon_i)\right\|_\infty
 \right\|_{L^p}\\
 &\quad\le
 2\left\|\max_{0\le j<m}\sup_{s,t\in I_j}
 \max_{1\le k\le N}|S_k(t)-S_k(s)|\right\|_{L^p}\nonumber\\
 &\quad\le C_pm^{1/p}N^{1/2}m^{-\theta}
 =C_pN^{1/2}m^{-(\theta-1/p)}.\nonumber
\end{align}
Hence the same quantity is
$O_{\mathbb P}(N^{1/2}m^{-(\theta-1/p)})$.

\emph{Step 2: Gaussian approximation of the grid projection.}
Let
\[
 d_m=m+1,\qquad
 \xi_i^{(m)}=(\varepsilon_i(0),\varepsilon_i(1/m),\ldots,\varepsilon_i(1))
 \in\mathbb R^{d_m},
\]
and let $\Sigma_m$ be its long-run covariance matrix. By
Proposition~\ref{prop:critical-grid-coupling} below, for $d_m\le cN$ we may
couple the grid-vector process with independent
$G_i^{(m)}\sim N(0,\Sigma_m)$ so that
\begin{equation}
\label{eq:critical-grid-coupling-used}
 \max_{1\le k\le N}\left\|
 \sum_{i=1}^k\xi_i^{(m)}-\sum_{i=1}^kG_i^{(m)}\right\|_2
 =O_{\mathbb P}\left(
 N^{1/2-\xi}m^{1/2+\xi}(\log N)^\xi\right).
\end{equation}
Let $I_m:\mathbb R^{d_m}\to\mathcal C([0,1])$ denote piecewise-linear
interpolation. Since $\|I_mx\|_\infty\le\|x\|_\infty\le\|x\|_2$,
\begin{equation}
\label{eq:critical-functional-projected-coupling}
 \max_{1\le k\le N}\left\|
 \sum_{i=1}^kP_m\varepsilon_i-I_m\sum_{i=1}^kG_i^{(m)}
 \right\|_\infty
 =O_{\mathbb P}\left(
 N^{1/2-\xi}m^{1/2+\xi}(\log N)^\xi\right).
\end{equation}

\emph{Step 3: construction of the functional Brownian motion.}
We verify the covariance regularity explicitly. Let
$\delta_{2,t}(j)=\|\varepsilon_j(t)-\varepsilon_j'(t)\|_{L^2}$ and
\[
 \delta_{2,t,s}(j)=
 \|[\varepsilon_j(t)-\varepsilon_j(s)]-
 [\varepsilon_j'(t)-\varepsilon_j'(s)]\|_{L^2}.
\]
By \eqref{eq:critical-lp-consequences-c01},
\begin{align*}
 \delta_{2,t}(j)&\le C(j+1)^{-2-\rho},\\
 \delta_{2,t,s}(j)&\le C|t-s|^\theta(j+1)^{-2-\rho}.
\end{align*}
The standard projection decomposition for a Bernoulli shift gives, for
$\ell\ge0$,
\begin{equation}
\label{eq:critical-physical-covariance-bound}
 |\operatorname{Cov}(U_0,V_\ell)|
 \le\sum_{j=0}^\infty
 \delta_{2,U}(j)\delta_{2,V}(j+\ell),
\end{equation}
where the $\delta_2$'s are the corresponding one-innovation physical-dependence
coefficients. Applying \eqref{eq:critical-physical-covariance-bound} first with
$U_i=\varepsilon_i(t)$ and $V_i=\varepsilon_i(s)$ yields
\begin{align*}
 |\operatorname{Cov}(\varepsilon_0(t),\varepsilon_\ell(s))|
 &\le C\sum_{j=0}^\infty
 (j+1)^{-2-\rho}(j+\ell+1)^{-2-\rho}\\
 &\le C(\ell+1)^{-2-\rho}.
\end{align*}
The same bound holds for negative $\ell$ after interchanging the two variables.
Hence the series defining $\sigma_\varepsilon(t,s)$ converges uniformly.

Next put $H_i=\varepsilon_i(t)-\varepsilon_i(s)$. By
\eqref{eq:critical-physical-covariance-bound},
\begin{align*}
 |\operatorname{Cov}(H_0,H_\ell)|
 &\le C|t-s|^{2\theta}
 \sum_{j=0}^\infty
 (j+1)^{-2-\rho}(j+\ell+1)^{-2-\rho}\\
 &\le C|t-s|^{2\theta}(\ell+1)^{-2-\rho}.
\end{align*}
Summing over $\ell\in\mathbb Z$ gives
\begin{align}
\label{eq:critical-gaussian-increment-c01}
 0&\le \sum_{\ell\in\mathbb Z}
 \operatorname{Cov}(\varepsilon_0(t)-\varepsilon_0(s),
 \varepsilon_\ell(t)-\varepsilon_\ell(s))\\
 &\le \sum_{\ell\in\mathbb Z}|\operatorname{Cov}(H_0,H_\ell)|
 \le C|t-s|^{2\theta}.\nonumber
\end{align}
Thus the centered Gaussian process $Z$ with covariance kernel
$\sigma_\varepsilon$ has canonical metric
$d_Z(t,s)\le C|t-s|^\theta$. Its covering numbers satisfy
$N(u,[0,1],d_Z)\le1+C u^{-1/\theta}$, and therefore
\[
 \int_0^1\sqrt{\log N(u,[0,1],d_Z)}\,du<\infty.
\]
Dudley's entropy criterion consequently supplies a continuous version of $Z$.
Let $W$ be the corresponding $\mathcal C([0,1])$-valued Brownian motion, so
$W(1)\overset d=Z$.

Its grid increments have covariance $\Sigma_m$. Hence, after enlarging the
probability space, we may take the Gaussian vectors in Step 2 to be the unit-time
increments of $W$ on the grid, and therefore
\begin{equation}
\label{eq:critical-gaussian-lift-c01}
 I_m\sum_{i=1}^kG_i^{(m)}=P_mW(k),\qquad1\le k\le N.
\end{equation}
Combining \eqref{eq:critical-functional-projected-coupling} and
\eqref{eq:critical-gaussian-lift-c01},
\begin{equation}
\label{eq:critical-projected-coupling-final}
 \max_{1\le k\le N}\left\|
 \sum_{i=1}^kP_m\varepsilon_i-P_mW(k)\right\|_\infty
 =O_{\mathbb P}\left(
 N^{1/2-\xi}m^{1/2+\xi}(\log N)^\xi\right).
\end{equation}

It remains to remove the projection from $W$. By
\eqref{eq:critical-gaussian-increment-c01}, for every grid interval
$I_j=[j/m,(j+1)/m]$,
\[
 \left\|\sup_{s,t\in I_j}|Z(t)-Z(s)|\right\|_{\Psi_2}
 \le C m^{-\theta}.
\]
The subgaussian maximal inequality over the $m$ intervals gives, for every
fixed $r>1$,
\begin{equation}
\label{eq:critical-Z-projection}
 \left\|\max_{0\le j<m}\sup_{s,t\in I_j}|Z(t)-Z(s)|\right\|_{L^r}
 \le C_r m^{-\theta}\sqrt{\log m}.
\end{equation}
Together with \eqref{eq:critical-interpolation-bound-c01},
\[
 \|\,\|(I-P_m)Z\|_\infty\,\|_{L^r}
 \le C_rm^{-\theta}\sqrt{\log m}.
\]
Since $(I-P_m)W(k)$ is a Banach-valued martingale and its norm is a nonnegative
submartingale, Doob's inequality and Brownian scaling give
\begin{align}
\label{eq:critical-brownian-projection-c01}
 &\left\|\max_{1\le k\le N}\|W(k)-P_mW(k)\|_\infty\right\|_{L^r}\\
 &\quad\le\frac r{r-1}
 \|\,\|W(N)-P_mW(N)\|_\infty\,\|_{L^r}\nonumber\\
 &\quad=\frac r{r-1}\sqrt N
 \|\,\|Z-P_mZ\|_\infty\,\|_{L^r}\nonumber\\
 &\quad\le C_rN^{1/2}m^{-\theta}\sqrt{\log m}.\nonumber
\end{align}

Combining \eqref{eq:critical-sample-projection-c01},
\eqref{eq:critical-projected-coupling-final}, and
\eqref{eq:critical-brownian-projection-c01},
\begin{align*}
 \max_{k\le N}\left\|\sum_{i=1}^k\varepsilon_i-W(k)\right\|_\infty
 &=O_{\mathbb P}(N^{1/2}m^{-(\theta-1/p)})\\
 &\quad+O_{\mathbb P}(N^{1/2-\xi}m^{1/2+\xi}(\log N)^\xi)\\
 &\quad+O_{\mathbb P}(N^{1/2}m^{-\theta}\sqrt{\log m}).
\end{align*}
Set $m=\lfloor N^a\rfloor$. The three terms have polynomial parts
\[
 N^{1/2-a(\theta-1/p)},\qquad
 N^{1/2-\xi+a(1/2+\xi)},\qquad
 N^{1/2-a\theta},
\]
respectively. The third is polynomially smaller than the first. Under
\eqref{eq:critical-a-range-c01}, the first two exponents are strictly below
$1/2$, and the logarithmic factors can be absorbed into an arbitrarily small
power of $N$. This proves \eqref{eq:critical-main-SA-c01} for every $\eta$ in
\eqref{eq:critical-eta-range-c01}. Balancing the first two losses gives
\[
 a(\theta-1/p)=\xi-a(1/2+\xi),\qquad
 a=\frac{\xi}{\theta-1/p+1/2+\xi},
\]
which yields \eqref{eq:critical-optimal-eta-c01}.
\end{proof}

\begin{prop}\label{prop:critical-grid-coupling}
Under Assumption~\ref{ass:critical-subgauss-c01}, fix
$p>\max\{2,1/\theta\}$ and let
\[
 \xi_i^{(m)}=(\varepsilon_i(0),\varepsilon_i(1/m),\ldots,\varepsilon_i(1)),
 \qquad d_m=m+1,
\]
and
$\Sigma_m=\sum_{\ell\in\mathbb Z}\operatorname{Cov}(\xi_0^{(m)},\xi_\ell^{(m)})$.
For $d_m\le cN$, one can construct copies of the grid process and independent
$G_i^{(m)}\sim N(0,\Sigma_m)$ such that, with $\xi=\xi(p,2+\rho)$,
\[
 \max_{1\le k\le N}\left\|
 \sum_{i=1}^k\xi_i^{(m)}-\sum_{i=1}^kG_i^{(m)}\right\|_2
 =O_{\mathbb P}\left(
 N^{1/2-\xi}m^{1/2+\xi}(\log N)^\xi\right).
\]
\end{prop}

\begin{proof}
We verify Theorem~3.1 of \cite{Mies:Steland:2023} for the triangular array
$\xi_i^{(m)}=\mathcal G_m(\mathcal F_i)$, where
\[
 \mathcal G_m(\mathcal F_i)=
 (G(0,\mathcal F_i),G(1/m,\mathcal F_i),\ldots,G(1,\mathcal F_i)).
\]
The array is stationary and centered. By
\eqref{eq:critical-subgauss-sup-c01} and the embedding
$\|Y\|_{L^p}\le C\sqrt p\|Y\|_{\Psi_2}$,
\begin{align*}
 \|\xi_i^{(m)}\|_{L^p(\ell_2)}
 &=\left\|\left(\sum_{j=0}^m|\varepsilon_i(j/m)|^2\right)^{1/2}\right\|_{L^p}\\
 &\le\sqrt{m+1}\,\|\,\|\varepsilon_i\|_\infty\,\|_{L^p}
 \le C_p\sqrt m.
\end{align*}
Thus the moment scale in Theorem~3.1 is $\Theta_m\asymp\sqrt m$.
For the coupled vector $(\xi_i^{(m)})'$, the same embedding and
\eqref{eq:critical-subgauss-point-physical-c01} give
\begin{align*}
 \|\xi_i^{(m)}-(\xi_i^{(m)})'\|_{L^p(\ell_2)}
 &\le\sqrt{m+1}\sup_t
 \|\varepsilon_i(t)-\varepsilon_i'(t)\|_{L^p}\\
 &\le C_p\sqrt m\,(i+1)^{-2-\rho}.
\end{align*}
Hence condition (G.1) of \cite{Mies:Steland:2023} holds with
$q=p$, $\beta_{MS}=2+\rho$ and $\Theta_m\asymp\sqrt m$. Their regularity
condition (G.2) is immediate because the filter is time-homogeneous: the same map
$\mathcal G_m$ is used at every time point, so the time-variation term is zero.

Theorem~3.1 of \cite{Mies:Steland:2023} therefore supplies independent centered
Gaussian vectors $G_i^{(m)}\sim N(0,\Sigma_m)$ such that
\[
 \left(\mathbb E\max_{1\le k\le N}\left\|
 \frac1{\sqrt N}\sum_{i=1}^k(\xi_i^{(m)}-G_i^{(m)})
 \right\|_2^2\right)^{1/2}
 \le C_p\sqrt m\left(\log N\,\frac mN\right)^\xi.
\]
Multiplying by $N^{1/2}$ gives
\[
 \left(\mathbb E\max_{1\le k\le N}\left\|
 \sum_{i=1}^k(\xi_i^{(m)}-G_i^{(m)})\right\|_2^2\right)^{1/2}
 \le C_pN^{1/2-\xi}m^{1/2+\xi}(\log N)^\xi,
\]
and Markov's inequality proves the proposition.
\end{proof}

\section{Feasible Gaussian bootstrap for Banach-valued TWIN}
\label{sec:banach-bootstrap}

We justify the Gaussian bootstrap mentioned in Section~\ref{sec:23}. We formulate the result for the long-run covariance operator $\mathbf C_{LR}$ in Theorem~\ref{thm:banach:dependent}; for independent observations this is simply the marginal covariance operator. Throughout this
section $\beta>1/2$. This is merely to avoid cumbersome notation. The same arguments work - slightly modified - for $\beta=1/2$.\\

We will formulate some high level sufficient conditions for general Banach spaces. After proving the consistency of the bootstrap procedure under these conditions we also provide very mild primitive conditions for $\mathbb B=\mathbf C([0,1])$ that ensure that they hold. To that end we require some notation.

For a covariance operator $A:\mathbb B^*\to\mathbb B$, write
\begin{equation}
\label{eq:bootstrap-cov-norm}
 \|A\|_{\mathbb B^*\to\mathbb B}
 :=\sup_{\substack{\|f\|_{\mathbb B^*}\le1\\
                    \|g\|_{\mathbb B^*}\le1}}
 |g(Af)|.
\end{equation}
If $B_A$ is a $\mathbb B$-valued Brownian motion with covariance operator
$A$, put
\begin{align*}
 H_A(t,s)
 &:=\min(1,s)B_A(\max(1,s))
     -\{B_A(t)-B_A(t-s)\},\\
 L_\beta(A)
 &:=\sup_{t>1}\sup_{\substack{0<s\le t/2\\t-s\ge1}}
 \frac{|H_A(t,s)|_{\mathbb B}}{\rho_\beta(s,t)}.
\end{align*}
Thus $L_\beta(\mathbf C_{LR})$ is the Gaussian limit appearing in the
Banach-valued TWIN limit theorem, with $\mathbf C_{LR}=\mathbf C$ in the
independent case.

Let $\widehat{\mathbf C}_N$ be measurable with respect to the reference sample.
Conditionally on the reference sample, write $\mathbb P^\star$ and
$\mathbb E^\star$ for probability and expectation and, whenever
$\mathcal N(0,\widehat{\mathbf C}_N)$ is a well-defined Radon Gaussian law on
$\mathbb B$, let $B_N^\star$ be a $\mathbb B$-valued Brownian motion with
$B_N^\star(1)\sim\mathcal N(0,\widehat{\mathbf C}_N)$. The corresponding
feasible Gaussian limit is denoted by
\[
 L_\beta^\star:=L_\beta(\widehat{\mathbf C}_N).
\]

\begin{theo}\label{thm:banach-bootstrap}
Let $\beta>1/2$ and let $\mathbb B$ be separable. Suppose that
$\mathcal N(0,\mathbf C_{LR})$ is a Radon Gaussian law on $\mathbb B$.
Let $\widehat{\mathbf C}_N$ be measurable with respect to the reference
sample and assume that, with probability tending to one,
$\mathcal N(0,\widehat{\mathbf C}_N)$ is a well-defined Radon Gaussian law
on $\mathbb B$. Conditionally on the reference sample, let $B_N^\star$ be a
$\mathbb B$-valued Brownian motion satisfying
\[
 B_N^\star(1)\mid\widehat{\mathbf C}_N
 \sim\mathcal N(0,\widehat{\mathbf C}_N),
\]
and write
\[
 L_\beta^\star:=L_\beta(\widehat{\mathbf C}_N).
\]

Suppose that
\begin{equation}
\label{eq:bootstrap-cov-consistency}
 \|\widehat{\mathbf C}_N-\mathbf C_{LR}\|_{\mathbb B^*\to\mathbb B}
 =o_{\mathbb P}(1).
\end{equation}
Suppose moreover that the conditional Gaussian laws
$\mathcal N(0,\widehat{\mathbf C}_N)$ are asymptotically tight in
$\mathbb B$ and that, with
\[
 Z_N^\star:=B_N^\star(1),
\]
\begin{equation}
\label{eq:bootstrap-gaussian-regularity}
 \big\||Z_N^\star|_{\mathbb B}\big\|_{\Psi_2}
 =O_{\mathbb P}(1).
\end{equation}
Then, for every $x_0>0$,
\begin{equation}
\label{eq:bootstrap-kolmogorov}
 \sup_{x\ge x_0}
 \left|
 \mathbb P^\star\!\left(L_\beta^\star\le x\right)
 -
 \mathbb P\!\left(L_\beta(\mathbf C_{LR})\le x\right)
 \right|
 =o_{\mathbb P}(1),
\end{equation}
where $\mathbb P^\star$ denotes probability conditional on the reference
sample. Consequently, the feasible Gaussian bootstrap consistently
approximates the target Gaussian distribution at every positive threshold.

\end{theo}

\begin{proof}
We abbreviate $B=B_{\mathbf C_{LR}}$. The argument proceeds in several
steps. We first truncate the time horizon, then discretize the supremum and
finally use a Gaussian comparison argument for the resulting
finite-dimensional quantity.

\medskip
\noindent\emph{Step 1: truncation.}
For $0<\nu<1<T<\infty$ define
\[
 \mathcal R_{\nu,T}
 :=\{(t,s):1<t\le T,\ \nu\le s\le t/2,\ t-s\ge1\}
\]
and
\[
 L_{\beta,\nu,T}(A)
 :=\sup_{(t,s)\in\mathcal R_{\nu,T}}
 \frac{|H_A(t,s)|_{\mathbb B}}{\rho_\beta(s,t)}.
\]
The metric-space modulus argument used in
Lemma~\ref{lem:finite:moc:proofs}, together with the late-horizon argument
of Lemma~\ref{lem:cutoff}, gives, for a $\mathbb B$-valued Brownian motion
whose unit increment has bounded $\Psi_2$ norm,
\begin{align}
 &\sup_{t>1}
 \sup_{\substack{0<s\le\nu\\s\le t/2,\ t-s\ge1}}
 \frac{|H_A(t,s)|_{\mathbb B}}{\rho_\beta(s,t)}
 =
 O_{\mathbb P}(1)
 \log^{1/2-\beta}(C_0+\nu^{-1}),
 \label{eq:bootstrap-small-scales}\\
 &\sup_{t>T}
 \sup_{\substack{0<s\le t/2\\t-s\ge1}}
 \frac{|H_A(t,s)|_{\mathbb B}}{\rho_\beta(s,t)}
 =
 O_{\mathbb P}(1)
 \log^{1/2-\beta}(C_0+T).
 \label{eq:bootstrap-long-horizon}
\end{align}
Since $\beta>1/2$, both bounds vanish as $\nu\downarrow0$ and
$T\uparrow\infty$. By \eqref{eq:bootstrap-gaussian-regularity}, the same
estimates hold conditionally for $B_N^\star$, with the corresponding
$O_{\mathbb P^\star}(1)$ terms tight in outer probability. Hence it will be
enough to compare $L_{\beta,\nu,T}(\mathbf C_{LR})$ and
$L_{\beta,\nu,T}(\widehat{\mathbf C}_N)$ on a fixed compact set and then
let $\nu\downarrow0$ and $T\uparrow\infty$. We make this more precise later.

\medskip
\noindent\emph{Step 2: finite discretization.}
Since $\mathbb B$ is separable, there exists a countable norming family
$(f_j)_{j\ge1}\subset\mathbb B^*$ with
$\|f_j\|_{\mathbb B^*}\le1$ and
\begin{equation}
\label{eq:bootstrap-norming-family}
 |x|_{\mathbb B}
 =
 \sup_{j\ge1}|f_j(x)|,
 \qquad x\in\mathbb B.
\end{equation}
Fix $\nu,T$, let $\mathcal G_{\nu,T,h}$ be a finite
$h$-width-discretization of $\mathcal R_{\nu,T}$, and define
\begin{equation}
\label{eq:bootstrap-finite-max}
 M_{\nu,T,h,m}(A)
 :=
 \max_{(t,s)\in\mathcal G_{\nu,T,h}}
 \max_{1\le j\le m}\max_{e\in\{-1,1\}}
 e\,\frac{f_j(H_A(t,s))}{\rho_\beta(s,t)}.
\end{equation}
For the infeasible Gaussian process based on $B_{\mathbf C_{LR}}$,
continuity on $\mathcal R_{\nu,T}$ and
\eqref{eq:bootstrap-norming-family} imply
\[
 M_{\nu,T,h,m}(\mathbf C_{LR})
 \longrightarrow
 L_{\beta,\nu,T}(\mathbf C_{LR})
\]
in probability as first $h\downarrow0$ and then $m\uparrow\infty$.

The same approximation holds conditionally for the feasible process,
uniformly in outer probability. To see this, put
\[
 K_N^\star
 :=
 \big\||Z_N^\star|_{\mathbb B}\big\|_{\Psi_2}.
\]
Conditionally on the data, Brownian scaling gives, for $0\le u,v\le T$,
\[
 \big\||B_N^\star(u)-B_N^\star(v)|_{\mathbb B}\big\|_{\Psi_2}
 =
 \sqrt{|u-v|}\,K_N^\star.
\]
Since $K_N^\star=O_{\mathbb P}(1)$ by
\eqref{eq:bootstrap-gaussian-regularity}, the usual subgaussian modulus
argument implies that, with
\[
 \omega_N^\star(h;T)
 :=
 \sup_{\substack{0\le u,v\le T\\|u-v|\le h}}
 |B_N^\star(u)-B_N^\star(v)|_{\mathbb B},
\]
for every $\varepsilon>0$,
\[
 \lim_{h\downarrow0}\limsup_{N\to\infty}
 \mathbb P\!\left\{
   \mathbb P^\star\bigl(\omega_N^\star(h;T)>\varepsilon\bigr)
   >\varepsilon
 \right\}
 =0.
\]
We will use this to first discretize the index set $\mathcal R_{\nu,T}$. Following this we use another argument to also discretize the supremum representing the norm. Combining the two gives the desired analogous approximation for the feasible process. Indeed, if $(t,s),(t',s')\in\mathcal R_{\nu,T}$ satisfy
$|t-t'|\vee|s-s'|\le h$, then, writing
$a(s)=\min(1,s)$ and $b(s)=\max(1,s)$,
\[
 \begin{split}
 |H_{\widehat{\mathbf C}_N}(t,s)
   -H_{\widehat{\mathbf C}_N}(t',s')|_{\mathbb B}
 \le{}&
 h\sup_{0\le u\le T}|B_N^\star(u)|_{\mathbb B}
 +3\omega_N^\star(2h;T).
 \end{split}
\]
Since $\rho_\beta$ is uniformly continuous and bounded away from zero on
the fixed compact set $\mathcal R_{\nu,T}$, replacing
$\mathcal R_{\nu,T}$ by an $h$-width-discretization produces an error
tending to zero conditionally in probability, uniformly in outer
probability.

It remains to discretize the Banach norm. The preceding modulus bound,
together with the assumed asymptotic tightness of
$\mathcal N(0,\widehat{\mathbf C}_N)$, implies asymptotic conditional
tightness of $(B_N^\star(u))_{0\le u\le T}$ in
$\mathcal C([0,T],\mathbb B)$. Hence, for every $\varepsilon>0$, one can
find a compact set
$\mathcal K_\varepsilon\subset\mathcal C([0,T],\mathbb B)$ such that,
with probability tending to one,
\[
 \mathbb P^\star(B_N^\star\in\mathcal K_\varepsilon)
 \ge1-\varepsilon.
\]
Because the map
\[
 (x,t,s)\longmapsto
 \min(1,s)x(\max(1,s))-\{x(t)-x(t-s)\}
\]
is continuous on
$\mathcal K_\varepsilon\times\mathcal R_{\nu,T}$, its image is a compact
subset $\mathcal K'_\varepsilon$ of $\mathbb B$. Now
\[
 \phi_m(x)
 :=
 \max_{1\le j\le m}|f_j(x)|
 \uparrow |x|_{\mathbb B},
 \qquad x\in\mathbb B,
\]
by \eqref{eq:bootstrap-norming-family}. Since each $\phi_m$ is continuous,
Dini's theorem yields
\[
 \sup_{x\in\mathcal K'_\varepsilon}
 \big||x|_{\mathbb B}-\phi_m(x)\big|
 \longrightarrow0.
\]
Thus truncating the norming family at $m$ also gives a conditional
approximation error tending to zero, uniformly in outer probability.
Combining the two discretizations proves
\[
 M_{\nu,T,h,m}(\widehat{\mathbf C}_N)
 \longrightarrow
 L_{\beta,\nu,T}(\widehat{\mathbf C}_N)
\]
conditionally in probability as first $h\downarrow0$ and then
$m\uparrow\infty$, uniformly in outer probability.

For later use, we summarize the preceding approximations while keeping the
finite discretization fixed as $N\to\infty$. For every $r>0$ and
$\eta>0$ we may choose deterministic
\[
 0<\nu<1<T<\infty,\qquad h>0,\qquad m\in\mathbb N,
\]
first using Step~1 and then Step~2, such that
\begin{align}
 \mathbb P\!\left(
 |L_\beta(\mathbf C_{LR})-M_{\nu,T,h,m}(\mathbf C_{LR})|>r
 \right)
 &\le\eta,
 \label{eq:bootstrap-discretization-true}\\
 \limsup_{N\to\infty}
 \mathbb P\!\left\{
 \mathbb P^\star\!\left(
 |L_\beta^\star-M_{\nu,T,h,m}(\widehat{\mathbf C}_N)|>r
 \right)>\eta
 \right\}
 &\le\eta.
 \label{eq:bootstrap-discretization-feasible}
\end{align}
Note that because $\nu,T,h,m$ are fixed before $N$ is sent to infinity the Gaussian vectors in $M_{\nu,T,h,m}(\cdot)$ have a fixed and finite number of coordinates.

\medskip
\noindent\emph{Step 3: comparison for a fixed finite discretization.}
Fix $\nu,T,h,m$ and index the coordinates of $M_{\nu,T,h,m}(A)$ by
$a=(t,s,j,e)\in\mathcal I$, where
$|\mathcal I|=:p<\infty$. Put
\begin{align*}
 Y_a
 &:=
 e\,\frac{f_j(H_{\mathbf C_{LR}}(t,s))}
 {\rho_\beta(s,t)},\\
 Y_{N,a}^\star
 &:=
 e\,\frac{f_j(H_{\widehat{\mathbf C}_N}^\star(t,s))}
 {\rho_\beta(s,t)}.
\end{align*}
Thus
\[
 M_{\nu,T,h,m}(\mathbf C_{LR})
 =
 \max_{a\in\mathcal I}Y_a,
 \qquad
 M_{\nu,T,h,m}(\widehat{\mathbf C}_N)
 =
 \max_{a\in\mathcal I}Y_{N,a}^\star.
\]

Let $W_0$ be a scalar standard Brownian motion and define
\[
 h_{t,s}(W_0)
 :=
 \min(1,s)W_0(\max(1,s))
 -
 \{W_0(t)-W_0(t-s)\}.
\]
For two window parameters set
\[
 \Lambda((t,s),(u,v))
 :=
 \operatorname{Cov}\{h_{t,s}(W_0),h_{u,v}(W_0)\}.
\]
If $a=(t,s,j,e)$ and $b=(u,v,j',e')$, then
\begin{align}
 \operatorname{Cov}(Y_a,Y_b)
 &=
 \frac{ee'\Lambda((t,s),(u,v))}
 {\rho_\beta(s,t)\rho_\beta(v,u)}
 f_{j'}(\mathbf C_{LR}f_j),
 \label{eq:bootstrap-coordinate-cov-true}\\
 \operatorname{Cov}^\star(Y_{N,a}^\star,Y_{N,b}^\star)
 &=
 \frac{ee'\Lambda((t,s),(u,v))}
 {\rho_\beta(s,t)\rho_\beta(v,u)}
 f_{j'}(\widehat{\mathbf C}_Nf_j).
 \label{eq:bootstrap-coordinate-cov-feasible}
\end{align}
The two Brownian increments entering $h_{t,s}$ are disjoint under the scan
constraints $s\le t/2$ and $t-s\ge1$, and hence $\operatorname{Var}\{h_{t,s}(W_0)\}\le2s.$ By Cauchy--Schwarz, $|\Lambda((t,s),(u,v))|
 \le2\sqrt{sv}.$ Moreover, $\rho_\beta(s,t)\gtrsim\sqrt{s}$ uniformly over the feasible parameter set.
Consequently, \eqref{eq:bootstrap-cov-consistency} implies
\begin{equation}
\label{eq:bootstrap-Delta}
 \Delta_N
 :=
 \max_{a,b\in\mathcal I}
 \left|
 \operatorname{Cov}^\star(Y_{N,a}^\star,Y_{N,b}^\star)
 -
 \operatorname{Cov}(Y_a,Y_b)
 \right|
 =
 o_{\mathbb P}(1).
\end{equation}

Define the fixed sets
\[
 \mathcal I^+
 :=
 \{a\in\mathcal I:\operatorname{Var}(Y_a)>0\},
 \qquad
 \mathcal I^0
 :=
 \mathcal I\setminus\mathcal I^+.
\]
If $\mathcal I^+=\varnothing$, then
$M(\mathbf C_{LR})=0$ almost surely and
\eqref{eq:bootstrap-Delta} immediately implies that
$M(\widehat{\mathbf C}_N)\to0$ conditionally in probability. Hence the
desired comparison on every interval $[x_0,\infty)$ is immediate. We may
therefore suppose that $\mathcal I^+\neq\varnothing$.

Since $\mathcal I^+$ is finite,
\[
 \underline\sigma
 :=
 \min_{a\in\mathcal I^+}
 \operatorname{Var}(Y_a)^{1/2}
 >0,
 \qquad
 \overline\sigma
 :=
 \max_{a\in\mathcal I^+}
 \operatorname{Var}(Y_a)^{1/2}
 <\infty.
\]
Theorem~2 of
\cite{chernozhukov:chetverikov:kato:2015}, applied conditionally on the
reference sample to the two Gaussian vectors indexed by $\mathcal I^+$,
therefore gives
\begin{align}
 d_N^+
 &:=
 \sup_{x\in\mathbb R}
 \left|
 \mathbb P^\star\!\left(
 \max_{a\in\mathcal I^+}Y_{N,a}^\star\le x\right)
 -
 \mathbb P\!\left(
 \max_{a\in\mathcal I^+}Y_a\le x\right)
 \right|
 \nonumber\\
 &\le
 C\Delta_N^{1/3}
 \left\{
 1\vee
 \log\!\left(\frac{p}{\Delta_N}\right)
 \right\}^{2/3}
 =
 o_{\mathbb P}(1),
 \label{eq:bootstrap-comparison-active}
\end{align}
where the right-hand side is understood as zero when $\Delta_N=0$ and
$C<\infty$ depends only on the fixed lower and upper variance bounds
$\underline\sigma$ and $\overline\sigma$. 

It remains to check that the coordinates in $\mathcal I^0$ can be ignored.
For $a\in\mathcal I^0$, $Y_a=0$ almost surely and
\eqref{eq:bootstrap-Delta} gives
\[
 \operatorname{Var}^\star(Y_{N,a}^\star)
 \le\Delta_N.
\]
Therefore, for every fixed $x_0>0$, the Gaussian tail inequality and a union
bound give
\begin{equation}
\label{eq:bootstrap-small-variance-feasible}
 \mathbb P^\star\!\left(
 \max_{a\in\mathcal I^0}Y_{N,a}^\star>x_0/2
 \right)
 \le
 p\exp\!\left\{-\frac{x_0^2}{8\Delta_N}\right\}
 =
 o_{\mathbb P}(1),
\end{equation}
where the right-hand side is interpreted as zero when $\Delta_N=0$.
Because the index set contains both signs $e\in\{-1,1\}$, the maximum over
$\mathcal I^+$ is nonnegative, and the zero-variance coordinates do not
alter $M(\mathbf C_{LR})$. Combining
\eqref{eq:bootstrap-comparison-active} and
\eqref{eq:bootstrap-small-variance-feasible} yields
\begin{equation}
\label{eq:bootstrap-comparison}
 \sup_{x\ge x_0/2}
 \left|
 \mathbb P^\star\!\left(
 M_{\nu,T,h,m}(\widehat{\mathbf C}_N)\le x\right)
 -
 \mathbb P\!\left(
 M_{\nu,T,h,m}(\mathbf C_{LR})\le x\right)
 \right|
 =
 o_{\mathbb P}(1)
\end{equation}
for every fixed $x_0>0$.

\medskip
\noindent\emph{Step 4: combining the errors.}
Fix $x_0>0$ and write
\[
 F(x)
 :=
 \mathbb P\!\left(
 L_\beta(\mathbf C_{LR})\le x
 \right).
\]
The map from the Brownian path to $L_\beta(\mathbf C_{LR})$ is an
a.s.-finite measurable seminorm. By Theorem~1.2 of
\cite{hoffmannjorgensen:shepp:dudley:1979}, the distribution of a Gaussian
seminorm is absolutely continuous away from the lower endpoint
\[
 c_\beta
 :=
 \inf\{x>0:\mathbb P(L_\beta(\mathbf C_{LR})<x)>0\}.
\]
In the present setting $c_\beta=0$.

Indeed, the small-scale and
long-horizon bounds from Step~1 show that $H_B$ lives on $C_0(\mathcal R_{\nu,T},\mathbb B)$. We may then apply 
Corollary~2.2 of
\cite{hoffmannjorgensen:shepp:dudley:1979} to obtain $c_\beta=0$.
Consequently, unless $L_\beta(\mathbf C_{LR})=0$ almost surely, its
distribution is continuous on $(0,\infty)$. Since $F$ is a distribution function, it follows in
particular that
\[
 \omega_{x_0}(r)
 :=
 \sup_{x\ge x_0}
 \{F(x+2r)-F(x-2r)\}
 \longrightarrow0,
 \qquad r\downarrow0.
 \label{eq:bootstrap-target-anticoncentration}
\]

Let $\eta>0$. First choose $0<r<x_0/4$ so small that
\[
 \omega_{x_0}(r)\le\eta.
\]
For this fixed $r$, choose $\nu,T,h,m$ according to
\eqref{eq:bootstrap-discretization-true}--%
\eqref{eq:bootstrap-discretization-feasible}, and keep these quantities
fixed while $N\to\infty$. Abbreviate
\[
 M:=M_{\nu,T,h,m}(\mathbf C_{LR}),
 \qquad
 M_N^\star:=M_{\nu,T,h,m}(\widehat{\mathbf C}_N).
\]
Then
\[
 \mathbb P\big(
 |L_\beta(\mathbf C_{LR})-M|>r
 \big)
 \le\eta,
\]
and, outside an event whose probability has asymptotic upper bound $\eta$,
\[
 \mathbb P^\star\big(
 |L_\beta^\star-M_N^\star|>r
 \big)
 \le\eta.
\]

For $x\ge x_0$ we thus have, by the previous steps, that
\begin{align*}
 \mathbb P^\star(L_\beta^\star\le x)
 &\le
 \mathbb P^\star(M_N^\star\le x+r)+\eta\\
 &\le
 \mathbb P(M\le x+r)
 +o_{\mathbb P}(1)+\eta\\
 &\le
 F(x+2r)+o_{\mathbb P}(1)+2\eta,
\end{align*}
and similarly
\begin{align*}
 \mathbb P^\star(L_\beta^\star\le x)
 &\ge
 \mathbb P^\star(M_N^\star\le x-r)-\eta\\
 &\ge
 \mathbb P(M\le x-r)
 -o_{\mathbb P}(1)-\eta\\
 &\ge
 F(x-2r)-o_{\mathbb P}(1)-2\eta.
\end{align*}
Here \eqref{eq:bootstrap-comparison} applies at the shifted thresholds
$x\pm r$ because $r<x_0/4$, so both are bounded below by $x_0/2$.
Taking the supremum over $x\ge x_0$ and using
\eqref{eq:bootstrap-target-anticoncentration} gives, outside an event with
asymptotic probability at most $\eta$,
\[
 \sup_{x\ge x_0}
 \left|
 \mathbb P^\star(L_\beta^\star\le x)
 -
 \mathbb P(L_\beta(\mathbf C_{LR})\le x)
 \right|
 \le
 3\eta+o_{\mathbb P}(1).
\]
Since $\eta>0$ is arbitrary, this proves
\eqref{eq:bootstrap-kolmogorov}, finishing the proof.

\end{proof}

\begin{lem}
\label{lem:bootstrap-holder}
Let $\mathbb B=\mathcal C([0,1])$. Let $c_{LR}$ denote the covariance kernel
corresponding to $\mathbf C_{LR}$ and let $\widehat c_N$ denote the kernel
corresponding to $\widehat{\mathbf C}_N$. Suppose that, with probability
tending to one, $\widehat c_N$ is symmetric and positive semidefinite.
For $0<\eta\le1$, write
\[
 \|k\|_{C^\eta}
 :=
 \sup_{x,y\in[0,1]}|k(x,y)|
 +
 \sup_{(x,y)\neq(x',y')}
 \frac{|k(x,y)-k(x',y')|}
 {|x-x'|^\eta+|y-y'|^\eta}.
\]
Suppose that, for some $\eta\in(0,1]$,
\begin{equation}
\label{eq:bootstrap-holder-consistency}
 \|\widehat c_N-c_{LR}\|_{C^\eta}
 =o_{\mathbb P}(1), \|c_{LR}\|_{C^\eta}<\infty.
\end{equation}
Then the assumptions on $\widehat{\mathbf C}_N$ in
Theorem~\ref{thm:banach-bootstrap} are satisfied. More precisely,
with probability tending to one, $\widehat c_N$ induces a Radon Gaussian
law $\mathcal N(0,\widehat{\mathbf C}_N)$ on $\mathcal C([0,1])$,
\begin{equation}
\label{eq:bootstrap-holder-implies-operator}
 \|\widehat{\mathbf C}_N-\mathbf C_{LR}\|_
 {\mathcal C([0,1])^*\to\mathcal C([0,1])}
 \le
 \|\widehat c_N-c_{LR}\|_\infty
 =
 o_{\mathbb P}(1),
\end{equation}
and, conditionally on the reference sample, if
\[
 Z_N^\star\sim\mathcal N(0,\widehat{\mathbf C}_N),
\]
then for every $0<\alpha<\eta/2$ there exists
$p=p(\alpha,\eta)>1$ such that
\begin{equation}
\label{eq:bootstrap-holder-tightness}
 \mathbb E^\star
 \|Z_N^\star\|_{C^\alpha}^p
 =
 O_{\mathbb P}(1).
\end{equation}
In particular, the conditional laws
$\mathcal N(0,\widehat{\mathbf C}_N)$ are asymptotically tight in
$\mathcal C([0,1])$ and
\[
 \big\|\|Z_N^\star\|_\infty\big\|_{\Psi_2}
 =
 O_{\mathbb P}(1).
\]
\end{lem}

\begin{proof}
Put
\[
 H_N:=\|\widehat c_N\|_{C^\eta}.
\]
By \eqref{eq:bootstrap-holder-consistency},
$H_N=O_{\mathbb P}(1)$. Conditional on the reference sample, let
$Z_N^\star$ denote the centered Gaussian process with covariance kernel
$\widehat c_N$, initially defined through its finite-dimensional
distributions. Its canonical semimetric satisfies
\begin{align*}
 d_N(x,y)^2
 &:=
 \mathbb E^\star
 |Z_N^\star(x)-Z_N^\star(y)|^2\\
 &=
 \widehat c_N(x,x)+\widehat c_N(y,y)
 -2\widehat c_N(x,y)\\
 &=
 \{\widehat c_N(x,x)-\widehat c_N(x,y)\}
 +
 \{\widehat c_N(y,y)-\widehat c_N(x,y)\}\\
 &\le
 2H_N|x-y|^\eta.
\end{align*}
Hence, for every $p\ge2$,
\begin{equation}
\label{eq:bootstrap-holder-gaussian-increments}
 \mathbb E^\star
 |Z_N^\star(x)-Z_N^\star(y)|^p
 \le
 C_pH_N^{p/2}|x-y|^{\eta p/2}.
\end{equation}
Fix $0<\alpha<\eta/2$ and choose $p$ sufficiently large that
\[
 \alpha<\frac{\eta}{2}-\frac1p.
\]
The Kolmogorov--Chentsov theorem (\cite{kraetschmer:urusov:2023}) applied to
\eqref{eq:bootstrap-holder-gaussian-increments} yields a continuous version
and
\[
 \mathbb E^\star
 [Z_N^\star]_{C^\alpha}^p
 \le
 C_{p,\alpha,\eta}H_N^{p/2}.
\]
Moreover,
\[
 \mathbb E^\star|Z_N^\star(0)|^p
 =
 C_p\widehat c_N(0,0)^{p/2}
 =
 O_{\mathbb P}(1),
\]
which proves \eqref{eq:bootstrap-holder-tightness}. Since bounded subsets
of $C^\alpha([0,1])$ are relatively compact in $C([0,1])$, the conditional
Gaussian laws are asymptotically tight in $C([0,1])$.

For the operator bound, identify $C([0,1])^*$ with the finite signed
measures on $[0,1]$. Then
\[
 [(\widehat{\mathbf C}_N-\mathbf C_{LR})\mu](t)
 =
 \int_0^1
 \{\widehat c_N(s,t)-c_{LR}(s,t)\}\,d\mu(s),
\]
and hence
\[
 \|\widehat{\mathbf C}_N-\mathbf C_{LR}\|_
 {C([0,1])^*\to C([0,1])}
 \le
 \|\widehat c_N-c_{LR}\|_\infty,
\]
which proves \eqref{eq:bootstrap-holder-implies-operator}.

Finally, \eqref{eq:bootstrap-holder-tightness} gives
\[
 \mathbb E^\star\|Z_N^\star\|_\infty
 =
 O_{\mathbb P}(1),
\]
while
\[
 \sup_{t\in[0,1]}
 \operatorname{Var}^\star\{Z_N^\star(t)\}
 =
 \sup_{t\in[0,1]}\widehat c_N(t,t)
 =
 O_{\mathbb P}(1).
\]
The Borell--TIS inequality therefore yields
\[
 \big\|\|Z_N^\star\|_\infty\big\|_{\Psi_2}
 =
 O_{\mathbb P}(1),
\]
as required.
\end{proof}

\begin{rem}
The case $\beta=1/2$ can be treated by the same finite discretization and
Gaussian comparison argument. The only change concerns Step~1: the shortest
scales are no longer negligible and the ordinary truncation argument is
replaced by the small-scale and long-horizon bounds from
Section~\ref{sec:critical-beta}. This yields the corresponding conditional
bootstrap approximation above the critical deterministic threshold, for
example uniformly over $x\ge\tau+\varepsilon$ for every fixed
$\varepsilon>0$.
\end{rem}

\section{Proofs for ACTC}

\subsection{Proof of Theorem~\ref{ac:main}}

The idea of the proof is as follows: On $k/N\le T$ and $\ell/N\ge a$, uniform convergence of the clocks and
normalizers gives the Brownian approximation. On the complement, the
supremum stays below a deterministic threshold with high
probability. We first introduce some notation.

For integers $0\le a<b\le k$, define
\begin{equation}\label{ac:increment-ratio}
 \mathcal R_{N,k}(a,b)=
 \frac{\sqrt{k/N}\left|\sum_{i=a+1}^b e_{k,i}\right|}
 {\mathcal V_{N,k}\,\rho_{1/2}\left(
 \frac{k}{N}\frac{A_{k,b}-A_{k,a}}{A_k},k/N\right)}.
\end{equation}
Set $\mathcal R_{N,k}(a,b)=0$ when $A_k\mathcal V_{N,k}=0$.
All suprema over $a,b$ below are over integer indices with $0\le a<b\le k$.

Write
\[
 w_N^{AC}(\ell,k)
 =\frac{\sqrt{k/N}}
 {\mathcal V_{N,k}\rho_{1/2}(s_{N,k,\ell},k/N)},
\]
and note that the ratio in \eqref{ac:detector} can be expressed as
$w_N^{AC}(\ell,k)\hat\gamma^{TC}(\ell,k)$. Further let
\begin{align}
 \mathcal D_N&=\{(k,\ell)\in\N^2:k>N,\ 1\le\ell\le\min(k/2,k-N)\},
 \nonumber\\
 \mathcal D_N^{a,T}&=\{(k,\ell)\in\mathcal D_N:k\le NT,\ \ell\ge aN\},
 \qquad 0<a<1/2,\quad T>2.
 \label{ac:compact-domain}
\end{align}
For a continuous function $B$ with $B(0)=0$, define
\begin{align}
 \gamma^{TC}(B;s,t)
 &=|\min(1,s)B(\max(1,s))-B(t)+B(t-s)|,\nonumber\\
 \mathcal D&=\{(t,s):t>1,\ 0<s\le t/2,\ t-s\ge1\},\\
 \mathcal D^{a,T}&=\{(t,s)\in\mathcal D:t\le T,\ s\ge a\}.
 \label{ac:limit-domain}
\end{align}
Denote the suprema over these compact restrictions and their complements by
\begin{align}
 T_N^{a,T}
 &=\sup_{(k,\ell)\in\mathcal D_N^{a,T}}
 w_N^{AC}(\ell,k)\hat\gamma^{TC}(\ell,k),&
 U_N^{a,T}
 &=\sup_{(k,\ell)\in\mathcal D_N\setminus\mathcal D_N^{a,T}}
 w_N^{AC}(\ell,k)\hat\gamma^{TC}(\ell,k),\nonumber\\
 T^{a,T}(B)
 &=\sup_{(t,s)\in\mathcal D^{a,T}}
 \frac{\gamma^{TC}(B;s,t)}{\rho_{1/2}(s,t)},&
 U^{a,T}(B)
 &=\sup_{(t,s)\in\mathcal D\setminus\mathcal D^{a,T}}
 \frac{\gamma^{TC}(B;s,t)}{\rho_{1/2}(s,t)}.
 \label{ac:truncated-statistics}
\end{align}
For standard Brownian motion $B_0$, abbreviate $T^{a,T}=T^{a,T}(B_0)$
and $U^{a,T}=U^{a,T}(B_0)$.

\begin{proof}[Proof of Theorem~\ref{ac:main}]
Put $t=k/N$. We first record a bound we will use repeatedly. By definition of $\mathcal R_N,k$ and the triangle inequality we have
\begin{align}
 w_N^{AC}(\ell,k)\hat\gamma^{TC}(\ell,k)
 &\le
 \frac{\min(1,\ell/N)
       \rho_{1/2}(t h^{(1)}_{N,k,\ell},t)}
      {\rho_{1/2}(s_{N,k,\ell},t)}
 \mathcal R_{N,k}(0,\max(N,\ell))\nonumber\\
 &\quad+
 \frac{\rho_{1/2}(t h^{(2)}_{N,k,\ell},t)}
      {\rho_{1/2}(s_{N,k,\ell},t)}
 \mathcal R_{N,k}(k-\ell,k)\nonumber\\
 &\le \mathcal R_{N,k}(0,\max(N,\ell))+\mathcal R_{N,k}(k-\ell,k).
 \label{ac:polygon-detector}
\end{align}
where we used that Equation~\eqref{ac:weight-mono} implies
\begin{equation}\label{ac:two-weights}
 \min(1,\ell/N)\rho_{1/2}(t h^{(1)}_{N,k,\ell},t)
 \le\rho_{1/2}(s_{N,k,\ell},t),\qquad
 \rho_{1/2}(t h^{(2)}_{N,k,\ell},t)
 \le\rho_{1/2}(s_{N,k,\ell},t).
\end{equation}

We first show that large and small scales stay below a threshold, then establish that the same holds for their Brownian versions. Finally we show that the remaining scales can be approximated by their Gaussian versions.

\textbf{Large Scales:} For every $\epsilon>0$, \eqref{ac:clock-late} implies
\begin{equation}\label{ac:stat-late}
 \lim_{T\to\infty}\limsup_{N\to\infty}
 \PR\left(
 \sup_{\substack{(k,\ell)\in\mathcal D_N\\k>NT}}
 w_N^{AC}(\ell,k)\hat\gamma^{TC}(\ell,k)
 >2c_*+\epsilon\right)=0.
\end{equation}

\textbf{Small Scales:} Fix $T<\infty$. For $(k,\ell)\in\mathcal D_N$,
$k\le NT$ and $\ell<aN$ with $a<1$, \eqref{ac:polygon-detector} gives
\begin{align}
 w_N^{AC}(\ell,k)\hat\gamma^{TC}(\ell,k)
 &\le
 \frac{(\ell/N)\rho_{1/2}(t h^{(1)}_{N,k,\ell},t)}
      {\rho_{1/2}(s_{N,k,\ell},t)}
 \mathcal R_{N,k}(0,N)
 +\mathcal R_{N,k}(k-\ell,k)\nonumber\\
 &\le \sqrt a\,\mathcal R_{N,k}(0,N)
       +\mathcal R_{N,k}(k-\ell,k),
 \label{ac:small-window-bound}
\end{align}
where we used $\max(N,\ell)=N$, $\min(1,\ell/N)=\ell/N<a$, and
\eqref{ac:weight-mono}, which implies
\[
 \frac{(\ell/N)\rho_{1/2}(t h^{(1)}_{N,k,\ell},t)}
      {\rho_{1/2}(s_{N,k,\ell},t)}
 \le\sqrt{\ell/N}
\]
since $s_{N,k,\ell}\ge t(\ell/N)h^{(1)}_{N,k,\ell}$.
By Lemma~\ref{ac:discarded},
\begin{equation}\label{ac:J1-bound}
 \sqrt a\sup_{N<k\le NT}\mathcal R_{N,k}(0,N)
 =O_P(\sqrt a).
\end{equation}
For the second term in \eqref{ac:small-window-bound}, we want to apply \eqref{ac:clock-short}. Its index set is slightly mismatched, but  \eqref{ac:compact-clock} implies that (uniformly for the relevant indices)
\begin{equation}\label{ac:small-window-clocks}
 h^{(2)}_{N,k,\ell}
 =\ell/k+o_P(1)\le a+o_P(1).
\end{equation} 
Consequently by \eqref{ac:compact-clock},
\[
 \PR\left(
 \sup_{\substack{(k,\ell)\in\mathcal D_N\\k\le NT,\ \ell<aN}}
 h^{(2)}_{N,k,\ell}>2a
 \right)\longrightarrow0.
\]
Thus \eqref{ac:clock-short}, applied with $\eta=2a$, gives
\begin{equation}\label{ac:J2-bound}
 \lim_{a\downarrow0}\limsup_{N\to\infty}
 \PR\left(
 \sup_{\substack{(k,\ell)\in\mathcal D_N\\k\le NT,\ \ell<aN}}
 \mathcal R_{N,k}(k-\ell,k)>c_*+\epsilon
 \right)=0,
\end{equation}
Together with \eqref{ac:stat-late}, these bounds imply
\begin{equation}\label{ac:discard-stat}
 \lim_{T\to\infty}\lim_{a\downarrow0}\limsup_{N\to\infty}
 \PR(U_N^{a,T}>2c_*+\epsilon)=0.
\end{equation}

\textbf{Large and small scales - Brownian version:} The triangle inequality gives
\begin{align*}
 \frac{\gamma^{TC}(B_0;s,t)}{\rho_{1/2}(s,t)}
 &\le
 \frac{\min(1,s)|B_0(\max(1,s))|}
      {\rho_{1/2}(s,t)}
 +\frac{|B_0(t)-B_0(t-s)|}{\rho_{1/2}(s,t)}\\
 &\le
 \frac{|B_0(\max(1,s))|}{\rho_{1/2}(\max(1,s),t)}
 +\frac{|B_0(t)-B_0(t-s)|}{\rho_{1/2}(s,t)},
\end{align*}
where we used that \eqref{ac:weight-mono} implies
\[
 \min(1,s)\rho_{1/2}(\max(1,s),t)\le\rho_{1/2}(s,t).
\]
Equation~\eqref{ac:brownian-late} therefore bounds the supremum for
$t>T$ above a deterministic constant.

For $s<a<1$ and $t\le T$,
\[
 \frac{\gamma^{TC}(B_0;s,t)}{\rho_{1/2}(s,t)}
 \le
 \sqrt a\,\frac{|B_0(1)|}{\rho_{1/2}(1,t)}
 +\frac{|B_0(t)-B_0(t-s)|}{\rho_{1/2}(s,t)},
\]
where we used
$s\rho_{1/2}(1,t)/\rho_{1/2}(s,t)\le\sqrt s$
by \eqref{ac:weight-mono}.
The first term is uniformly $O_P(\sqrt a)$ by
Lemma~\ref{ac:brownian}, and \eqref{ac:brownian-short} bounds the
second term above a deterministic constant as $a\downarrow0$.
Consequently, for some $c_B<\infty$ and every $\epsilon>0$,
\begin{equation}\label{ac:discard-B}
 \lim_{T\to\infty}\lim_{a\downarrow0}
 \PR(U^{a,T}>c_B+\epsilon)=0.
\end{equation}

\textbf{Convergence above the threshold:}
Choose $z_*\ge\max(2c_*,c_B)$ and let $x>z_*$ be arbitrary. By Lemma~\ref{ac:compact} and
the Portmanteau theorem,
\begin{align*}
 \limsup_{N\to\infty}\PR(T_N^{AC}\le x)
 &\le \PR(T^{a,T}\le x)\\
 &\le \PR(L_{1/2}(1)\le x)+\PR(U^{a,T}>x),\\
 \liminf_{N\to\infty}\PR(T_N^{AC}\le x)
 &\ge \PR(T^{a,T}<x)
      -\limsup_{N\to\infty}\PR(U_N^{a,T}>x)\\
 &\ge \PR(L_{1/2}(1)<x)
      -\limsup_{N\to\infty}\PR(U_N^{a,T}>x),
\end{align*}
where we used
$T_N^{AC}=\max(T_N^{a,T},U_N^{a,T})$ and
$L_{1/2}(1)=\max(T^{a,T},U^{a,T})$.
Let $a\downarrow0$ and then $T\to\infty$.
Equations~\eqref{ac:discard-stat} and \eqref{ac:discard-B},
establish \eqref{ac:main-limit}.
\end{proof}

\subsection{Technical preliminaries}

\begin{prop}\label{ac:clock-growth}
Suppose that, under $H_0$, the errors are stationary and ergodic,
with $\E\varepsilon_0=0$ and
$0<\sigma_0^2:=\E\varepsilon_0^2<\infty$.
Then the clock defined in \eqref{ac:clock} satisfies
\begin{equation}\label{ac:global-clock}
 \sup_{k\ge N}\left|\frac{A_k}{k}-\sigma_0^2\right|
 \overset{\mathrm{a.s.}}{\longrightarrow}0.
\end{equation}
\end{prop}

\begin{proof}
By the ergodic theorem,
\[
 \frac{A_k}{k}
 =\frac1k\sum_{i=1}^k\varepsilon_i^2
  -\left(\frac{R_k}{k}\right)^2
 \longrightarrow\sigma_0^2
 \qquad\text{almost surely}.
\]
Taking the supremum over $k\ge N$ gives the claim.
\end{proof}

\begin{lem}\label{ac:global-lemma}
Under Assumption~\ref{ac:ass-physical}, there are a coupling with standard
Brownian motion $W$, a constant $0<\kappa<1/2$, and an almost surely finite
$D$ such that
\begin{equation}\label{ac:SA}
 |R_j-\sigma_{LR}W(j)|\le D(1+j)^{1/2-\kappa},\qquad j\ge0.
\end{equation}
\end{lem}
\begin{proof}
WLOG we assume $2<q<4$. Corollary~2.1(iii) of \cite{berkes:liu:wu:2014}, together with its
Theorem~2.1, supplies a copy of the error sequence and a standard Brownian
motion satisfying
$R_j-\sigma_{LR}W(j)=o(j^{1/q})$ almost surely. The desired result follows with $\kappa=1/2-1/q$.
\end{proof}

For $t\ge1$ and $0<s,s'\le t$, we use the elementary bounds
\begin{align}
 \rho_{1/2}(s,t)^2
 &\ge cs\{1+\log(t/s)+\log t\},\label{ac:weight-lower}\\
 \rho_{1/2}(s',t)
 &\le\max\{1,\sqrt{s'/s}\}\rho_{1/2}(s,t).
 \label{ac:weight-mono}
\end{align}
The first bound follows from the definition of $\rho_{1/2}$, with $c>0$
depending only on $C_0$. The second follows because $\rho_{1/2}(s,t)$
is increasing in $s$, while $\rho_{1/2}(s,t)/\sqrt s$ is decreasing.

We will divide the monitoring horizon into geometric shells
$H_j/2\le k\le H_j$, where $H_j=2^{j+1}N$. On each shell,
\eqref{ac:SN} and \eqref{ac:LE} give bounds uniformly over consecutive intervals. Their failure probabilities are summable over $j$, so the bounds hold throughout monitoring with probability tending to one. Let us make this explicit and give names to the corresponding events for later reference. \\

Choose $x_0,L_0$ so that $c_Sx_0^2>4$ and $c_EL_0>4$. Let $\mathcal E_N$ be the event that, for every
$j\ge0$ and every interval $I\subset\{1,\ldots,H_j\}$,
\begin{align}
 |R_I|&\le x_0\sqrt{Q_I\log H_j},\label{ac:shell-S}\\
 Q_I&\ge v|I|\quad\text{if }|I|\ge L_0\log H_j.\label{ac:shell-E}
\end{align}
There are at most $H^2$ such intervals at horizon $H$. Consequently,
\begin{equation}\label{ac:good-event}
 \PR(\mathcal E_N^c)
 \le\sum_{j\ge0}\{C_SH_j^{2-c_Sx_0^2}+C_EH_j^{2-c_EL_0}\}
 \longrightarrow0.
\end{equation}

Similarly \eqref{ac:global-V}-\eqref{ac:global-clock} give us global information on the behaviour of the clock and the rescaling. We will also give the associated event on which the behaviour is nice a name. Fix a deterministic $K$ larger than
$2\max(\sigma_0^2,\sigma_{LR}^{-2},1)$ and let
\begin{equation}\label{ac:normal-event}
 \mathcal K_N=\left\{\sup_{k\ge N}A_k/k\le K,
                         \ \sup_{k\ge N}k/\mathcal V_{N,k}^2\le K\right\}.
\end{equation}
Then $\PR(\mathcal K_N)\to1$ by \eqref{ac:global-V}-\eqref{ac:global-clock}.

The following lemma records bounds on the sample means and sums in terms of the shell size and the clock. This is the central concentration result that we will use repeatedly throughout the proof.
\begin{lem}\label{ac:center-lemma}
For all sufficiently large $N$, on $\mathcal E_N\cap\mathcal K_N$,
for every shell $H=H_j$, $H/2\le k\le H$, and interval $I\subset\{1,\ldots,k\}$,
\begin{align}
 |\bar\varepsilon_k|^2&\le 2Kx_0^2\log H/k,\label{ac:mean-bound}\\
 \left|\sum_{i\in I}e_{k,i}\right|
 &\le C_K\sqrt{\left(\sum_{i\in I}e_{k,i}^2\right)\log H}.
 \label{ac:center-bound}
\end{align}
If $|I|\ge L_0\log H$, then also
\begin{equation}\label{ac:clock-length}
 \sum_{i\in I}e_{k,i}^2\ge\frac v4|I|.
\end{equation}
Furthermore, fix $0<u<1$. For integers $0\le a<b\le k$ with
$h=(A_{k,b}-A_{k,a})/A_k\ge H^{-u}$, we have, for all
sufficiently large $N$,
\begin{equation}\label{ac:physical-length}
 \frac{b-a}{k}\le C_Kh.
\end{equation}
\end{lem}
\begin{proof}
By \eqref{ac:shell-S} and
\eqref{ac:normal-event},
\begin{align*}
 \bar\varepsilon_k^2
 &=\frac{R_k^2}{k^2}
 \le\frac{x_0^2\log H}{k^2}Q_{[1,k]}\\
 &=\frac{x_0^2\log H}{k^2}
   \left(\sum_{i\le k}e_{k,i}^2+k\bar\varepsilon_k^2\right)\\
 &\le\frac{x_0^2\log H}{k}
   \left(\frac{A_k}k+\bar\varepsilon_k^2\right)
 \le\frac{x_0^2\log H}{k}(K+\bar\varepsilon_k^2).
\end{align*}
Uniformly for $H=H_j$ and $H/2\le k\le H$, the coefficient on the right
is at most $1/2$ for large $N$. This proves \eqref{ac:mean-bound} and also
$\sup_{k\ge N}\bar\varepsilon_k^2\to0$ on these events.\\

For $\ell=|I|\ge L_0\log H$ and sufficiently  large $N$ we have that \eqref{ac:shell-E} and $\bar\varepsilon_k^2\le v/4$ give
\begin{equation}\label{ac:center-Q}
 \sqrt{\sum_{i\in I}e_{k,i}^2}\ge\sqrt{Q_I}-\sqrt\ell|\bar\varepsilon_k|
 \ge\tfrac12\sqrt{Q_I}\ge\tfrac12\sqrt{v\ell}.
\end{equation}
Thus
\[
 \frac{|\sum_{i\in I}e_{k,i}|}{\sqrt{\sum_{i\in I}e_{k,i}^2}}
 \le2x_0\sqrt{\log H}+2\sqrt{\ell/v}|\bar\varepsilon_k|
 \le C_K\sqrt{\log H}.
\]
If $\ell<L_0\log H$, Cauchy-Schwarz gives
the second claim, i.e. $|\sum_{i\in I}e_{k,i}|\le\sqrt{L_0\sum_{i\in I}e_{k,i}^2\log H}$, also for small $\ell$.

The third claim follows from \eqref{ac:center-Q}. Finally, for the fourth claim, apply \eqref{ac:clock-length} to $I=\{a+1,\ldots,b\}$ when
$b-a\ge L_0\log H$; otherwise use $b-a<L_0\log H$. In both cases,
\[
 \frac{b-a}{k}
 \le\frac{4A_k}{vk}h+\frac{L_0\log H}{k}
 \le\frac{4K}{v}h+\frac{L_0\log H}{k}.
\]
Since $kh\ge H^{1-u}/2$,
\[
 \frac{L_0\log H}{kh}\le\frac{2L_0\log H}{H^{1-u}}\longrightarrow0
\]
uniformly over the shells as $N\to\infty$. This proves
\eqref{ac:physical-length}.
\end{proof}

We will now show that $\mathcal{R}_{N,k}(a,b)$ stays bounded whenever the clock increment associated to $[a,b]$ is small, i.e. that small scales stay below a threshold.

\begin{lem}\label{ac:fine}
Fix $0<u<\min(1,2\kappa)$. There is a deterministic $F<\infty$ such that,
on $\mathcal E_N\cap\mathcal K_N$, for all sufficiently large $N$,
\begin{equation}\label{ac:fine-bound}
 \sup_{j\ge0}\sup_{H_j/2\le k\le H_j}
 \sup_{\substack{0\le a<b\le k\\A_{k,b}-A_{k,a}\le A_kH_j^{-u}}}
 \mathcal R_{N,k}(a,b)\le F.
\end{equation}
\end{lem}
\begin{proof}
Fix $j\ge0$, $H_j/2\le k\le H_j$, and integers $0\le a<b\le k$
satisfying the constraint in \eqref{ac:fine-bound}. Put
\[
     h=\frac{A_{k,b}-A_{k,a}}{A_k}.
\]
On $\mathcal E_N\cap\mathcal K_N$, \eqref{ac:center-bound} gives
\[
 \mathcal R_{N,k}(a,b)
 \le C_K\sqrt{\frac{A_k}{\mathcal V_{N,k}^2}}\,
       \frac{\sqrt{(k/N)h\log H}}{\rho_{1/2}((k/N)h,k/N)}.
\]
By \eqref{ac:normal-event},
$A_k/\mathcal V_{N,k}^2\le K^2$. Moreover, $h\le H^{-u}$
and \eqref{ac:weight-lower} imply
\[
 \rho_{1/2}((k/N)h,k/N)^2
 \ge c(k/N)h\log(1/h)
 \ge cu(k/N)h\log H.
\]
Consequently,
\[
 \mathcal R_{N,k}(a,b)\le \frac{C_KK}{\sqrt{cu}}=:F.
\]
The bound is uniform over the indices in \eqref{ac:fine-bound}.
\end{proof}

\subsection{Bounds on long horizons, small scales and convergence on compacta}
We now record the short increment and long horizon bounds for the Brownian version of the statistic. 
\begin{lem}\label{ac:brownian}
Let $B$ be standard Brownian motion. Then
\[
 \sup_{t\ge1}\sup_{0\le r<q\le t}
 \frac{|B(q)-B(r)|}{\rho_{1/2}(q-r,t)}
\]
is finite almost surely and has a square-exponential tail.
There is $b_*<\infty$ such that
\begin{align}
 \lim_{T\to\infty}\PR\left(\sup_{t\ge T,\,0\le r<q\le t}\frac{|B(q)-B(r)|}{\rho_{1/2}(q-r,t)}>b_*\right)&=0,
 \label{ac:brownian-late}\\
 \lim_{\eta\downarrow0}\PR\left(
 \sup_{\substack{1\le t\le T,\,0\le r<q\le t\\q-r\le\eta}}
 \frac{|B(q)-B(r)|}{\rho_{1/2}(q-r,t)}>b_*\right)&=0,\qquad T<\infty.
 \label{ac:brownian-short}
\end{align}
\end{lem}

\begin{proof}
We first control these quantities on grids and then generalize by chaining. To that end consider the time intervals
$[0,2^{j+1}]$, $j\ge0$, and for each interval the possible increment lengths of order $2^{j+1-r}$,
$r\ge0$. 

For $j,r\ge0$, set $d_{j,r}=2^{j+1-r}$ and consider the grid
\[
 \{l d_{j,r}:l=0,\ldots,2^r\}
\]
on $[0,2^{j+1}]$. Its $2^r$ adjacent increments are normal with variance
$d_{j,r}$. Thus, for $x>0$, the Gaussian tail bound and a union bound give
\begin{equation}\label{ac:gaussian-grid}
 \PR\left(
 \max_{0\le l<2^r}
 |B((l+1)d_{j,r})-B(ld_{j,r})|
 >x\sqrt{d_{j,r}(j+r+1)}
 \right)
 \le 2^{r+1}e^{-x^2(j+r+1)/2}.
\end{equation}
Let $\mathcal G_{j,r}(x)$ denote the event that the maximum in
\eqref{ac:gaussian-grid} does not exceed its displayed threshold.
Choose $x_1>0$ sufficiently large. Summing the geometric series yields,
for $x\ge x_1$,
\begin{align}
 \sum_{j,r\ge0}2^{r+1}e^{-x^2(j+r+1)/2}
 &\le Ce^{-cx^2},\label{ac:gaussian-sum}\\
 \sum_{j\ge J,r\ge0}2^{r+1}e^{-x_1^2(j+r+1)/2}
 &\longrightarrow0
 \quad(J\to\infty),\label{ac:gaussian-j}\\
 \sum_{0\le j\le J,r\ge R}2^{r+1}e^{-x_1^2(j+r+1)/2}
 &\longrightarrow0
 \quad(R\to\infty,\ J<\infty).
 \label{ac:gaussian-r}
\end{align}
This establishes the desired results on grids by the union bound.

We next extend the adjacent-grid bounds to arbitrary increments via chaining.
Fix $t\ge1$ and $0\le v<w\le t$. Choose $j,r\ge0$ so that
\[
 2^j\le t<2^{j+1},
 \qquad
 d_{j,r+1}<w-v\le d_{j,r}.
\]
The mesh $d_{j,r}$ is therefore within a factor of two of $w-v$.
For $m\ge r$, approximate the endpoints from below by
\[
 v_m=d_{j,m}\lfloor v/d_{j,m}\rfloor,
 \qquad
 w_m=d_{j,m}\lfloor w/d_{j,m}\rfloor.
\]
The initial grid points are at most two mesh widths apart, and each
subsequent approximation moves an endpoint by at most one mesh width
on the finer grid:
\[
 |w_r-v_r|\le2d_{j,r},
 \qquad
 v_{m+1}-v_m,\ w_{m+1}-w_m\in\{0,d_{j,m+1}\}.
\]
Moreover, $v_m\to v$ and $w_m\to w$. On
$\bigcap_{m\ge r}\mathcal G_{j,m}(x)$, continuity and telescoping give
\begin{align*}
 |B(w)-B(v)|
 &\le |B(w_r)-B(v_r)|\\
 &\quad+\sum_{m\ge r}
 \bigl\{|B(w_{m+1})-B(w_m)|
       +|B(v_{m+1})-B(v_m)|\bigr\}\\
 &\le 2x\sqrt{d_{j,r}(j+r+1)}
       +2x\sum_{m\ge r}\sqrt{d_{j,m+1}(j+m+2)}.
\end{align*}
The series is bounded uniformly in $j,r$, since
\begin{align*}
 \sum_{m\ge r}\sqrt{d_{j,m+1}(j+m+2)}
 &=\sqrt{d_{j,r}(j+r+1)}
   \sum_{l\ge0}2^{-(l+1)/2}
   \sqrt{1+\frac{l+1}{j+r+1}}\\
 &\le\sqrt{d_{j,r}(j+r+1)}
   \sum_{l\ge0}2^{-(l+1)/2}\sqrt{l+2}.
\end{align*}
Using $d_{j,r}<2(w-v)$, we obtain $|B(w)-B(v)|\le Cx\sqrt{(w-v)(j+r+1)}$.The choice of $j,r$ also gives $ j+r+1\le C\{1+\log(t/(w-v))+\log$ t\}. Consequently, \eqref{ac:weight-lower} implies
\begin{align*}
 \rho_{1/2}(w-v,t)^2
 &\ge c(w-v)\{1+\log(t/(w-v))+\log t\}\\
 &\ge c'(w-v)(j+r+1).
\end{align*}
Thus there is a deterministic $C_1<\infty$ such that
\[
 \frac{|B(w)-B(v)|}{\rho_{1/2}(w-v,t)}
 \le C_1x
 \quad\text{on }\bigcap_{m\ge r}\mathcal G_{j,m}(x).
\]

For the global assertion, intersect $\mathcal G_{j,r}(x)$ over all
$j,r\ge0$. A union bound and \eqref{ac:gaussian-sum} give
\[
 \PR\left(
 \sup_{t\ge1}\sup_{0\le v<w\le t}
 \frac{|B(w)-B(v)|}{\rho_{1/2}(w-v,t)}>C_1x
 \right)
 \le Ce^{-cx^2},
 \qquad x\ge x_1.
\]
This proves both almost sure finiteness and the square-exponential tail.

For \eqref{ac:brownian-late}, if $t\ge2^J$, its shell index satisfies
$j\ge J$. The same argument therefore gives
\[
 \PR\left(
 \sup_{t\ge2^J}\sup_{0\le v<w\le t}
 \frac{|B(w)-B(v)|}{\rho_{1/2}(w-v,t)}>C_1x_1
 \right)
 \le
 \sum_{j\ge J,r\ge0}2^{r+1}e^{-x_1^2(j+r+1)/2}
 \longrightarrow0.
\]

For \eqref{ac:brownian-short}, fix $T<\infty$ and choose $J$ with
$T<2^{J+1}$. If $1\le t\le T$, then $j\le J$. If additionally
$w-v\le2^{-R}$, the choice
$d_{j,r+1}<w-v$ implies $r\ge R$. Hence
\[
 \PR\left(
 \sup_{\substack{1\le t\le T,\ 0\le v<w\le t\\w-v\le2^{-R}}}
 \frac{|B(w)-B(v)|}{\rho_{1/2}(w-v,t)}>C_1x_1
 \right)
 \le
 \sum_{0\le j\le J,r\ge R}
 2^{r+1}e^{-x_1^2(j+r+1)/2}
 \longrightarrow0.
\]
Monotonicity in the upper bound on $w-v$ extends this conclusion to
arbitrary $\eta\downarrow0$. Taking $b_*=C_1x_1$ proves both limits.
\end{proof}

\begin{lem}\label{ac:discarded}
There is a deterministic $c_*<\infty$ such that, for every $\epsilon>0$,
\begin{align}
 &\lim_{T\to\infty}\limsup_{N\to\infty}
 \PR\left(\sup_{k\ge NT}\sup_{0\le a<b\le k}
 \mathcal R_{N,k}(a,b)>c_*+\epsilon\right)=0,
 \label{ac:clock-late}\\
 &\lim_{\eta\downarrow0}\limsup_{N\to\infty}
 \PR\left(\sup_{N\le k\le NT}
 \sup_{\substack{0\le a<b\le k\\A_{k,b}-A_{k,a}\le\eta A_k}}
 \mathcal R_{N,k}(a,b)>c_*+\epsilon\right)=0
 \quad(T<\infty).\label{ac:clock-short}
\end{align}
The unrestricted supremum over $\mathcal R_{N,k}(a,b)$ is $O_P(1)$.
\end{lem}
\begin{proof}
We work on the high probability set $\mathcal E_N\cap\mathcal K_N$. Let $u$ be as in Lemma~\ref{ac:fine} and choose $j\ge0$ so that $H_j/2\le k\le H_j$. For the sake of readability we set $H=H_j$,
$h=(A_{k,b}-A_{k,a})/A_k$, $t=k/N$, and $x=(b-a)/k$.\\

For $h\le H^{-u}$, use \eqref{ac:fine-bound}. The remainder of the proof is concerned with $h\ge H^{-u}$. Equation \eqref{ac:SA} gives
\begin{align}
 \frac{1}{\mathcal V_{N,k}}\sum_{i=a+1}^b e_{k,i}
 & =\frac{R_b-R_a-xR_k}{\mathcal V_{N,k}}\nonumber\\
 & =\frac{\sigma_{LR}}{\mathcal V_{N,k}}
       \{W(b)-W(a)-xW(k)\}+O(DH^{-\kappa}),
 \label{ac:motion-correction}
\end{align}
where the remainder bound is uniform on $\mathcal E_N\cap\mathcal K_N$. By \eqref{ac:increment-ratio}, the reverse triangle inequality, and
\eqref{ac:weight-lower},
\begin{align}
 &\left|\mathcal R_{N,k}(a,b)
 -\frac{\sigma_{LR}\sqrt t\,|W(b)-W(a)-xW(k)|}
 {\mathcal V_{N,k}\rho_{1/2}(th,t)}\right|\nonumber\\
 &\qquad\le C_KD H^{-\kappa+u/2}
 =o(1).
 \label{ac:remainder-small}
\end{align}

We may thus restrict ourselves to analyzing the Brownian motion version. Let $B_N(z)=W(Nz)/\sqrt N$. By \eqref{ac:physical-length},
$x\le C_Kh$. Equation~\eqref{ac:weight-mono} gives
\[
 \rho_{1/2}(tx,t)\le\max\{1,\sqrt{x/h}\}\rho_{1/2}(th,t)
 \le C_K\rho_{1/2}(th,t),
 \qquad
 \rho_{1/2}(th,t)\ge\sqrt h\,\rho_{1/2}(t,t).
\]
Consequently, on $\mathcal K_N$,
\begin{align}
 &\frac{\sigma_{LR}\sqrt t\,|W(b)-W(a)-xW(k)|}
 {\mathcal V_{N,k}\rho_{1/2}(th,t)}\nonumber\\
 &\quad\le C_K\left\{
 \frac{|B_N(b/N)-B_N(a/N)|}{\rho_{1/2}(tx,t)}
 +\sqrt h\,\frac{|B_N(t)|}{\rho_{1/2}(t,t)}
 \right\}.
 \label{ac:brownian-interval-bound}
\end{align}
Lemma~\ref{ac:brownian} bounds both ratios uniformly by $O_P(1)$.
Together with \eqref{ac:fine-bound} and \eqref{ac:remainder-small},
this proves the unrestricted $O_P(1)$ bound.

For $t\ge T$, \eqref{ac:brownian-late} applies to both ratios in
\eqref{ac:brownian-interval-bound}. Since $\sqrt h\le1$, this proves
\eqref{ac:clock-late} after choosing $c_*$ sufficiently large.

For $1\le t\le T$ and $h\le\eta$, the first ratio has increment length
$tx\le TC_K\eta$, so \eqref{ac:brownian-short} applies. The second term
is bounded by
\[
 \sqrt\eta\sup_{1\le t\le T}
 \frac{|B_N(t)|}{\rho_{1/2}(t,t)}
 =O_P(\sqrt\eta).
\]
Combining these bounds with \eqref{ac:fine-bound} and
\eqref{ac:remainder-small} proves \eqref{ac:clock-short}, with $c_*$
larger than $F$ and the constants obtained from the Brownian bounds.
\end{proof}
We now establish the Brownian approximation on the compact sets.

\begin{lem}\label{ac:compact}
For fixed $a,T$ as above, $T_N^{a,T}\cond T^{a,T}$.
\end{lem}
\begin{proof}
We first note that the clocks and the normalizer converge uniformly - allowing us to replace clock times by determinstic ones and the normalizer by the long-run variance:  Uniformly on $\mathcal D_N^{a,T}$, \eqref{ac:compact-clock} gives
\[
 h^{(1)}_{N,k,\ell}=\max(N,\ell)/k+o_P(1),\qquad
 h^{(2)}_{N,k,\ell}=\ell/k+o_P(1),\qquad
 s_{N,k,\ell}=\ell/N+o_P(1).
\]
Also $\mathcal V_{N,k}/\sqrt k=\sigma_{LR}+o_P(1)$ uniformly.
Let $B_N(t)=W(Nt)/\sqrt N$. By \eqref{ac:SA},
\begin{align*}
 &\sup_{(k,\ell)\in\mathcal D_N^{a,T}}
 \left|\begin{aligned}
 &N^{-1/2}\{\min(1,\ell/N)S_{\max(N,\ell)}-(S_k-S_{k-\ell})\}\\
 &\quad-\sigma_{LR}
 \{\min(1,s)B_N(\max(1,s))-B_N(t)+B_N(t-s)\}
 \end{aligned}\right|\\
 &\hspace{35mm}\le
 3D(1+NT)^{1/2-\kappa}/\sqrt N=o_P(1),
 \qquad t=k/N,\quad s=\ell/N.
\end{align*}
Moreover, uniformly on $\mathcal D_N^{a,T}$,
\[
 \frac{\mathcal V_{N,k}}{\sqrt k}
 \rho_{1/2}(s_{N,k,\ell},t)
 =\sigma_{LR}\rho_{1/2}(s,t)+o_P(1).
\]

A few applications of the triangle inequality thus yield
\[
 \sup_{(k,\ell)\in\mathcal D_N^{a,T}}
 \left|
 w_N^{AC}(\ell,k)\hat\gamma^{TC}(\ell,k)
 -\frac{\gamma^{TC}(B_N;\ell/N,k/N)}
 {\rho_{1/2}(\ell/N,k/N)}
 \right|=o_P(1).
\]
As we are on a compact set the approximation error incurred by switching between the grid and the whole domain is negligible by sample path continuity of Brownian motion. Scaling invariance (i.e. $B_N$ having the same distribution for all $N$) yields the desired result.
\end{proof}

\subsection{Verification of the clock and normalizer constructions}
\label{ac:verification}
The quadratic clock satisfies \eqref{ac:compact-clock} and
\eqref{ac:global-clock} under stationarity and ergodicity. 

\begin{prop}\label{ac:iid-clock}
For centered stationary ergodic errors with
$0<\sigma_0^2=\E\varepsilon_0^2<\infty$, the clock \eqref{ac:clock}
satisfies \eqref{ac:compact-clock} and \eqref{ac:global-clock} almost surely.
\end{prop}
\begin{proof}
Almost surely,
\[
 \sup_{k\ge N}\left|\frac{A_k}k-\sigma_0^2\right|
 \le
 \sup_{k\ge N}\left|
 \frac1k\sum_{i=1}^k\varepsilon_i^2-\sigma_0^2
 \right|
 +\sup_{k\ge N}(R_k/k)^2
 \longrightarrow0
\]
by the ergodic theorem. This proves \eqref{ac:global-clock}.

For fixed $T<\infty$, almost surely,
\begin{align*}
 &\sup_{N\le k\le NT}\max_{j\le k}
 \frac{|A_{k,j}-\sigma_0^2j|}{N}\\
 &\quad\le
 \frac1N\max_{j\le NT}
 \left|\sum_{i=1}^j(\varepsilon_i^2-\sigma_0^2)\right|
 +2\sup_{k\ge N}|R_k/k|\,
       \frac{\max_{j\le NT}|R_j|}{N}
 +T\sup_{k\ge N}(R_k/k)^2
 \longrightarrow0
\end{align*}
where we expanded the centered squares and used the ergodic theorem.
For the maxima over $j\le NT$, we also used that $z_j=o(j)$ implies
$\max_{j\le NT}|z_j|/N\to0$ for each fixed $T$.
Consequently,
\[
 \sup_{N\le k\le NT}\max_{j\le k}
 \left|\frac{A_{k,j}}{A_k}-\frac jk\right|
 \le
 \frac{2\sup_{N\le k\le NT}\max_{j\le k}
       |A_{k,j}-\sigma_0^2j|/N}
      {\inf_{N\le k\le NT}A_k/N}
 \longrightarrow0
 \quad\text{a.s.}.
\]
This proves \eqref{ac:compact-clock}.
\end{proof}

For the dependent construction, we use the following maximal inequality: For a centered stationary causal
sequence $(Y_i)$ with $Y_0\in L^r$, $r>1$, Theorem~1(iii) of
\cite{wu:2007} implies
\begin{equation}\label{ac:maximal}
 \left\|\max_{1\le m\le M}\left|\sum_{i=1}^mY_i\right|\right\|_r
 \le C_rM^{1/\min(r,2)}
       \sum_{l\ge0}\|Y_l-Y_l^{\{0\}}\|_r,
\end{equation}
provided the series on the right is finite. Here $Y_l^{\{0\}}$
is obtained by replacing the innovation at time zero by an independent
copy.
\begin{prop}\label{ac:dependent-clock}
Suppose Assumption~\ref{ac:ass-physical} holds and $b_N\to\infty$,
$b_N/N\to0$. Then the clock \eqref{ac:clock} and normalizer
\eqref{ac:block-clock} satisfy
\eqref{ac:global-V}, \eqref{ac:compact-clock} and
\eqref{ac:global-clock}.
\end{prop}
\begin{proof}
The clock conditions follow from Proposition~\ref{ac:iid-clock}, since
a stationary causal process driven by iid innovations is ergodic. It remains
to prove \eqref{ac:global-V}.
A superscript $0$ denotes the block sums and normalizer in
\eqref{ac:block-clock} with $e_{k,i}$ replaced by $\varepsilon_i$.
In particular,
\[
 D^0_{N,k,r}
 =D_{N,k,r}+\ell_{N,k,r}R_k/k,\qquad
 (\mathcal V^0_{N,k})^2=\sum_r(D^0_{N,k,r})^2.
\]

\textbf{Maximal Bound for Squared Block Sums:}
We first establish a maximal inequality for sums of block sums as they will frequently appear in what follows. Put $p=\min(q/2,2)>1$. For an integer $b\ge1$, define
\[
 Y_{b,r}
 =\frac1b\left(\sum_{i=(r-1)b+1}^{rb}\varepsilon_i\right)^2,
 \qquad r\in\mathbb Z.
\]
Equation~\eqref{ac:maximal} gives
\begin{equation}\label{ac:block-moment}
 \left\|\max_{1\le l\le b}
 \left|\sum_{i=(r-1)b+1}^{(r-1)b+l}\varepsilon_i\right|\right\|_{2p}
 \le C\sqrt b,
\end{equation}
where we used $2p\le q$ and
$\sum_{j\ge0}\delta_{2p}(j)<\infty$ by \eqref{ac:physical}.

Group the innovations into blocks of length $b$, and write $Y_{b,r}^*$
for the variable obtained by replacing the innovations at times
$-b+1,\ldots,0$ by an independent copy. For $r\ge1$,
\begin{equation}\label{ac:block-coupling}
 \|Y_{b,r}-Y_{b,r}^*\|_p
 \le Cb^{-1/2}
 \sum_{i=(r-1)b+1}^{rb}\sum_{j=-b+1}^0\delta_{2p}(i-j),
\end{equation}
where we factored the difference of squares, applied H\"older's
inequality and \eqref{ac:block-moment}, and telescoped the
single-innovation replacements. Consequently,
\[
 \sum_{r\ge1}\|Y_{b,r}-Y_{b,r}^*\|_p
 \le Cb^{-1/2}\sum_{l\ge1}\min(l,b)\delta_{2p}(l)
 \le Cb^{-1/2}\sum_{l\ge1}l\delta_q(l)
 \le C,
\]
since $\nu>2$. The term for $r=0$ is also bounded uniformly in $b$
by \eqref{ac:block-moment}. Applying \eqref{ac:maximal} with the
blocks as innovations therefore gives
\begin{equation}\label{ac:block-max}
 \E\max_{1\le m\le M}
 \left|\sum_{r=1}^m(Y_{b,r}-\E Y_{b,1})\right|^p
 \le CM.
\end{equation}

We also have
\begin{equation}\label{ac:block-bias}
 \E Y_{b,1}
 =\sum_{|j|<b}(1-|j|/b)
   \operatorname{cov}(\varepsilon_0,\varepsilon_j)
 \longrightarrow\sigma_{LR}^2.
\end{equation}
Indeed, the covariance sequence is absolutely summable, since
\[
 |\operatorname{cov}(\varepsilon_0,\varepsilon_j)|
 \le\sum_{l\ge0}\delta_2(l)\delta_2(j+l)
 \le C(1+j)^{-\nu},\qquad j\ge0.
\]

\textbf{Establishing \eqref{ac:global-V}:}
From now on, put $b=b_N$. Equation~\eqref{ac:block-max} gives,
for every $\eta>0$,
\begin{align}
 &\PR\left(
 \sup_{m\ge\lfloor N/b\rfloor}
 \left|\frac1m\sum_{r=1}^m(Y_{b,r}-\E Y_{b,1})\right|>\eta
 \right)\nonumber\\
 &\quad\le
 \sum_{j\ge0}
 \frac{C2^{j+1}\lfloor N/b\rfloor}
      {\eta^p(2^j\lfloor N/b\rfloor)^p}
 \le C_p\eta^{-p}\lfloor N/b\rfloor^{1-p}
 \longrightarrow0,
 \label{ac:block-uniform}
\end{align}
where we divided the indices into
$2^j\lfloor N/b\rfloor\le m<2^{j+1}\lfloor N/b\rfloor$
and applied Markov's inequality.
Similarly,
\begin{equation}\label{ac:partial-block}
 \PR\left(
 \sup_{r\ge\lfloor N/b\rfloor}
 \frac{\max_{1\le l\le b}|R_{rb+l}-R_{rb}|^2}{rb}>\eta
 \right)
 \le C\eta^{-p}\sum_{r\ge\lfloor N/b\rfloor}r^{-p}
 \longrightarrow0
\end{equation}
by \eqref{ac:block-moment} and a union bound.

For $k\ge N$,
\begin{align*}
 \left|\frac{(\mathcal V^0_{N,k})^2}{k}-\sigma_{LR}^2\right|
 &\le
 \frac bk\left|
 \sum_{r=1}^{\lfloor k/b\rfloor}(Y_{b,r}-\E Y_{b,1})
 \right|
 +|\E Y_{b,1}-\sigma_{LR}^2|
 +\frac{b\sigma_{LR}^2}{k}\\
 &\quad+
 \frac1k\max_{1\le l\le b}
 |R_{b\lfloor k/b\rfloor+l}-R_{b\lfloor k/b\rfloor}|^2,
\end{align*}
where we separated the blocks of length $b$ from the possible final
block of length less than $b$.
Taking the supremum over $k\ge N$ and applying
\eqref{ac:block-uniform}, \eqref{ac:block-bias}, and
\eqref{ac:partial-block} yields
\begin{equation}\label{ac:uncentered-consistency}
 \sup_{k\ge N}
 \left|\frac{(\mathcal V^0_{N,k})^2}{k}-\sigma_{LR}^2\right|
 \overset P\longrightarrow0.
\end{equation}

The effect of centering satisfies
\begin{equation}\label{ac:centering-normalizer}
 \frac{|\mathcal V_{N,k}-\mathcal V^0_{N,k}|}{\sqrt k}
 \le
 \frac{|R_k|}{k\sqrt k}
 \left(\sum_r\ell_{N,k,r}^2\right)^{1/2}
 \le\sqrt b\,|R_k/k|,
\end{equation}
by the reverse triangle inequality and
$\sum_r\ell_{N,k,r}^2\le bk$. Moreover,
\[
 \left\|\sup_{k\ge N}|R_k/k|\right\|_q
 \le
 \sum_{j\ge0}(2^jN)^{-1}
 \left\|\max_{k\le2^{j+1}N}|R_k|\right\|_q
 \le CN^{-1/2}
\]
by \eqref{ac:maximal}. Thus the supremum in
\eqref{ac:centering-normalizer} is
$O_P(\sqrt{b/N})=o_P(1)$. Together with
\eqref{ac:uncentered-consistency}, this proves \eqref{ac:global-V}.

\end{proof}

\subsection{Examples for Assumption \ref{ac:ass-local}} \label{sec:example}

We give  a small example that illustrates Assumption \ref{ac:ass-local}.

 \begin{rem}\label{ac:examples}
Assumptions~\ref{ac:ass-physical} and \ref{ac:ass-local} hold for:
\begin{enumerate}
 \item Centered iid errors with
 $0<\E\varepsilon_0^2$ and $\E|\varepsilon_0|^{2+\delta}<\infty$;
 \item Stationary causal AR($p$) processes
 \[
  \varepsilon_i=\sum_{j=1}^p\phi_j\varepsilon_{i-j}+\nu_i,
 \]
 where $1-\sum_{j=1}^p\phi_jz^j$ has no zeros in $|z|\le1$ and
 the iid innovations satisfy
 $\E\nu_0=0$, $0<\E\nu_0^2$, and
 $\E|\nu_0|^{2+\delta}<\infty$.
\end{enumerate}
\end{rem}

The key tool for the proof is the inequality \begin{equation}\label{ac:iid-SN-bound}
 \PR\left(
 \left|\sum_{i=1}^n\xi_i\right|
 >x\sqrt{\sum_{i=1}^n\xi_i^2}
 \right)
 \le Ce^{-cx^2},
 \qquad n\ge1,\quad x\ge0,
\end{equation}
for constants $C,c>0$. It holds for centered iid random variables $(\xi_i)$ with
$0<\E\xi_0^2<\infty$ and is a simple consequence of Theorem~2.5 in \cite{Gine:Gotze:Mason}.

\begin{prop}\label{ac:examples-verification}
The processes in Remark~\ref{ac:examples} satisfy
Assumptions~\ref{ac:ass-physical} and \ref{ac:ass-local}.
\end{prop}
\begin{proof}
By stationarity, it suffices to verify the interval bounds for
$I=\{1,\ldots,n\}$, uniformly in $n\ge1$.

For iid errors, \eqref{ac:SN} follows directly from
\eqref{ac:iid-SN-bound}. Also, for $\lambda>0$,
\[
 \PR(Q_{[1,n]}<vn)
 \le
 \left(e^{\lambda v}\E e^{-\lambda\varepsilon_0^2}\right)^n.
\]
Since $\PR(\varepsilon_0\ne0)>0$, the expectation is strictly less
than one. Choosing $v>0$ sufficiently small proves \eqref{ac:LE}.
The physical-dependence coefficients vanish for $j\ge1$, and
$\sigma_{LR}^2=\E\varepsilon_0^2>0$, proving
Assumption~\ref{ac:ass-physical}.

Now consider the AR($p$) process. We have the causal
representation
\[
 \varepsilon_i=\sum_{j\ge0}\psi_j\nu_{i-j},
 \qquad |\psi_j|\le C\rho^j
\]
for some $0<\rho<1$. Consequently, with $q=2+\delta$,
\[
 \|\varepsilon_0\|_q
 \le\|\nu_0\|_q\sum_{j\ge0}|\psi_j|<\infty,\qquad
 \delta_q(j)
 =|\psi_j|\|\nu_0-\nu_0'\|_q
 \le C\rho^j.
\]
Moreover,
\[
 \sigma_{LR}^2
 =\frac{\E\nu_0^2}{(1-\sum_{j=1}^p\phi_j)^2}>0.
\]
Thus Assumption~\ref{ac:ass-physical} holds.

For the self-normalized bound, first suppose $n>p$. A simple summation yields
\begin{align*}
 \left(1-\sum_{j=1}^p\phi_j\right)R_n
 &=
 \sum_{i=p+1}^n\nu_i
 +R_p-\sum_{j=1}^p\phi_jR_{p-j}
 -\sum_{j=1}^p\phi_j(R_n-R_{n-j}).
\end{align*}
Hence by Cauchy-Schwartz there is a constant $C_\phi<\infty$ such that
\begin{equation}\label{ac:AR-boundary}
 \left|1-\sum_{j=1}^p\phi_j\right||R_n|
 \le
 \left|\sum_{i=p+1}^n\nu_i\right|
 +C_\phi\sqrt{Q_{[1,n]}}.
\end{equation}

Substituting $\nu_i=\varepsilon_i-\sum_{j=1}^p\phi_j\varepsilon_{i-j}$ and using the triangle inequality also yields
\begin{equation}\label{ac:AR-squares}
 \left(\sum_{i=p+1}^n\nu_i^2\right)^{1/2}
 \le
 \left(1+\sum_{j=1}^p|\phi_j|\right)
 \sqrt{Q_{[1,n]}}.
\end{equation}

Subtract $C_\phi\sqrt{Q_{[1,n]}}$ from \eqref{ac:AR-boundary}  and lower bounding the resulting equation with the help of \eqref{ac:AR-squares}. When $|R_n|>x\sqrt{Q_{[1,n]}}$ for $x$ sufficiently large this implies
\[
 \left\{|R_n|>x\sqrt{Q_{[1,n]}}\right\}
 \subseteq
 \left\{
 \left|\sum_{i=p+1}^n\nu_i\right|
 >c_\phi x\sqrt{\sum_{i=p+1}^n\nu_i^2}
 \right\},
\]
where $c_\phi>0$ is independent of $n$ and $x$.
Applying \eqref{ac:iid-SN-bound} to the $n-p$ innovations proves
\eqref{ac:SN} for $n>p$. Bounded $x$ are absorbed into the prefactor.
For $n\le p$, Cauchy--Schwarz gives
$|R_n|\le\sqrt p\,\sqrt{Q_{[1,n]}}$, so enlarging the same prefactor
covers these intervals as well.

For \eqref{ac:LE}, nondegeneracy of $\nu_0$ gives constants $d>0$
and $\vartheta<1$ such that
\[
 \sup_{u\in\mathbb R}\PR(|\nu_0+u|\le d)\le\vartheta.
\]
With $\mathcal F_{i-1}=\sigma(\nu_j:j\le i-1)$, it follows that
\[
 \E\!\left(e^{-\lambda\varepsilon_i^2}\mid\mathcal F_{i-1}\right)
 \le
 \vartheta+(1-\vartheta)e^{-\lambda d^2}<1,
 \qquad \lambda>0,
\]
because $\sum_{j=1}^p\phi_j\varepsilon_{i-j}$ is
$\mathcal F_{i-1}$-measurable and $\nu_i$ is independent of
$\mathcal F_{i-1}$. Successive conditioning and Markov's inequality
therefore give
\[
 \PR(Q_{[1,n]}<vn)
 \le
 \left[
 e^{\lambda v}
 \{\vartheta+(1-\vartheta)e^{-\lambda d^2}\}
 \right]^n
 \le e^{-cn}
\]
for sufficiently small $v>0$. This proves \eqref{ac:LE}, uniformly
in $n$, and completes the verification.
\end{proof}

\section{Nonparametric extension of TWIN}\label{sec:22}
Throughout this section, let $\beta>1/2$. We use the monitoring-time notation of Section~\ref{sec:21}.
We extend the TWIN methodology to the detection of structural breaks in the entire cumulative distribution function (cdf). This problem was previously studied by \cite{kojadinovic:verdier:2020} and (in a somewhat different setting) \cite{horvath:kokoszka:wang:2021}. An advantage of this approach is its robustness to heavy tailed data (in particular covering data that are not subgaussian). In this spirit, we develop here a nonparametric version of TWIN, called NP-TWIN. A mathematical challenge is the extension of Lemma \ref{lem:twin:modulus} to a partial sum process, indexed in distribution functions. 
To formulate our approach, we define for a bounded function $f:\mathbb{R}\to \mathbb{R}$ the supremum norm
\[
\|f\|_\infty:=\sup_{x \in \mathbb{R}}|f(x)|.
\]
Next, we define the (unscaled) cdf of the data $X_1,....X_j$ as
\[
    \hat G_j(x):=\sum_{i=1}^j1\{X_i\leq x\}.
\]
Then, in analogy to $\widehat \Gamma^{TC}$ we define
\[
    \widehat \Gamma^{TC}_F(k):=\max_{\substack{1\leq \ell \leq (N+k)/2,\, \ell\le k}}w^{TC,\beta}(\ell,k)\Big\|\min(1,\ell/N)\hat G_{\max(\ell,N)}(\cdot)-(\hat G_{N+k}(\cdot)-\hat G_{N+k-\ell}(\cdot))\Big\|_\infty.
\]
 The problem is now (generalizing \eqref{e:H0X} vs. \eqref{e:H1X}) to test the pair of hypotheses
\begin{align*}
    H_0^F:F^{(1)}&:=F_1=...=F_N=F_{N+1}=F_{N+2}=..., \\
    H_1^F:F^{(1)}&:=F_1=...=F_{N+k^\star}\neq F_{N+k^\star+1}=F_{N+k^\star+2}=...=:F^{(2)}, 
\end{align*}
where $F_i$ is the cumulative distribution function of $X_i$. 

\textbf{Theoretical analysis} As for mean changes, our approach relies on a lemma describing the modulus of continuity of a centered partial sum process. This one is indexed in two components: time ($t\ge 0$) and location ($x \in \mathbb{R}$)
\begin{align*}
    U_N(t,x):= & N^{-1/2}\bigg(\sum_{i=1}^{\lfloor Nt\rfloor}(1\{X_i\leq x\}-F_i(x))\\ 
    &+(Nt-\lfloor Nt\rfloor)(1\{X_{\lfloor Nt\rfloor+1}\leq x\}-F_{\lfloor Nt\rfloor+1}(x))\bigg).  
\end{align*}
Under the iid null, the limit is $K(t,F_1(x))$, where $K$ is the \textit{Kiefer--Müller process} on $[0,\infty)\times [0,1]$, which is characterized by the covariance structure 
\[
        \text{Cov}(K(t,x),K(s,y))=(t\land s)(x\land y-xy)~.
\]
For some background on the Kiefer--Müller process, we refer to Section 2.12.1 in \cite{vaart:wellner:1996}.
In the following we will assume that the data $X_i$ have a continuous distribution. Then, by the probability integral transformation, the test statistic is distribution free under  $H_0^F$. The following lemma is formulated for iid data (that is, in the absence of a change).

\begin{lem}
\label{lem:finite:moc:empirical}      
    Let $\delta>0$ and $K$ be a Kiefer--Müller process. Suppose that $(X_i)_i$ are iid with continuous distribution function. Define the random variables
    \begin{align*}
    M_{U_N} := \sup_{\substack{0< s<t<\infty, \\ h=|t-s|}} \frac{\|U_N(t,\cdot)-U_N(s,\cdot)\|_\infty}{\sqrt{h \left(1+ \log\frac{t}{h} + \delta |\log t| \right)}},\\
    M_{K} := \sup_{\substack{0< s<t<\infty, \\ h=|t-s|}} \frac{\|K(t,\cdot)-K(s,\cdot)\|_\infty}{\sqrt{h \left(1+ \log\frac{t}{h} + \delta |\log t| \right)}}.
    \end{align*}
    Then $M_{U_N}^{2}$ and $M_K^2$ admit exponential moments that are uniformly bounded in $N$.
\end{lem}
The lemma implies an analogue to Theorem \ref{thm:main:1}. Notice that the limit is pivotal, as the probability integral transform reduces the problem to the uniform case. 
\begin{theo}\label{thm:main:2}
    Suppose that the data $(X_i)_i$ are iid with continuous distribution function. Then, it follows that
    \begin{align*}
       \sup_{k \ge 1}\widehat{\Gamma}^{TC}_F(k) \overset{d}{\to} \sup_{t>1}\sup_{\substack{0<s \leq t/2\\ t-s \geq 1 }}\frac{\Big\|\min(1,s)K(\max(1,s),\cdot)-(K(t,\cdot)-K(t-s,\cdot))\Big\|_\infty}{\sqrt{s}\big\{\log^\beta(C_0+1/s)+\log^\beta(C_0+t)\big\}}=:L_F.
    \end{align*}
    
\end{theo}

As for mean changes we may now define a detection scheme by letting $q_{1-\alpha}^F$ be the $(1-\alpha)$-quantile of $L_F$ and declaring at time $k$ the existence of a change point whenever   $\widehat \Gamma^{TC}_F(k)>q_{1-\alpha}^F$~. Note that this also defines a test of asymptotic level at most $\alpha$ even if the data is not continuous; breaking ties randomly only increases the value of KS type statistics like $\hat \Gamma^{TC}_F$  (see page 40 in \cite{Lehmann:2006}) and yields the same asymptotics as in the continuous case.\\ 
Denoting the time of first detection by $\hat k$ we obtain the analogue of Corollary \ref{cor:power:delay}.
\begin{cor}
\label{cor:power:delay:np}
    Suppose that the vectors $(X_i)_{1 \le i \le N+k_N^\star}$ and $(X_i)_{i > N+k_N^\star}$, each consist of iid observations, with continuous distribution functions $F^{(1)}, F^{(2)}$ respectively, and that the vectors are independent of each other. We suppose that the alternative $H_1^F$ holds, with a distributional change occurring at time $N+k_N^\star$ which is of size 
    \[
    \Delta_N^F:=\|F^{(1)}-F^{(2)}\|_\infty.
    \]
    Then, if 
    \begin{align} \label{e:conscond2}
    \Delta_N^F \gg \frac{\log^\beta(C_0+N/k_N^\star)+\log^\beta(C_0+k_N^\star/N)}{\sqrt{N+k^\star_N}} 
    \end{align}
    it follows that
    \[
    \mathbb{P}(\hat k<\infty)\to 1.
    \]
    If condition \eqref{e:conscond2} holds, then the delay time satisfies
    \[
    \max(0, \hat k-k_N^\star)= O_P \bigg( 1+\frac{\big\{\log^\beta(N)+\log^\beta(C_0+k^\star_N/N)\big\}^2}{(\Delta_N^F)^2}\bigg).
    \]
\end{cor}

\subsection{Proofs for NPTC}
\label{sec:proofs:nptc}
Recall the notational conventions about the time $k$ made at the beginning of the Supplement.  In particular, we write for $k>N$
\begin{align*}
    \widehat \Gamma^{TC}_F(k)=\max_{\substack{1 \leq \ell \leq k/2 \\ k-\ell \geq N}}w^{TC,\beta}(\ell,k)\hat \gamma^{TC}_F(\ell,k),
\end{align*}
where 
\begin{align*}
    \hat \gamma^{TC}_F(\ell,k)=\Big\|\min(1, \ell/N) \cdot \hat G_{\max(\ell,N)}-(\hat G_k-\hat G_{k-\ell})\Big\|_\infty,\\
     w^{TC, \beta}(\ell, k)  :=\frac{\ell^{-1/2}}{\log^\beta(C_0+N/\ell)+\log^\beta(C_0+k/N)}.
\end{align*}
Also, in the following considerations we may and will assume WLOG that $X_i\sim \text{Unif}[0,1]$ by the probability integral transformation.

\textbf{Modulus of continuity results}
\begin{lem}
\label{lem:empirical:sub-Gaussian}
    Let $(X_i)_{i \in \N} \overset{iid}{\sim} \text{Unif}[0,1]$ and consider the partial sum process $U_N$ and the Kiefer--Müller process (both defined precisely in Section \ref{sec:22}). Then for some constant $C>0$ and all $0\le s<t$
    \[
        \Big\| \|U_N(t,\cdot)-U_N(s,\cdot)\|_\infty\Big\|_{\Psi_2}\leq C\sqrt{t-s},
    \]
    as well as
    \[
        \Big\| \|K(t,\cdot)-K(s,\cdot)\|_\infty\Big\|_{\Psi_2}\leq C\sqrt{t-s}.
    \]
    
\end{lem}
\begin{proof}

    For $U_N$ the result follows immediately from the Dvoretzky–Kiefer–Wolfowitz inequality. For $K$, we have
    \[
        K(t,\cdot)-K(s,\cdot)\overset d=\sqrt{t-s}\,B^0(\cdot),
    \]
    where $B^0$ is a standard Brownian bridge. Its supremum norm has a finite $\Psi_2$ norm by Fernique's theorem, giving the second bound.
\end{proof}

With Lemma \ref{lem:empirical:sub-Gaussian} at our disposal we may follow the proofs of Lemmas \ref{lem:finite:moc:proofs} (use p.573 from \cite{Friz:Victoir:2010} instead of Lemma 1 from \cite{chevallier2023}) and \ref{lem:cutoff} verbatim to obtain the two next results.
\begin{lem}
\label{lem:finite:moc:empirical:proofs}      
    Let $\delta>0$ and consider the processes $K, U_N$ as in Lemma \ref{lem:empirical:sub-Gaussian}. Define the random variables
    \begin{align*}
    M_{U_N} := \sup_{\substack{0< s<t<\infty, \\ h=|t-s|}} \frac{\|U_N(t,\cdot)-U_N(s,\cdot)\|_\infty}{\sqrt{h \left(1+ \log\frac{t}{h} + \delta |\log t| \right)}},\\
    M_{K} := \sup_{\substack{0< s<t<\infty, \\ h=|t-s|}} \frac{\|K(t,\cdot)-K(s,\cdot)\|_\infty}{\sqrt{h \left(1+ \log\frac{t}{h} + \delta |\log t| \right)}}.
    \end{align*}
    Then, $M_{U_N}^{2}$ and $M_K^2$ admit exponential moments that are uniformly bounded in $N$.
\end{lem}
 
\begin{lem}
\label{lem:cutoff:distributon}
    We have, for any $\eta>0$, that
    \begin{align*}
        &\sup_{t \geq N^\eta}\sup_{\substack{0<s \leq t/2\\ t-s \geq 1 }}\frac{\Big\|\min(1,s)K(\max(1,s),\cdot)-(K(t,\cdot)-K(t-s,\cdot))\Big\|_\infty}{\rho_\beta(s,t)}=o_{\PR}(1),        
    \end{align*}
    as well as
    \begin{align*}
        & \sup_{t \geq N^\eta}\sup_{\substack{0<s \leq t/2\\ t-s \geq 1 }}\frac{\Big\|\min(1,s)U_N(\max(1,s),\cdot)-(U_N(t,\cdot)-U_N(t-s,\cdot))\Big\|_\infty}{\rho_\beta(s,t)}=o_{\PR}(1).
    \end{align*}            

\end{lem}

\textbf{Derivation of distributional limit}
\begin{theo}    
    Under the assumptions of Theorem~\ref{thm:main:2}, we have that
    \begin{align}        
        &\sup_{k\geq N}\sup_{ \substack{1 \leq \ell \leq k/2 \\ k-\ell \geq N}}\frac{\Big\|\min(1,\frac{\ell}{N})\hat G_{\max(N,\ell)}(\cdot)-(\hat G_k(\cdot)-\hat G_{k-\ell}(\cdot))\Big\|_\infty}{\sqrt{N}\rho_\beta(\ell/N,k/N)}\\        
        \overset{d}{\to}&\sup_{t \geq 1}\sup_{ \substack{ 0<s\leq t/2 \\ t-s\geq 1}}\frac{\Big\|\min(1,s)K(\max(1,s),\cdot)-(K(t,\cdot)- K(t-s,\cdot))\Big\|_\infty}{\rho_\beta(s,t)}  ,  
    \end{align}
    where $K$ is a Kiefer--Müller Process.
\end{theo}
\begin{proof}
   Note that by Theorem 4  from \cite{kmt:1975} we may find countably many independent Brownian bridges $(B_i)_{i \in \N}$ such that
    \[
        \Big\|\hat G_k(\cdot)-kF_1(\cdot)-\sum_{i=1}^kB_i(\cdot)\Big\|_\infty\lesssim \log(k)^2.
    \]
    Using Lemma \ref{lem:construct:brownian} we may therefore find a Kiefer--Müller process $K$ such that
     \[
        \Big\|\hat G_k(\cdot)-kF_1(\cdot)-K(k,\cdot)\Big\|_\infty\lesssim \log(k)^2.
    \]
    The centering terms cancel in the TWIN contrast. From here we may proceed, verbatim, in the same manner as in the proof of Theorem \ref{theo:limit} to arrive at the desired conclusion.    
\end{proof}

\section{Further details}

\subsection{Comparison of contrasts under a common weight}
\label{sec:common-weight}

The choice of contrast and weight is closely linked. To illustrate this,
consider iid observations with variance $\sigma^2$, let $n=N+k$, and
denote the recent window length by $h$. Here, $k$ counts observations
after the training period. Writing $A_{\mathrm{TC}}$ and
$A_{\mathrm{FC}}$ for the expressions inside the absolute values in
$\hat\gamma^{TC}(h,k)$ and $\hat\gamma^{FC}(k-h,k)$, respectively, gives
\begin{align*}
\operatorname{Var}\left(\frac{A_{\mathrm{TC}}(h,n)}{\sqrt h}\right)
&=\sigma^2\left(1+\frac{h}{\max\{N,h\}}\right)
\le 2\sigma^2,\\
\operatorname{Var}\left(\frac{A_{\mathrm{FC}}(h,n)}{\sqrt h}\right)
&=\sigma^2\left(1+\frac{h}{n-h}\right).
\end{align*}
For $h=k$, the latter equals $\sigma^2(1+k/N)$. Thus, Full CUSUM
can amplify estimation error in a short reference block when compared
with a much longer recent window. TWIN is designed to avoid this issue.

To examine the contribution of the TWIN contrast under a common weight,
we equip TWIN, CUSUM, Page CUSUM, Full CUSUM, and modified MOSUM
with $w^{TC,1/2}(h,k)$, setting $C_0=20$. Each method retains its
usual admissible windows. The weight is evaluated at the recent
window length: $h=k$ for CUSUM, $h=k-\ell$ for Page and Full CUSUM,
and $h=k-\lfloor bk\rfloor$ for modified MOSUM, with $b=0.4$.
FC and WC coincide after replacing their original weights and are
therefore represented by one row.

We consider iid standard Gaussian errors, $N=100$, $T=20N$,
$\Delta=0.5$, and $k^\star\in\{4N,10N\}$. All methods use the
empirical variance of the initial $N$ observations and full $1/N$
grids. For each method, the critical value is calibrated separately
to a nominal level of $5\%$ over the finite horizon using $50{,}000$
Gaussian null runs. Size is evaluated on $10{,}000$ independent null
runs, and power on $10{,}000$ paired runs per change location. 

\begin{table}[ht]
\centering
\small
\setlength{\tabcolsep}{4pt}
\begin{tabular}{@{}lrrrrr@{}}
\toprule
 & & \multicolumn{2}{c}{$k^\star=4N$}
   & \multicolumn{2}{c}{$k^\star=10N$} \\
\cmidrule(lr){3-4}\cmidrule(lr){5-6}
Contrast & Size (\%) & Power (\%) & Delay & Power (\%) & Delay \\
\midrule
TWIN           & 5.32 & 100.00 &  98 & 100.00 & 102 \\
CUSUM          & 5.05 &  97.20 & 545 &  66.88 & 586 \\
Page CUSUM     & 5.00 &  99.55 & 309 &  98.02 & 317 \\
Full CUSUM     & 4.84 & 100.00 & 350 & 100.00 & 376 \\
Modified MOSUM & 4.90 &  99.76 & 294 &  97.74 & 425 \\
\bottomrule
\end{tabular}
\caption{Comparison of contrasts equipped with the TWIN weight.
Power is the probability of detection by $T$, conditional on no
pre-change alarm. Delay is the median of $\widehat k-k^\star$
over post-change detections only.}
\label{tab:contrasts-twin-weight}
\end{table}

Table~\ref{tab:contrasts-twin-weight} shows that all methods achieve
null rejection rates close to the nominal level. TWIN and Full CUSUM
both attain empirical power of $100\%$, but their median delays
differ substantially: $98$ and $102$ for TWIN, compared with $350$
and $376$ for Full CUSUM. The remaining contrasts also have longer
delays, and classical CUSUM loses considerable power for the later
change. Hence, the performance difference cannot be attributed to
the weighting alone. Under the common weight, the contrast and its
associated scan make a substantial contribution to TWIN's shorter
detection delays.

\subsection{Application to epidemiological time series}
\label{sec:application}
Real-time data analysis is crucial in public health contexts, where both high statistical certainty and fast detection are necessary to allow for well-informed and swift policy responses. The procedures developed in this paper satisfy both criteria: they have a hypothesis testing interpretation and thus come with guaranteed bounds on the false positive rate, just like any retrospective test. Moreover, our methods are designed to have very short delays between the emergence of a signal and its detection. As we will see below, it is this property which makes them most advantageous in practical contexts, where detecting a change too late is just as problematic as missing it entirely. 

\textbf{Data description} As an illustration, we apply our new methodology to a dataset of COVID19 data, gathered over a period of 5 years, from 2020 to 2024 in Mexico. Of course, one cannot genuinely perform a real-time analysis on a historical dataset. However, we can make a counterfactual study that shows what results our monitoring procedure would have yielded if it had been available at the time. The dataset under investigation was obtained from \cite{DeArcosJimenez2025}, and we refer to that work for a precise description. The data consist of individual Ct (cycle threshold) values of patients infected with COVID19. As explained in that work, Ct values are quantitative proxies for viral load in patients, with lower Ct values indicating higher viral load. Ct values are naturally bounded, since they assume positive values smaller than some maximum number of cycles used by the laboratory. Consequently, it is appropriate to think of them as compactly supported and in particular subgaussian. In the past, Ct values have been mostly used to measure the infectiousness of individual patients \citep{shah:etal:2021}, but recent studies indicate that they contain valuable information at a population level \citep{dehesh:etal:2022}. For instance, drops in Ct values often predict subsequent rises in case counts, and changes in Ct distributions have also been associated with the emergence of new virus variants \citep{park:etal:2024}. For these reasons, \cite{DeArcosJimenez2025} conclude that “Integrating Ct monitoring into surveillance systems could enhance pandemic preparedness, improve outbreak forecasting, and strengthen epidemiological modeling.” We also note that the use of Ct values is not confined to COVID19 \citep{bentahar:2025} and could play a role in monitoring other epidemics.

\textbf{Details on the analysis} We use sequential change point tests to study Ct values from 1st May 2020 on. Earlier measurements are too scarce for a stable analysis. Measurements of all patients on a given day are summarized by taking the median. For sequential monitoring as considered in this work, we need a baseline period, and we take the month of May 2020 ($N=31$). The variance of the data is estimated by taking the empirical variance over this reference period. We note that a goodness-of-fit test failed to reject the hypothesis of white noise in the training sample. To evaluate the performance of our procedure and compare with existing benchmarks, we run TC, NPTC, C, PC, FC, WC, MM, and RC (both with empirical orlicz and bootstrap calibration) on the data.\\
A subset of procedures also entail a natural estimate of the change point $k^\star$, namely $NPTC$, $TC$, $PC$, $FC$, $WC$, and $RC$, where we define
\[
  \hat\ell := \underset{\ell\text{ admissible at }\hat k}{\arg\max} \ \omega(\ell,\hat k)\gamma(\ell,\hat k)~.
\]
For $NPTC$ and $TC$, $\hat\ell$ is the recent window length, so $\hat k^\star=\hat k-\hat\ell$. For the other procedures, $\hat k^\star$ is the maximizing split location in monitoring time.
We plot the data, times of detection and, if applicable, estimated change point of each procedure in Figure \ref{fig:covid}.

\begin{figure}
  \centering
  \includegraphics[width=0.8\textwidth]{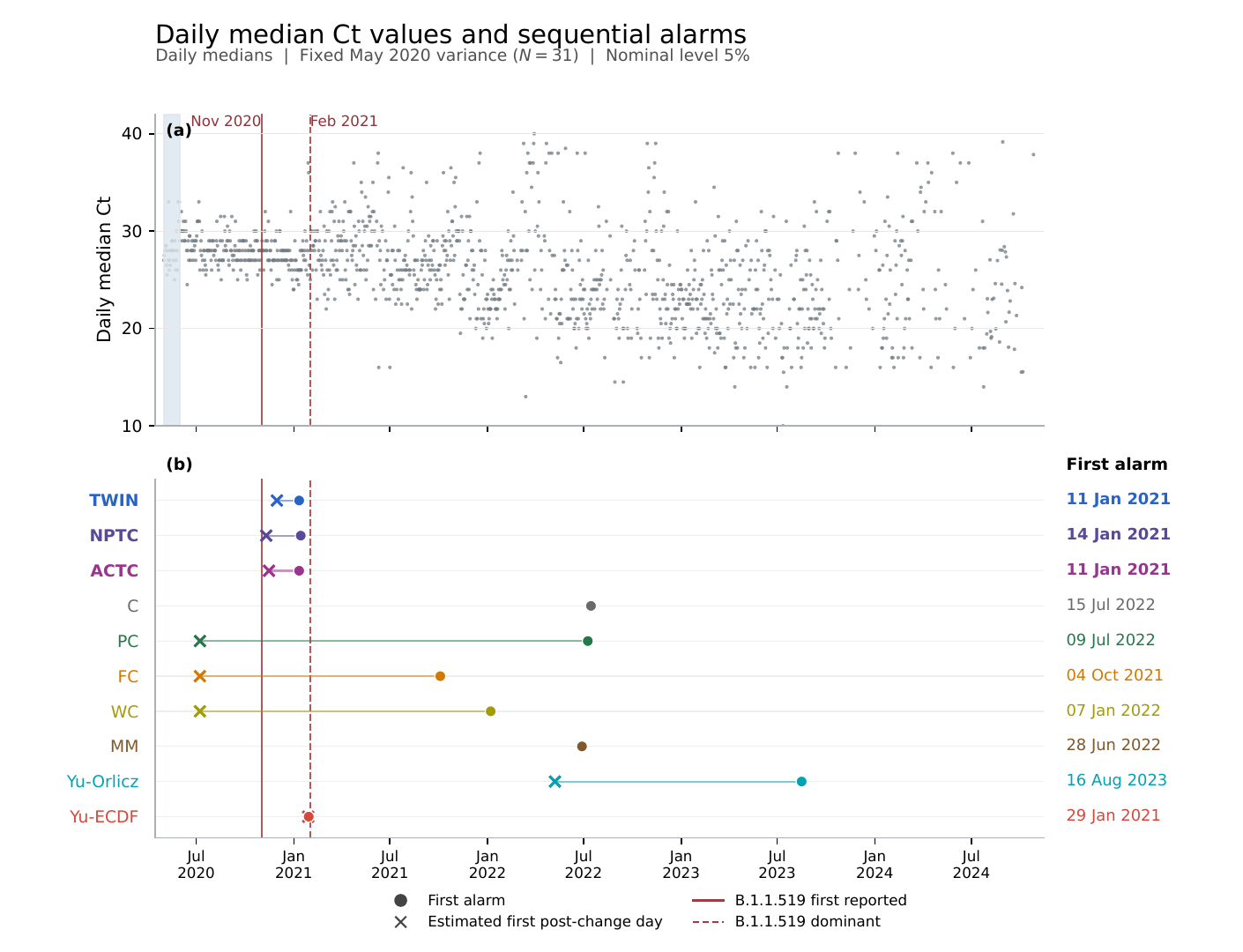}  
  \caption{Detection times (circles) and estimated changes (crosses) for daily median Ct values. Data from \cite{DeArcosJimenez2025}. The red vertical lines indicate when the strain B.1.1.519 first emerges (solid) and then becomes dominant (dashed).}
  \label{fig:covid}
  \vspace{-0.25cm}
\end{figure}

\textbf{Results and interpretation} To interpret our statistical findings, some additional background on the COVID pandemic in Mexico is needed. \cite{Taboada2021} describe the genetic evolution of COVID variants in Mexico between February 2020 and March 2021. For our purposes, a key observation is that there was no dominant strain of the virus until November 2020, when the B.1.1.519 lineage started appearing. Over the next three months, this variant became dominant, representing roughly 70\% of all cases at the beginning of February 2021. The first emergence and the first time of dominance for B.1.1.519 are displayed in Figure \ref{fig:covid} as vertical lines. As we can see, all monitoring procedures detect some change in Ct values sometime between the emergence of the variant and August 2023. However, the delay times differ drastically. Our new test statistics TC, ACTC and NPTC detect the change in infection patterns very soon after it emerges, and before B.1.1.519 becomes dominant in February 2021. Changes are dated back to November 2020, close to the first known cases of the new variant. Among the remaining procedures, Yu-ECDF is the first to pick up the change, also in January 2021, shortly after our new procedures. However, Yu-ECDF does not reliably hold the significance level in our simulations and also does not date the change back to November 2020. Yu-Orlicz, on the other hand, detects a change only in August 2023. All remaining methods detect the change only about a year or more after it has occurred. The most extreme delays are reached by the $PC$ and $C$ methods, which have delays of about twenty months. These delays are so large that by the time $PC$ and $C$ detect a change, the majority of the Mexican population had already been vaccinated for a long time. Most of these methods also entail a change-point estimator, but the values cluster at the very beginning of the monitoring period. This behavior seems to lack epidemiological plausibility.

\subsection{Additional literature review} \label{sec:app:lit}

In Section \ref{sec:lit}, we have composed a literature overview for sequential change point testing, focusing on work that is most comparable to ours and can be applied (possibly after minor modifications) for testing changes in the mean parameter of a scalar time series. Here, we want to give a short outlook on recent developments  in sequential change point detection, for more general classes of problems. The field is highly dynamic and we can only point to some overall trends. \\
Besides the mean, changes in various other parameters have been investigated. A flexible class of detectors is given by sequential U-statistics, as investigated by \cite{kirch:stoehr:2021}. Here, not only consistency has been studied, but also the asymptotic distribution of delay times \cite{kirch:stoehr:2022}. Importantly, U-statistics include Wilcoxon type statistics, that can be used for robust sequential inference. Change point detection used for monitoring entire distribution functions is  also more robust against outliers and examples include \cite{kojadinovic:verdier:2020} and \cite{horvath:kokoszka:wang:2021}. 
Monitoring methods play an increasing role in economic and finance application, where they can be used, e.g., to detect volatility changes  or the emergence of bubbles. For some work in these directions, we point to \cite{horvath:liu:rice:wang:2020,horvath:liu:lu:2022,horvath:lazar:liu:wang:xue:2025}. Another method in this vein was developed by \cite{otto:breitung:2023}. The approach is related to  the Full CUSUM statistic discussed in Section \ref{sec:lit}, but focused on detecting changes in the slope parameter of a linear regression model. That paper also employs monitoring to detect changes in COVID data, with the important difference to our work that it uses case counts rather PCR values. 
A recent trend in sequential change point testing is the extension to complex data structures. Examples include functional data, with functional linear regression models studied in \cite{aue:hormann:horvath:huskova:2014} and changes in the mean function by \cite{kutta:kokoszka:2025}.
Monitoring for high dimensional data has also been studied, namely by \cite{gosmann:stoehr:heiny:dette:2022} for the high-dimensional mean and by \cite{doernemann:kokoszka:kutta:lee:2025} for high dimensional matrices. Besides these, there exist various works on high dimensional change point monitoring, though typically outside the testing paradigm of \cite{chu:stinchcombe:white:1996}. We refer to the work  \cite{chen:wang:samworth:2024} and \cite{yu:madridpadilla:wang:rinaldo:2024} for examples of high-dimensional monitoring without a training period.

We conclude with a connection to retrospective multiple change point detection. To detect a change with a delay of $h$, a monitoring procedure has only $h$ post-change observations to work with. In the retrospective problem, if the mean changes and returns to its original level $h$ observations later, only those $h$ observations carry the signal. Both problems therefore call for sensitivity on short windows. There is also a quantitative parallel. For a jump of fixed size, TWIN with $\beta=1/2$ achieves a delay of $\mathcal{O}_P(\log N)$ when $k^\star\asymp N$. In a retrospective sample of $N$ independent Gaussian observations, the minimum length needed to reliably detect a segment of fixed height at an unknown location is likewise of order $\log N$ \citep{verzelen:fromont:lerasle:reynaud-bouret:2023}.

Building on the multiscale tests of \cite{duembgen:spokoiny:2001}, \cite{frick:munk:sieling:2014} introduced SMUCE for inference on the number and locations of change points. For independent Gaussian noise, SMUCE can use all window lengths for scanning. Its general exponential family theory though imposes a minimum window length for the Gaussian approximation, as does the extension to dependent data by \cite{dette:eckle:vetter:2020}. In practice, this smallest window length needs to be chosen (balancing Gaussian approximation versus small scale sensitivity). A more recent method is Narrowest Significance Pursuit by \citep{fryzlewicz:2024}. It provides a finite-sample guarantee that, with probability at least $1-\alpha$, every reported interval contains a change point. For Gaussian mean changes of fixed size, its theory allows logarithmic spacing between jumps (so it is rate optimal). More closely related to our use of H\"older weights, \cite{kutta:dette:wang:2025} develop retrospective scans for multiple changes in dependent Banach-valued time series, including functional data. 
The parallel work by \cite{koehne:mies:2025} is focused on scalar time series, and establishes above-threshold convergence for critically weighted retrospective scans.

Despite their close relation, monitoring and multiple change point detection pose different challenges. For dependent data, practical calibration can be harder in the retrospective setting, since one typically has no uncontaminated training sample from which to estimate the long-run variance. On the theoretical side, monitoring requires additional work because the probability of a false alarm must be controlled over an unbounded horizon. The scan statistics used for the two problems also differ. Retrospective multiple change point detectors often compare adjacent windows, whereas monitoring procedures such as TWIN compare the training sample (or an extension of it) with the most recent observations.

\subsection{Empirical calibration of Yu's method}\label{sec:yu-implementation}

The method of \cite{yu:padilla:wang:rinaldo:2023} uses a known upper bound on the subgaussian scale. We briefly describe the empirical versions used in Section~\ref{sec:simulation}. With $S_j=\sum_{i=1}^jX_i$ as in Section~\ref{sec:lit}, write
\[
 D_N(k):=\max_{1\leq\ell<N+k}
 \left|S_\ell-\frac{\ell}{N+k}S_{N+k}\right|
 \sqrt{\frac{N+k}{8\ell(N+k-\ell)\log((N+k)/\alpha)}},
 \qquad k\geq1.
\]

\textit{Empirical Orlicz norm (Yu-Orlicz).}
Using the convention in \eqref{e:SG}, we estimate the subgaussian scale by the new  estimator given in \cite{twin-sim:mies:2026}:
\[
 \widehat C_{\varepsilon,n}
 :=\inf\left\{c>0:\frac1n\sum_{i=1}^n
 \exp\left(\frac{(X_i-S_n/n)^2}{c^2}\right)\leq2\right\}.
\]
The estimate is updated using all observations available at each arrival, and Yu-Orlicz rejects when $D_N(k)>\widehat C_{\varepsilon,N+k}$. 

\textit{Bootstrap calibration (Yu-ECDF and Yu-CBB).}
For independent data, we resample from the empirical distribution of $X_1,\ldots,X_N$ (Yu-ECDF). For dependent data, we use the circular block bootstrap of \cite{twin-sim:politis:romano:1992}, separately for block lengths $5$ and $8$ (Yu-CBB). Starting indices are drawn independently and uniformly from $\{1,\ldots,N\}$, and consecutive observations are taken with indices interpreted modulo $N$. In both cases, each bootstrap trajectory has length $N+T$ and includes a newly resampled training segment; circular blocks may cross the training/monitoring boundary. For each of $B=499$ bootstrap trajectories, we compute $\max_{1\leq k\leq T}D_N^*(k)$, where $D_N^*$ is defined as above using the resampled observations. At $\alpha=0.05$, the cutoff is the $475$th ordered bootstrap maximum. 

\end{appendices}
\section{Competing interests}
No competing interest is declared.

\section{Author contributions statement}

P.B. and T.K. developed the statistical methodology and wrote the manuscript. T.K. procured the data set and P.B. implemented the methodology in R.

\putbib
\end{bibunit}

\end{document}